\documentclass[a4paper,9pt,twoside]{extarticle}

\usepackage{graphicx}
\usepackage{caption}
\usepackage{subcaption}
\usepackage{multirow}
\usepackage[utf8]{inputenc}
\usepackage[T1]{fontenc}
\usepackage{ dsfont }
\usepackage[english]{babel}

\usepackage{charter}

\usepackage{indentfirst}
\usepackage{xcolor}
\usepackage{verbatim}

\usepackage{hyperref}
		
\usepackage[top=3.cm, bottom=4.0cm, left=2.2cm, right=2.2cm]{geometry}
\usepackage{amsmath}
\usepackage{cases}
\usepackage{amsthm}	
\usepackage{amsfonts}	
\usepackage{amssymb	
            ,bbm
            ,units 
            ,stmaryrd
           }
\usepackage{enumerate}

\usepackage[square,numbers,sort&compress]{natbib}

\numberwithin{equation}{section}
\numberwithin{figure}{section}
 \usepackage[nodayofweek]{datetime}

\renewcommand*{\thefootnote}{\fnsymbol{footnote}}

\title{Wellposedness, exponential ergodicity and numerical approximation of fully super-linear McKean--Vlasov SDEs and associated particle systems}

\author{
\normalsize Xingyuan Chen\textit{$^{a}$} \\
        \small   X.Chen-176@sms.ed.ac.uk 
\and
 \normalsize Gon\c calo dos Reis\textit{$^{a,b,}$}\footnote{G.d.R. acknowledges support from the \emph{Funda{\c c}\~ao para a Ci\^{e}ncia e a Tecnologia} (Portuguese Foundation for Science and Technology) through the project UIDB/00297/2020 and UIDP/00297/2020 (Centro de Matem\'atica e Aplica\c c\~oes CMA/FCT/UNL), and by the UK Research and Innovation (UKRI) under the UK government’s Horizon Europe funding Guarantee [Project APP55638].
  } \\
        \small  G.dosReis@ed.ac.uk
\and
\normalsize Wolfgang Stockinger \textit{$^{c}$} \\
        \small   w.stockinger@imperial.ac.uk
}

\date{%
    \footnotesize 
    $^{a}$~School of Mathematics, University of Edinburgh, JCMB, 
    Peter Guthrie Tait Road, Edinburgh, EH9 3FD, UK
    \\
    $^{b}$~ Centro de Matem\'atica e Aplica\c c\~{o}es (Nova Math), FCT, UNL, 2829-516 Caparica, Portugal
    \\
    $^{c}$~Department of Mathematics, Imperial College London, London SW7 2AZ, UK
    \\
    \longdate \today \ (\currenttime)
    \vspace{-1.0cm}
}

\theoremstyle{plain}
\newtheorem{theorem}{Theorem}[section]
\newtheorem{lemma}[theorem]{Lemma}
\newtheorem{proposition}[theorem]{Proposition}

\newtheorem{definition}[theorem]{Definition}

\newtheorem{remark}[theorem]{Remark}
\newtheorem{example}[theorem]{Example}
\newtheorem{assumption}[theorem]{Assumption}

\newcommand{\bE}{\mathbb{E}}
\newcommand{\bF}{\mathbb{F}}

\newcommand{\bN}{\mathbb{N}}
\newcommand{\bP}{\mathbb{P}}

\newcommand{\bR}{\mathbb{R}}

\newcommand{\cF}{\mathcal{F}}

\newcommand{\cN}{\mathcal{N}}

\newcommand{\cP}{\mathcal{P}}

\newcommand{\1}{\mathbbm{1}}

\newcommand{\tnp}{{t_{n+1}}}

\newcommand{\dd}{\mathrm{d}}

\newcommand{\hx}{ \hat{X} }
\newcommand{\hs}{ \overline{\sigma} }
\newcommand{\hz}{ \hat{Z} }

\newcommand{\hm}{ \hat{\mu} }

\newcommand{\fs}{ f_{\sigma} }
\newcommand{\bdg}{  {\boldsymbol{g}} }

\begin{document}

\selectlanguage{english}

\maketitle
\renewcommand*{\thefootnote}{\arabic{footnote}}

\begin{abstract} 
We study a class of McKean--Vlasov Stochastic Differential Equations (MV-SDEs) with drifts and diffusions having super-linear growth in measure and space -- the maps have general polynomial form but also satisfy a certain monotonicity condition. The combination of the drift's super-linear growth in measure (by way of a convolution) and the super-linear growth in space and measure of the diffusion coefficient requires novel technical elements in order to obtain the main results. We establish wellposedness, propagation of chaos (PoC), and under further assumptions on the model parameters, we show an exponential ergodicity property alongside the existence of an invariant distribution. No differentiability or non-degeneracy conditions are required. 

Further, we present a particle system based Euler-type split-step scheme (SSM) for the simulation of this type of MV-SDEs. 
The scheme attains, in stepsize, the strong error rate $1/2$ in the non-path-space root-mean-square error metric and we demonstrate the property of mean-square contraction. 
Our results are illustrated by numerical examples including: estimation of PoC rates across dimensions, preservation of periodic phase-space, and the observation that taming appears to be not a suitable method unless strong dissipativity is present. 
\end{abstract}
{\bf Keywords:} 
McKean--Vlasov equations, split-step methods, ergodicity, interacting particle systems, super-linear growth in measure


\footnotesize
\tableofcontents
\normalsize


\section{Introduction}
In this work, we analyse a class of  McKean--Vlasov Stochastic Differential Equations (MV-SDEs) having drift and diffusion components of convolution type, akin to the porous media equation or interaction kernel modelling. \color{black}The main feature of the class is the joint \textit{super-linear growth in measure and space} in both drift and diffusion coefficients, concretely, super-linear in the sense that in \eqref{Eq:General MVSDE shape of v}, the maps $x\mapsto u(x,\cdot)$, $x \mapsto f(x)$ and $x\mapsto f_\sigma(x)$ are of polynomial growth (under some radial growth conditions). 
\color{black}

We work with MV-SDE dynamics of the form
\begin{align}
\label{Eq:General MVSDE}
\dd X_{t} &= \big( v(X_{t},\mu_{t}^{X}) + b(t,X_{t}, \mu_{t}^{X})\big)\dd t + \hs(t,X_{t}, \mu_{t}^{X})\dd W_{t},  \quad X_{0} \in L^m( \bR^{d}),~m >2
\\
\label{Eq:General MVSDE shape of v}
&
\textrm{where }~ 
\begin{cases}
    v(x,\mu)
&= \int_{\bR^{d}  } f(x-y) \mu(\dd y) +u(x,\mu),
\\ 
\hs(t,x,\mu)
&=\sigma(t,x,\mu)+\int_{\bR^{d}  } \fs(x-y) \mu(\dd y)
\end{cases},
\end{align}
where $\mu_{t}^{X}$ denotes the law of the solution process $X$ at time $t$, $L^{m}(\mathbb{R}^{d})$ is the space of {\color{black}$\mathcal{F}_0$-measurable} random variables with finite $m$-th moments, $W$ is a multidimensional Brownian motion, $u,b, \sigma$ and $f, f_{\sigma}$ are measurable maps.
Critically, $f,f_{\sigma},u,\sigma$ are  maps of super-linear growth but not assumed to be differentiable and $\hs$ may degenerate. The functions $u$ and $\sigma$ allows to incorporate measure dependencies other than convolution type. 

In terms of a particle dynamics modelling perspective, \eqref{Eq:General MVSDE}-\eqref{Eq:General MVSDE shape of v} model the dynamics of particle motion where the particle is affected by different sources of forcing. The map $u$ represents a multi-well gradient potential confining the particle and the convolution map $f$ contains information on the forces affecting the particles (e.g., attractive, repulsive), see \cite{HerrmannImkellerPeithmann2008,2013doublewell,adams2020large}. As argued in \cite{HerrmannImkellerPeithmann2008}, under certain assumptions, $v$ and $f$ add inertia to the particle's dynamic in turn affecting its exit time from a domain of attraction (by accelerating or delaying it) and alters exit locations \cite{HerrmannImkellerPeithmann2008,dosReisSalkeldTugaut2017,adams2020large,de2021reducing}. To motivate the study of equations with a (nonlinear) convolution term 
$(f_{\sigma} \ast \mu) (\cdot) := \int_{\bR^{d}  } \fs(\cdot-y) \mu(\dd y)$ in the diffusion component, which is the main feature of our work, we first mention \cite{MR3582237}. There, a Cucker--Smale model incorporating random communication is rewritten as a Cucker--Smale model with multiplicative noise (the diffusion coefficient has the form 
\color{black}
$\left(X_t-\mathbb{E}[X_t] \right)=\int_{\bR^d}(X_t-y)\mu_t(\dd y)$\color{black}), which helps to stabilize flocking states as the effect of the noise diminishes the closer the particles concentrate around their mean; see also \cite{MR2761313}. 
These works give a clear motivation to analyse convolution type diffusion maps (diffusions whose strength depends on the density) -- also \cite{Duan2010kineticflockingwithdiffusion}  studies a kinetic flocking model with a more general distance potential (communication rate) function than \cite{MR3582237}.  
In addition, \cite{MR2860672} considers general stochastic systems of interacting particles with Brownian noise to study  models for the collective behaviour (swarming) -- more particularly, \cite[Section 1.2.2]{MR2860672} highlights several open-question model extensions to nonlinear diffusion coefficients (though beyond the scope of this work). 
The recent works \cite{carrillo2018analytical,borghi2023constrained} investigate Consensus-Based Optimization (CBO) methods for solving  high-dimensional nonlinear unconstrained minimization problems. 
A CBO scheme updates the  particle's  position in an iterative manner to explore the optimization landscape. There, particles far away from the equilibrium state are expected to exhibit more exploration (i.e., the noise level should be larger) compared to particles close to it. Inspired by the above discussed works, we offer a new class of MV-SDEs adding a new element in the diffusion coefficient by means of a reversion to the population mean expressed through a \textit{fully non-Lipschitz} $f_{\sigma} \ast \mu$ significantly beyond the linear interaction diffusion coefficients studied in the mentioned works. 

More generally, the motivation to study this class of MV-SDEs and associated interacting particle systems is to present a unified framework to address wellposedness and establish properties useful for downstream applications. For instance, from emerging models of mean-field type in neuroscience \cite{erny2021strong}, understating particle motion and exit times \cite{HerrmannImkellerPeithmann2008,herrmann2012self,2013doublewell,adams2020large,MostLIkelyTransitionPath2022}, parametric inference \cite{GENONCATALOT2021,belomestny2021semiparametric,comte2022nonparametric}. 
We also point to Section 4 and 5 of \cite{jin2020random} for a variety of general interacting systems that are subsumed by our class. Our results can also be viewed as an addition to the literature on granular media type equations as studied in \cite{carrillo2003kinetic,guillin2019uniform,ren2021exponential}.

The existence and uniqueness of solutions to MV-SDEs, in a strong and weak sense, has been extensively studied, see e.g., \cite{RAY2,Sznitman1991,mishura2020existence,Lacker2018,HuangRenWang2021DDSDE-wellposedness,Malrieu2003,rockner2018well,C1,kalinin2022stability,mehri2020propagation,zacONE2023malliavin} and references therein, but \textit{none} cover the setting presented here. To the best of our knowledge, the existence and uniqueness of strong solutions to equations with super-linear growth in the measure component of the drift and the diffusion has not been addressed in general. There exist various works considering super-linearly growing coefficients (in state) but do not incorporate $f$ or $f_{\sigma}$, see e.g., \cite{mehri2020propagation,Wang2018DDSDE-LandauType,kumar2020wellposedness} and its references. 
In \cite{adams2020large} the authors deal with a super-linear $f$, $f_{\sigma} \equiv 0$,  and a (unbounded) uniformly Lipschitz continuous $\sigma$, and derive wellposedness (i.e., existence and uniqueness of a strong solution) and large deviation results. Further, \cite{zhang2021existence} allows for a setting similar to ours but requires upfront strong dissipativity and non-degeneracy. Their aim was to study  ergodicity, nonetheless, it is unclear how to adapt their methodology if working with the goal of proving wellposedness over $[0,T]$ under milder conditions. From the initial work \cite{adams2020large}, our goal is to develop a general framework to study \eqref{Eq:General MVSDE}-\eqref{Eq:General MVSDE shape of v} in terms of wellposedness (over $[0,T]$), with a super-linearly growing $\sigma$ and $f_{\sigma}$, ergodicity and approximation schemes. 
 
Our \textit{first main contribution} concerns wellposedness and propagation of chaos (PoC) results for the finite time horizon $[0,T]$ case. The critical nontrivial hurdle of this setting is in establishing $L^p$-moment bounds for $p>2$ under the  presence of the super-linear growths of $f$, $f_{\sigma}$ and $\sigma$ -- \textit{this issue appears solely due } to the simultaneous presence of nonlinearities (in space and measure) in the drift and diffusion, otherwise techniques like those of \cite{adams2020large} or \cite{kumar2020wellposedness} would suffice. 
To overcome this hurdle, we introduce a new condition dubbed `additional symmetry', which is new in the literature (to the best of our knowledge). For a quick perspective, we suggest a glance at Lemma \ref{AppendixLemma-aux1} and \ref{AppendixLemma-aux2} (in Appendix) and the proof of Theorem \ref{Thm:MV Monotone Existence} to see how one deals with the convolution terms and the necessity of the `additional symmetry' condition -- a discussion is presented in Remark \ref{rem-on-additionalSymmetry}. 
We also address a propagation of chaos result \cite{Sznitman1991,Meleard1996,fournier2015rate,delarue2018masterCLT,dosReisSalkeldTugaut2017} for this class. We show the interacting particle system, obtained by replacing $\mu^X$ by the system's $N$-particle empirical distribution, recovers the original MV-SDE in the particle limit $N\to \infty$. Under a very mild higher-integrability assumption, a convergence rate is obtained.

Our \textit{second main contribution} addresses another key element of MV-SDE theory which is the existence and uniqueness of an invariant probability measure and  exponential convergence to it (i.e., ergodicity). This is a particularly important property in applications involving statistical inference \cite{belomestny2021semiparametric} (usually in neuroscience \cite{ableidinger2017stochastic,buckwar2020spectral,buckwarbrehier2021FHNmodelandsplittingBSTT2020,genon2021parametric,genon2022inference}) or associated long-term behaviour connected to metastability \cite{2013doublewell,tugaut2013self,adams2020large}. 
   
    We extend the wellposedness result to the infinite time horizon and then analyse the long-time behaviour of our class of MV-SDEs. We prove an exponential ergodicity property and the existence of an invariant measure.
    The proof arguments loosely follow those of \cite{Wang2018DDSDE-LandauType,hu2021long,liu2022ergodicity} and, critically, do not make use of Lyapunov functions or the Krylov--Bogoliubov machinery. 
    In fact, to reach this type of results for McKean--Vlasov equations, the Krylov--Bogoliubov machinery is not a suitable one due to the nonlinearities of the involved semi-group and hence the classical tightness argument does not apply, see \cite{Wang2018DDSDE-LandauType,hu2021long,liu2022ergodicity}. 
    Lyapunov function arguments are also difficult to use for the particular MV-SDE class in this manuscript -- this is due to the presence of a convolution operator with a function $f$ (and $f_{\sigma}$) that is not of linear growth. If the polynomial exponent was applied to the convolution term, instead of being a convolution of it, then the Lyapunov machinery would successfully carry through \cite{liu2022ergodicity}. 
    For our case, an additional difficulty arises \textit{solely due} to the simultaneous presence of the super-linear growth in $f$, $f_{\sigma}$ and $\sigma$. The ergodicity/invariance proof arguments of  \cite{Wang2018DDSDE-LandauType,hu2021long,liu2022ergodicity} leverage the completeness of the space of probability measures with finite second order moments to identify the invariant measure, but our wellposedness result requires a '\textit{sufficiently integrable}' initial condition (with strictly more than second order moments). This issue leads to a more involved proof -- see discussion prior to Theorem \ref{Thm:ergodicity}.

    The \textit{third main contribution} of this work is a numerical method to approximate \eqref{Eq:General MVSDE}-\eqref{Eq:General MVSDE shape of v} over $[0,T]$ via its interacting particle system. Most of our theoretical results are only proven for the finite time case, but we successfully apply the scheme for the long-term simulation of a particle system as well. 
    There are presently many studies for numerical methods allowing super-linear spatial growth of drifts (and diffusions): Euler type methods, e.g., taming \cite{reis2018simulation}, time-adaptive \cite{reisinger2020adaptive}, semi-implicit methods \cite{2021SSM}, projection methods \cite{belomestny2018projected}; Milstein type methods e.g., \cite{kumar2020explicit,neelima2020wellposedness,bao2021first,bao2020milstein} with some allowing super-linear $\sigma$ in space. There are variations on the assumptions, but \textit{all} these contributions require drifts and diffusions to be globally Lipschitz continuous in measure (with respect to the Wasserstein distance with quadratic cost). Two recent contributions \cite{li2022strong,Mezerdi2021Caratheodory} allow for weaker continuity conditions than Lipschitz for the coefficients but require a linear growth in space and measure. Only \cite{chen2022SuperMeasure} allows for general super-linear growth in $f,u$ but still limits $\hs$ to satisfy  Lipschitz assumptions -- we detail below the differences between \cite{chen2022SuperMeasure} and this manuscript in more detail. 
    Lastly, we mention the recent work \cite{chen2024improved} proposing a non-Markovian type Euler scheme for infinite time horizon with a weak error rate of 1.5 for a class of (Langevin) MV-SDE with constant diffusion coefficients.
    
    The scheme we propose belongs to the split-step method (SSM) class. Such schemes were analysed for MV-SDEs with drifts which are Lipschitz in measure and diffusion coefficients satisfying  a uniform Lipschitz condition \cite{2021SSM}; the scheme appeared originally in \cite{higham2002strong} for standard SDEs. We follow the strategy of approximating (in time) the interacting particle system associated with the MV-SDE and using a quantitative propagation of chaos (PoC) convergence result, see \cite{bossytalay1997} and \cite{reisinger2020adaptive,neelima2020wellposedness,reis2018simulation} for earlier uses of this strategy. 
    From a methodological point of view, the convergence proof of our numerical scheme is different from any used to study MV-SDE numerical schemes in the literature; our highly non-linear setting forces us to draw on the stochastic $C$-stability and $B$-consistency mechanics proposed in \cite{2015ssmBandC}. Its use in the context of numerical schemes for MV-SDEs and interacting particle systems is novel in the literature  -- except for the very recent \cite{biswasetal2022RandomisedMS} that studies higher order strong scheme for MV-SDE under non-differentiability conditions using a randomisation method. In \cite{biswasetal2022RandomisedMS},  the authors work with generic Lipschitz assumptions and need to change the underpinning error norms to cope with the complexity arising from the randomization step due to an explicit non-differentiability assumption on the drift coefficient. Our approach and requirements differ, and so does the analysis (albeit similar at points). 
    We show that it is possible to work directly with the concepts of \cite{2015ssmBandC} to deal with the interacting particle system, see Section \ref{sec:C-B} -- we emphasize that the main goal of the analysis is to guarantee that  core moment estimates are uniformly independent of the number of particles $N$ of the interacting system (but may depend on the initial system's underlying dimension $d$). 
    
    Closest to our work with regards to the SSM is \cite{chen2022SuperMeasure}, where the authors propose an SSM scheme similar to the one here for interacting particle systems that have \eqref{Eq:General MVSDE} (with $f_{\sigma} \equiv 0$ and $\sigma$ globally Lipschitz continuous in space and measure) as limit. There, they overcome the barrier of super-linear growth in space and measure for the drift (that \cite{reis2018simulation,reisinger2020adaptive,2021SSM} do not), but work with a diffusion component of Lipschitz type (\cite{chen2022SuperMeasure} focuses solely on the numerical scheme not on wellposedness nor ergodicity). 
Our setting is more involved than \cite{chen2022SuperMeasure,2021SSM} and requires novel proof techniques to deal with the simultaneous super-linearity in drift and diffusion. 
In \cite{chen2022SuperMeasure}, higher order moments bounds of the discrete process could by obtained by commonly used assumptions like  $(\mathbf{A}^u,~\mathbf{A}^\sigma)$ (see Assumption \ref{Ass:Monotone Assumption} below).
For our situation, the super-linearity in the diffusion coefficient (in space and measure) is controlled by a drift satisfying a suitable one-sided Lipschitz condition. However, the simultaneous appearance of nonlinearities in the diffusion and the nonlinear convolution $f \ast \mu$ in the drift causes difficulties. 
This is the reason why, for the scheme in this manuscript, only $L^2$-moment bounds were established, and the proof methodology takes recourse in the stochastic $C$-stability and $B$-consistency mechanics of \cite{2015ssmBandC} (which does not require to establish bounds for higher order moments of the SSM). It remains unclear how to obtain higher moments.

    In terms of findings, we show the scheme achieves a strong convergence rate of order $1/2$ and establish sufficient conditions for mean-square stability in the sense of \cite[Definition 2.8]{2021SSM}. 
    We present several numerical examples of interest and for comparison we implement, without proof, two intuitive versions of taming methods \cite{reis2018simulation,neelima2020wellposedness}. 
    Our examples show the SSM to perform very well for the approximation of a solution to \eqref{Eq:General MVSDE} for $T<\infty$ in an $L^{2}$-sense or the approximation for the ergodic distribution. The numerical results using taming are mixed but hint which version can be expected to converge (theoretically). \textit{We show a surprising numerical divergence finding for taming given the choice of initial condition and which does not appear when using the SSM } (see our Section \ref{section:example: dw1}) -- this confirms the SSM as a stable/robust choice  of scheme for this class. 
    The SSM is shown to preserve periodicity of phase-space. 
    Lastly, we provide a numerical example with the aim of estimating the PoC rate across dimensions and which highlights a gap in the literature: we observe the rate of \cite{Meleard1996,delarue2018masterCLT} (not applicable to our setting) instead of those in \cite{fournier2015rate,CarmonaDelarue2017book1} (which are used to prove our PoC result). Future research focuses on the study of uniform in-time PoC results and strong convergence rates for the SSM on $[0,\infty)$.

\smallskip 
\textbf{Organization of the paper.} Section \ref{sec:two} contains: notations, framework, wellposedness and ergodicity results, the particle approximation and the propagation of chaos statement, and the numerical scheme alongside associated convergence results. Several numerical examples are provided in Section \ref{sec:examples}. They cover the non-dissipative case in short and long time horizons; approximation of the invariant distribution; preservation of periodicity in phase-space and numerical estimation of PoC rates across dimension. 
All proofs are postponed to Section \ref{section: proof of the main results}. Generic auxiliary results are given in the Appendix.



\section{Main results}
\label{sec:two}

\subsection{Notation and Spaces}

We follow the notation and framework set in  \cite{adams2020large,2021SSM}. 
Let $\bN$ be the set of natural numbers starting at $0$ and for $a,b\in \bN$ with $a\leq b$,  define $\llbracket a,b\rrbracket:= [a,b] \cap \bN =  \{a,\ldots,b\}$. 
For $x,y \in \bR^d$ denote the inner product of vectors by   $\langle x, y\rangle$, 
and $|x|=(\sum_{j=1}^d x_j^2)^{1/2}$ the Euclidean distance. 
Let $\1_B$ be the indicator function of the set $B\subset \bR^d$. For a matrix $A \in \bR^{d\times l}$ we denote by $A^\intercal$ its transpose and its Frobenius norm by $|A|=\textit{Trace}\{A A^\intercal\}^{1/2}$.  
 
We introduce on the measurable space $(\bR^d,\mathcal{B}(\mathbb{R}^d))$, where $\mathcal{B}(\mathbb{R}^d)$ denotes the Borel $\sigma$-field over $\mathbb{R}^d$, the set of all probability measures $\cP(\bR^d)$ and its subset $\cP_r(\bR^d)$ of those with finite $r\in [1,\infty)$ moment. The space $\cP_r(\bR^d)$ is a Polish space when endowed with the Wasserstein distance
\begin{align*} 
W^{(r)}(\mu,\nu) := \inf_{\pi\in\Pi(\mu,\nu)} \Big(\int_{\bR^d\times \bR^d} |x-y|^r\pi(\dd x,\dd y)\Big)^\frac1r, \quad \mu,\nu\in \cP_r(\bR^d),
\end{align*}   
where $\Pi(\mu,\nu)$ is the set of couplings for $\mu$ and $\nu$ such that $\pi\in\Pi(\mu,\nu)$ is a probability measure on $\bR^d\times \bR^d$ with $\pi(\cdot\times \bR^d)=\mu$ and $\pi(\bR^d \times \cdot)=\nu$. 
For a given function $f$ with domain on $\bR^d$, $x\in\bR^d$ with $\mu \in \cP(\bR^d)$, the convolution operator $\ast$ is defined as $(f\ast \mu) (x):=\int_{\bR^d} f(x-y) \mu(\mathrm{d}y)$.
\color{black}
Let our probability space be a completion of $(\Omega, \bF, \cF,\bP)$ with $\bF=\lbrace \cF_t \rbrace_{t\geq 0}$ being the natural filtration of the Brownian motion $W=(W^1,\ldots,W^l)$ with $l$-dimensions, 
{\color{black}  
augmented with a sufficiently rich sub $\sigma$-algebra $\cF_0$ independent of $W$. We denote by $\bE[\cdot]=\bE^\bP[\cdot]$ the usual expectation operator with respect to $\bP$.  }

We consider some finite terminal time $T<\infty$ and use the following notation for spaces, which are standard in the literature \cite{reis2018simulation,2021SSM}. For $p \geq 1$, we denote by $L^{p}\left(\Omega,\mathcal{F}_t,\mathbb{P};\mathbb{R}^{d}\right)$ the space of $\mathbb{R}^{d}$-valued, $\mathcal{F}_t$-measurable random variables $X$ with finite $|X|_{L^{p}(\Omega,\mathcal{F}_t,\mathbb{P};\mathbb{R}^d)}$-norm given by  $|X|_{L^{p}( \Omega,\mathcal{F}_t,\mathbb{P};\mathbb{R}^d)}=\mathbb{E}\left[|X|^{p}\right]^{1 / p}$.  
Define $\mathbb{S}^{p}([0,T])$, as the space of $\mathbb{R}^{d}$-valued, $\mathbb{F}$-adapted, continuous processes $Z$ with finite $\|Z\|_{\mathbb{S}^{p}}$-norm defined as  $\|Z\|_{\mathbb{S}^{p}}=\mathbb{E}\left[\sup _{0 \leqslant t \leqslant T}|Z_t|^{p}\right]^{1 / p}$.

Throughout the text, $C$ denotes a generic positive real-valued constant that may depend on the problem's data, and change from line to line, but is always independent of the constants $h,M,N$ (associated with the numerical scheme and specified below).

\subsection{Framework}

Let $v:\bR^d \times\cP_2(\bR^d) \to \bR^d$, $b:[0,\infty) \times \bR^d \times\cP_2(\bR^d) \to \bR^d$ and $\hs:[0,\infty) \times \bR^d \times \cP_2(\bR^d) \to \bR^{d\times l}$ be measurable maps.
The MV-SDE of interest of this work is Equation \eqref{Eq:General MVSDE} (for some $m > 2$), where $\mu_{t}^{X}$ denotes the law of the process $X$ at time $t$, i.e., $\mu_{t}^{X}=\bP\circ X_t^{-1}$. 
We make the following assumptions on the coefficients.
\begin{assumption}
\label{Ass:Monotone Assumption}
The functions $b$ and $\sigma$ are $1/2$-H\"{o}lder continuous in time, uniformly in $x\in \bR^d$ and $\mu\in \cP_2(\bR^d)$ and $\sup_{t \in [0,\infty)} \big(|b(t,0,\delta_{0})| + |\sigma(t,0,\delta_{0})| \big)   \leq L$, for some constant $L \geq 0$.  
\\
$ (\mathbf{A}^b)$ Let $b $ be uniformly Lipschitz continuous in the sense that there exists $L_{(b)}^{(1)},~L_{(b)}^{(3)} \ge0$ and $ L_{(b)}^{(2)}\in \bR$ such that for all $t \in[0,\infty)$, $x, x'\in \bR^d$ and $\mu, \mu'\in \cP_2(\bR^d)$ we have that 
\begin{align*}
|b(t, x, \mu)-b(t, x', \mu')|^2
&\leq L_{(b)}^{(1)} \big(|x-x'|^2 + \big( W^{(2)}(\mu, \mu')\big)^2 \big),
\\
\langle x-x', b(t, x, \mu)-b(t, x', \mu') \rangle
&\leq L_{(b)}^{(2)} |x-x'|^2  + L_{(b)}^{(3)} \big( W^{(2)}(\mu, \mu')\big)^2 
.
\end{align*}
$(\mathbf{A}^u,~\mathbf{A}^\sigma)~$ Let $u, \sigma$  satisfy: there exist $ L^{(1)}_{(u\sigma)} \in \bR$, and $L^{(2)}_{(u\sigma)},L^{(3)}_{(u\sigma)},L^{(4)}_{(u\sigma)},q \ge 0$ 
such that for all $t\in[0,\infty)$, $ x, x'\in \bR^d$ and $\mu, \mu'\in \cP_2(\bR^d)$, with $m>2$ in \eqref{Eq:General MVSDE},  we have that 
\begin{align}
\label{eq:condition:driftoffsetdiffusion}
\color{black}
 &\langle x-x', u(x,\mu) -  u(x',\mu') \rangle +  2(m-1)  |\sigma(t, x, \mu)- \sigma(t, x', \mu') |^2 
\\\nonumber
&\quad\quad
\leq L^{(1)}_{(u\sigma)} |x-x'|^{2}+L^{(2)}_{(u\sigma)}\big(W^{(2)}(\mu, \mu') \big)^2, 
\\ 
\label{eq:condition:drifdiffusiongrowth}
&|u(x,\mu)-u( x',\mu)|  + |\sigma(t, x, \mu)-\sigma(t, x', \mu)|
\\ \nonumber
&\quad\quad
\leq L^{(3)}_{(u\sigma)}(1+ |x|^{q} + |x'|^{q}) |x-x'|,
\\ \nonumber
&|u(x,\mu)-u( x,\mu')|^2 + |\sigma(t, x, \mu)-\sigma(t, x, \mu')|^2
\leq 
L^{(4)}_{(u\sigma)} \big( W^{(2)}(\mu, \mu')\big)^2.  
\end{align}

$(\mathbf{A}^f,~\mathbf{A}^{\fs})~$ Let $f, \fs$  satisfy: there exist $ L^{(1)}_{(f)},L^{(3)}_{(f)}\in \bR$, and  $L^{(2)}_{(f)},q\ge 0$, such that for all $ x, x'\in \bR^d$, $2< p\leq m $,  we have that 
\begin{align*}
\langle x-x', f(x)-f(x') \rangle 
+
2(m-1)&|\fs(x) - \fs(x')|^2
\\
\qquad 
\leq L^{(1)}_{(f)}|x-x'|^{2},
 &\qquad \textrm{(One-sided Lipschitz, monotonicity condition)},
\\
 |f(x)-f( x')| 
+
|\fs(x) - \fs(x')
&\leq L^{(2)}_{(f)}(1+ |x|^{q} + |x'|^{q}) |x-x'| ,
~ \textrm{(Locally Lipschitz)},
\\
f(x)&=-f(-x), 
\qquad \qquad \qquad \qquad  \qquad \qquad \textrm{(Odd function)},
\\ 
( |x|^{p-2}-|x'|^{p-2})\langle  x+x', f(x-x') \rangle
&\leq L^{(3)}_{(f)}(|x|^{p}+|x'|^{p}  )
\qquad   \qquad \textrm{(Additional symmetry)}
. 
\end{align*}

Assume the normalization\footnote{This constraint is not restrictive since the framework allows to easily redefine $f$ as $\hat f(x):=f(x)-f(0)$ with $f(0)$ merged into $b$.} $f(0)=f_{\sigma}(0)=0$.
\end{assumption}  

\begin{remark}[Time dependency for $u$]
To avoid added complexity to an already complex work, we do not address time-dependence on $u$. A close inspection of the proof for wellposedness and convergence of the numerical scheme shows that as long as the time dependence does not interfere with constraints imposed by Assumption \ref{Ass:Monotone Assumption} the results will hold. Additionally one would require a $1/2$-H\"{o}lder continuity property for the function.
\end{remark}

All elements in the above assumption are standard, except the `additional symmetry' restriction. 
The `additional symmetry' is a new type of restriction which we have not found previously in the literature and we discuss it in more detail at several points in the text, in particular, in Remark \ref{rem-on-additionalSymmetry}. 

This condition is trivially satisfied when $d=1$ (see \eqref{eq:additionalsymmetric-IN-1D}) or when the function is {\color{black} of linear growth}. We next provide a non-trivial example in $d>1$  for $f$ satisfying the `extra symmetry' condition.
\begin{example} 
For $x\in\bR^d$ define $f(x)=-x|x|^2$. Then, for any $p> 2$, $x,y\in\bR^d$ it holds that 
\begin{align*}
    (|x|^{p-2}-|y|^{p-2} )\langle  x+ y, -(x-y)|x-y|^2   \rangle    
    &=  - (|x|^{p-2}-|y|^{p-2} ) (|x|^{2}-|y|^{2} )
    |x-y|^2
    \leq 0,
\end{align*}
\color{black} 
and the conclusion follows from the monotonicity of the polynomial function. 
\end{example}

\begin{remark}[Implied properties]
\label{remark:ImpliedProperties}
Let Assumption \ref{Ass:Monotone Assumption} hold with $m>2$. We provide the following estimates for some positive constant $C$ which may change line by line, {\color{black}which are derived using the one-sided Lipschitz condition and Young's inequality }(see \cite[Remark 2.2]{chen2022SuperMeasure} for details). For all $t \in [0,T]$, $x,x',z\in \bR^{d}$ and $\mu,\mu'\in \cP_2(\bR^d)$,   we have 
\begin{align}
    \label{eq: remark eq1}
    \langle x,f(x)\rangle + 2(m-1) |\fs(x)|^2  &\leq L^{(1)}_{(f)}|x|^2,
    \\
    \label{eq: remark eq2}
    | b(t,x,\mu)|^2
   &\le
   C \big(1+|x|^2+(W^{(2)}(\mu,\delta_0))^2 \big),
    \\
    \label{eq: remark eq3}
        \langle x,u(x,\mu)\rangle +(m-1) | \sigma(t,x,\mu)|^2
    &\leq
    C \big(1+ |x|^2 +  (W^{(2)}(\mu,\delta_0))^2 \big)
    ,
    \\\label{eq: remark eq4}
    \langle x-x',u(x,\mu)-u(x',\mu') \rangle 
     &\le 
     C \big(  |x-x'|^2 +  (W^{(2)}(\mu,\mu'))^2 \big),  
    \\
    \label{eq: remark eq5}
    \langle x,b(t,x,\mu) \rangle 
    &\le
    C(1+|x|^2+(W^{(2)}(\mu,\delta_0))^2 ),
    \\
    \label{eq: remark eq6}
    \langle x-x',v(x,\mu)-v(x',\mu) \rangle
     &\le (L^{(1)}_{(u\sigma)}+L^{(1)}_{(f)}) |x-x'|^2, 
   \\
   \nonumber 
   |\hs(t,x,\mu)|^2
   &\leq 2|\sigma(t,x,\mu)|^2
   +
   2\big|\int_{\bR^{d}  } \fs(x-y) \mu(\dd y)\big|^2
   \\
   \label{eq: remark eq7}
   & \leq
   2\Big( |\sigma(t,x,\mu)|^2
   +
   \int_{\bR^{d}  } |\fs(x-y) |^2\mu(\dd y) \Big) 
   .
\end{align}

Let $f:\bR\rightarrow\bR$  satisfying the one-sided Lipschitz condition, then $f$ satisfies the additional symmetry condition, i.e., for $x,y\in \bR, x\neq y, p\ge 2$, we have
\color{black}   
\begin{align}
\nonumber 
    (|x|^{p-2}-|y|^{p-2} )\langle  x+ y, f(x-y)  \rangle 
    &=\frac{(|x|^{p-2}-|y|^{p-2} )(x+y) }{ x-y  }
        \langle x-y  , f(x-y)   \rangle
    \\ \label{eq:additionalsymmetric-IN-1D}
    &\leq C (|x|^p+|y|^p).
\end{align}
\color{black}
The following decomposition is crucial for the remaining parts of this work
\color{black} 
for $x\in\bR^d$, $m \geq p > 2$, $\mu \in \cP_m(\bR^d)$ it holds that 
\begin{align}
    \nonumber
    \int_{\bR^d}& \int_{\bR^d} |x|^{p-2} \langle  x, f(x-y)  \rangle \mu(\mathrm{d}y)\mu( \mathrm{d}x)
    = \frac{1}{2}
    \int_{\bR^d} \int_{\bR^d}  \langle |x|^{p-2} x
    - |y|^{p-2} y
    , f(x-y)  \rangle \mu(\mathrm{d}y)\mu(\mathrm{d}x)
    \\
    \nonumber
    & = \frac{1}{2}
     \int_{\bR^d} \int_{\bR^d}  \langle (|x|^{p-2} x
    - |x|^{p-2} y) + (|x|^{p-2} y
    - |y|^{p-2} y)
    , f(x-y)  \rangle \mu(\mathrm{d}y)\mu(\mathrm{d}x)
    \\
    \label{eq: lp moment expectation result}
    &= \int_{\bR^d} \int_{\bR^d}
    \big( \tfrac{1}{2}
    |x|^{p-2}\langle  x-y
    , f(x-y)  \rangle
    +\tfrac{1}{4}
    (|x|^{p-2} - |y|^{p-2}  ) 
    \langle x+ y
    , f(x-y)  \rangle
    \big)
     \mu(\mathrm{d}y)\mu(\mathrm{d}x).
\end{align}
{\color{black}   
The decomposition in  \eqref{eq: lp moment expectation result} along with $(\mathbf{A}^f,~\mathbf{A}^{\fs})~$ will be used to incorporate the nonlinearity of $f_{\sigma}$.}
    
\end{remark}

\subsection{Existence, uniqueness and  ergodicity of the MV-SDE}

Let us start by stating the wellposedness result of MV-SDE \eqref{Eq:General MVSDE}. 
\begin{theorem}[Wellposedness] 
\label{Thm:MV Monotone Existence}

	Let Assumption \ref{Ass:Monotone Assumption} hold with $m> 2q+2$, then there exists a unique strong solution $X$ to MV-SDE \eqref{Eq:General MVSDE} satisfying the following estimates:   
	For some constant $C>0$, we have a pointwise estimate 
\begin{align*}
	\sup_{t\in[0,T]} \mathbb E \big[  |X_{t}|^{\widetilde m} \big] 
	\leq C \left(1+ \bE[|X_0|^{\widetilde m}]\right) e^{C T},\qquad \textrm{for any }~ \widetilde m \in [2,m].
\end{align*} 
\end{theorem}
The proof of the wellposedness theorem is postponed to Section \ref{subsection: proof of well-pose and momentbound}.  
\begin{remark}[On the `additional symmetry' restriction]
\label{rem-on-additionalSymmetry}
The critical element of the proof for this result, is the difficulty in establishing (finite) bounds for higher order moments of the solution process. 
The `additional symmetry' assumption is a technical condition without which we were not able to establish $L^p$-moment bounds for $p>2$ (and $d>1$) -- proving $L^2$-moment bounds or uniqueness of the solution is straightforward and the condition is not needed. The requirement of `additional symmetry' stems solely from having a super-linearly growing $\sigma$, $f_{\sigma}$ \textit{and} a super-linear growth of the convolution term appearing in the drift. If either of them is of linear growth (or $d=1$), then the `additional symmetry' condition can be removed and the results hold.

The strategy used in \cite{adams2020large} to establish $L^p$-moment bounds, working with Assumption \ref{Ass:Monotone Assumption} but with a linearly growing $\sigma$, is to bound $\bE\big[|X_{t}|^{2p} \big]$ via 
    \begin{align*}
    \bE\big[|X_{t}|^{2p} \big] 
    &\le
     C \big( \bE\big[\big |X_{t}-\bE[X_t]\big|^{2p} \big]+   \bE\big[|X_t|^2\big]^{p}  \big),
    \end{align*}
and then noticing that     
    \begin{align*}
     \bE\big[\big |X_{t}-\bE[X_t]\big|^{2p} \big]
       & =
       \int_{\bR^d} 
         \big|x-\int_{\bR^d} y\mu_t(\dd y) \big|^{2p}
         \mu_t(\dd x)
      \\ &
      \le
         \int_{\bR^d} \int_{\bR^d}
         |x- y |^{2p}\mu_t(\dd y)
         \mu_t(\dd x) 
    =  \bE\big[|X_{t}-\tilde{X}_t|^{2p} \big], 
    \end{align*}
with $\tilde X$ an independent copy of $X$ driven by its independent Brownian motion, see Lemma \ref{AppendixLemma-aux1} and Lemma \ref{AppendixLemma-aux2} for extra details. This trick allows to deal with the convolution term, employing its symmetry, (see Lemma \ref{AppendixLemma-aux2}), but does not give control of the super-linear diffusion. To be precise, It\^o's formula applied to $|X-\tilde X|^{2p}$ forces one to use the polynomial growth condition on $\sigma$ \eqref{eq:condition:drifdiffusiongrowth}, which  involves higher moments, instead of \eqref{eq:condition:driftoffsetdiffusion}.

Without the trick described above, and following more classical approaches \cite[Theorem 2.1]{neelima2020wellposedness}, it is possible to control the super-linear growth of $\sigma$ in space (via \eqref{eq:condition:driftoffsetdiffusion}) but it is unclear how to simultaneously control the super-linear growth of the convolution terms in a tractable way (the tricks of Lemma \ref{AppendixLemma-aux1} and Lemma \ref{AppendixLemma-aux2} do not carry over). 

All in all, there is competition between the growths of $f$ and $\sigma$, $f_{\sigma}$, and neither just described technique is adequate to establish $L^p$-moment estimates. The `additional symmetry' condition offsets this difficulty. See details in the proof in Section \ref{subsection: proof of well-pose and momentbound}. Lifting this restriction is left as an open question.  
\end{remark}


\subsection{Particle approximation of the MV-SDE}

We now turn to the particle approximation of the MV-SDE with the ultimate goal of establishing a working numerical scheme for the equation. All results here are only concerned with the finite-time case. 

As in \cite{bossytalay1997,reisinger2020adaptive,2021SSM}, we approximate the MV-SDE \eqref{Eq:General MVSDE} (driven by the Brownian motion $W$) by an interacting particle system, i.e., an $N$-dimensional system of $\bR^d$-valued interacting particles. Let $i\in \llbracket 1,N\rrbracket$ and consider $N$ particles $(X_t^{i,N})_{t\in[0,T]}$ with independent and identically distributed (i.i.d.)~initial data ${X}_{0}^{i,N}=X_{0}^{i}$ (an independent copy of $X_0$) satisfying the $\bR^{Nd}$-valued SDE with components
\begin{align}
\label{Eq:MV-SDE Propagation}
\dd {X}_{t}^{i,N} 
= \big( v(X_t^{i,N}, \mu^{X,N}_{t} )+ b(t,{X}_{t}^{i,N}, \mu^{X,N}_{t} )\big) \dd t 
+ \hs(t,{X}_{t}^{i,N} , \mu^{X,N}_{t} ) \dd W_{t}^{i}
, \quad X^{i,N}_0=X_0^i, 
\end{align}
where $\mu^{X,N}_{t}(\dd x) := \frac{1}{N} \sum_{j=1}^N \delta_{X_{t}^{j,N}}(\dd x)$ with $\delta_{x}$ being the Dirac measure at $x \in \mathbb{R}^d$, and  $W^{i}$ being independent Brownian motions (also independent of the Brownian motion appearing in \eqref{Eq:General MVSDE}).
\color{black}  
We introduce similarly to \cite[Remark 2.4]{chen2022SuperMeasure} the auxiliary maps $V$, and $\hat{\Sigma}$ to view \eqref{Eq:MV-SDE Propagation} as a system in $\bR^{Nd}$.
\begin{lemma}[Properties of the particle system as a system in $\bR^{Nd}$]
\label{remark:OSL for the whole function / system V}
 Define $V :\bR^{Nd}\to \bR^{Nd}$, $ \hat{\Sigma}: [0,T]\times \bR^{Nd} \to \bR^{Nd \times Nl}$ by $V(x^N)= \big(\ldots,v(x^{i,N}, \mu^{x,N}),\ldots \big),$ and $\hat{\Sigma}(t,x^N)=\big(\ldots,\hs(t,x^{i,N}, \mu^{x,N}),\ldots \big)$ with $x^N=(x^{1,N},\ldots,x^{N,N} )\in \bR^{Nd}$, $~t\in[0,T]$. 
 
 Then, under Assumption \ref{Ass:Monotone Assumption} with $m>2$, for any $x^{N}, y^{N}\in \bR^{Nd}$ with corresponding empirical measures $\mu^{x,N} = \tfrac{1}{N} \sum_{j=1}^{N} \delta_{x^{j,N}}$, and $\mu^{y,N} = \tfrac{1}{N} \sum_{j=1}^{N} \delta_{y^{j,N}}$, the functions $V,~\hat{\Sigma}$ also satisfy a One-sided Lipschitz  {\color{black}    (see first item of $(\mathbf{A}^f,~\mathbf{A}^{\fs})~$ in Assumption \ref{Ass:Monotone Assumption})} in $\bR^{N d}$ (with constants independent of $N$).

\end{lemma}
\begin{proof}
From Assumption \ref{Ass:Monotone Assumption}, \eqref{eq: remark eq1}, \eqref{eq: remark eq3} in Remark \ref{remark:ImpliedProperties} and Jensen's inequality, we deduce, for all $x^{N}, y^{N}\in \bR^{Nd},~t\in[0,T]$,
\begin{align*}
    &\langle x^N-y^N,V(x^N)-V(y^N) \rangle 
    +  \frac{(m-1)}{2} | \hat{\Sigma}(t,x^N) -  \hat{\Sigma}(t,y^N)    |^2
    \\
    &
    \leq \frac{1}{2N}\sum_{i=1}^N\sum_{j=1}^N \big\langle (x^{i,N}-x^{j,N})-(y^{i,N}-y^{j,N}) 
    , f(x^{i,N}- x^{j,N})-f(y^{i,N}- y^{j,N})\big\rangle
    \\
    &~ + 
    \sum_{i=1}^N  \Big(\big\langle x^{i,N}-y^{i,N}, 
     u(x^{i,N}, \mu^{x,N})-   u(y^{i,N}, \mu^{y,N}) 
    \big\rangle\\
    &~ \qquad \qquad + 
     (m-1)| \sigma(t,x^{i,N}, \mu^{x,N})-\sigma(t,y^{i,N}, \mu^{y,N}) |^2
    \Big)
    \\
    &~ +\frac{m-1}{N} 
    \sum_{i=1}^N
    \sum_{j=1}^N |   \fs(x^{i,N}- x^{j,N})-\fs(y^{i,N}- y^{j,N}) |^2 
    \le C |x^N-x^N|^2,
\end{align*}
where $C>0$ is independent of $N$.

\end{proof}
\color{black}

\textbf{Propagation of chaos (PoC).} 
In order to show that the particle approximation \eqref{Eq:MV-SDE Propagation} is effective to approximate the underlying MV-SDE, we present a pathwise propagation of chaos result (convergence as the number of particles increases). 
To do so, we introduce the system of non interacting particles  
\begin{align}
	\label{Eq:Non interacting particles}
	\dd X_{t}^{i} = 
	\big( 
	v(X_{t}^{i}, \mu^{X^{i}}_{t})+ b(t, X_{t}^{i}, \mu^{X^{i}}_{t}) \big)\dd t + \hs(t,X_{t}^{i}, \mu^{X^{i}}_{t}) \dd W_{t}^{i} ,\quad t\in [0,T],
\end{align}
which are (decoupled) MV-SDEs with i.i.d. initial conditions  $X_{0}^{i}$ (an independent copy of $X_0$). Since the $X^{i}$'s are
independent, $\mu^{X^{i}}_{t}=\mu^{X}_{t}$ for all $
i$ (and $\mu^{X}_{t}$ the marginal law of the solution to \eqref{Eq:General MVSDE}). 
We are interested in the strong error-type metrics for the numerical approximation and the relevant PoC result for our case is given in the next theorem,  the proof is postponed to Section \ref{section: proof of the main results}.  

\begin{theorem}[Propagation of Chaos]
\label{theorem:Propagation of Chaos}
	Let Assumption \ref{Ass:Monotone Assumption} hold for some $m > 2(q+1)$.  
Then, there exists a unique solution $X^{i,N}$ to \eqref{Eq:MV-SDE Propagation} and for any $1\leq p\leq m$ there exists $C>0$ independent of $N$ such that 
\begin{align*} 
    \sup_{ i\in \llbracket 1,N \rrbracket} \sup_{t\in[0,T]} \bE\big[  |X^{i,N}_t|^p \big] \leq C(1+\bE\big[\,|X_0|^p \big]).
\end{align*}

Moreover, suppose that $m > 2(q+1)$ and $m>4$, then we have the following convergence result 
	\begin{align}
	\label{eq:propogation of chaos, poc result} 
	\sup_{ i\in \llbracket 1,N \rrbracket}	\sup _{t \in[0, T]} \mathbb{E}\big[|X_{t}^{i, N}-X_{t}^{i}|^{2}\big] \leq C
	\begin{cases}N^{-1 / 2}, & d<4,
	\\ N^{-1 / 2} \log N, & d=4, 
	\\ N^{-\frac{2}{d+4}}, & d>4,\end{cases}
	\end{align}
{\color{black}where $X^{i}$ is the solution to \eqref{Eq:Non interacting particles} with driving Brownian motion $W^{i}$ (the same as for the $i$-th particle) in the sense of Theorem \ref{Thm:MV Monotone Existence}.}
\end{theorem}
This result shows that the particle approximation will converge to the MV-SDE with a given rate. Therefore, to establish convergence of our numerical scheme to the MV-SDE (in a strong sense), we only need to show that the discrete-time version of the particle system converges to the ``true'' particle system.


\subsection{Ergodicity of the MV-SDE}

Next, recall the constants $q,m$ from Assumption \ref{Ass:Monotone Assumption}, we consider the long-time behaviour, an exponential ergodic property and the existence of an invariant measure for the MV-SDEs of interest.
We point the reader to \cite{Wang2018DDSDE-LandauType,wang2021distribution1,wang2021distribution2} for a review on recent results. 
To this end, we need to estimate differences of \eqref{Eq:General MVSDE} with different initial conditions and introduce the associated nonlinear semigroup. 

Define the nonlinear semigroup $(P^*_{s,t})$ for $0\leq s \le t <\infty $ on $\cP_{\ell}(\bR^d),~\ell>2q+2$ by setting $P^*_{s,t}\mu := \textrm{Law}(X_{s,t})$ and $X_{s,\cdot}$ is the solution to \eqref{Eq:General MVSDE} starting from time $s$ such that $\textrm{Law}(X_{s,s})=\mu$. Note that standard literature sets the semigroup in $\cP_{2}(\mathbb{R}^d)$ but in this manuscript Theorem \ref{Thm:MV Monotone Existence} requires higher integrability of the initial condition and working with $\cP_{\ell}(\mathbb{R}^d),\ell>2q+2$ reflects that. 
In the notation introduced earlier, we have $P^*_{0,t} \mu_0^X:= \mu_t^X$, and more generally $P^*_{s,t}=P^*_{s,r}P^*_{r,t}$ for $s\le r \le t$. Crucially, if $b$ and $\sigma$ in \eqref{Eq:General MVSDE} are independent of time, then $ P^*_{s,t}=P^*_{0,t-s}$ (see \cite{Wang2018DDSDE-LandauType}). 

We say that $\bar \mu$ is an invariant distribution of the semigroup $P^*$ if $P^*_{0,t}\bar \mu=\bar \mu$ holds for all $t\geq 0$. The semigroup satisfies an ergodic property if there exists $\widehat \mu \in \cP_{\ell}(\mathbb{R}^d)$ such that $\lim_{t\to \infty} P^*_{0,t} \nu=\widehat \mu$ (weakly) for all $\nu$ (at this point, we leave unclear the space where $\nu$ belongs to). In the proof of the theorem below, we show that the property holds true for any $\nu\in \cP_{2\ell-2}(\mathbb{R}^d) \subset \cP_{\ell}(\mathbb{R}^d),~\ell>2q+2$ with convergence taking place through the $W^{(\ell)}$-metric in $\cP_{2\ell-2}(\mathbb{R}^d)  $. These results differ from those in \cite{Wang2018DDSDE-LandauType} and the proof requires further care. 

\begin{theorem}[Contraction, exponential ergodicity property and invariance]
\label{Thm:ergodicity}
Let Assumption \ref{Ass:Monotone Assumption} hold with $m > 4q+2$. Assume that there exist constants    $ L^{(1)}_{(bu\sigma)},L^{(3)}_{(bu\sigma)}\ge 0$, and $L^{(2)}_{(bu\sigma)}\in\bR$  such that for all $t\in [0,\infty)$, $x \in \bR^d$ and $\mu \in \cP_2(\bR^d)$ we have that  
\begin{align*} 
\big \langle x, u(x,\mu) + b(t,x,\mu) \big\rangle +  (m-1) |\sigma(t,x, \mu)|^2
& 
\leq L^{(1)}_{(bu\sigma)} + L^{(2)}_{(bu\sigma)}|x|^{2}+L^{(3)}_{(bu\sigma)}\big(W^{(2)}(\mu, \delta_0) \big)^2 .
\end{align*} 
Then the following three assertions hold:
\begin{enumerate}
    
    \item 
    Let $\mu \in \cP_\ell(\bR^d)$ with $2q+2 < \ell \leq m $,  
      $\rho_{1,\ell}= \ell( L^{(2)}_{(bu\sigma)}+ L^{(3)}_{(bu\sigma)}+2L^{(1),+}_{(f)}+L^{(3)}_{(f)}/2)+(\ell-2)/\ell$,  $t\in [0,T],~T<\infty$. \color{black} 
    Then for some constant $C$ depending on $\ell,~L^{(1)}_{(bu\sigma)}$  
    and $ \sup_t|b(t,0,\delta_0)|$, but independent of time $t$, we have  
    \begin{align}
    \label{eq:Bounded Orbit}
        \big(W^{(\ell)}(P^*_{0,t} \mu,\delta_0)\big)^\ell 
    \le   
    e^{\rho_{1,\ell} t}
    \big(W^{(\ell)}( \mu,\delta_0)\big)^\ell  
    +\frac{C}{\rho_{1,\ell}}(e^{\rho_{1,\ell} t}-1) \1_{ \rho_{1,\ell} \neq 0 }+ Ct  \1_{ \rho_{1,\ell} = 0 }
    . 
    \end{align}
    
    \item For $\mu,\nu\in \cP_{\ell}(\bR^d)$, with $2q+2 < \ell   $,   $L^{(1),+}_{(f)}=\max\{L^{(1)}_{(f)},0\} $ and $\rho_2= 2L_{(b)}^{(2)}+ 4L_{(b)}^{(3)}+2 L^{(1)}_{(u\sigma)} +4 L^{(2)}_{(u\sigma)} + 4L^{(1),+}_{(f)}$,
    $t\in [0,T],~T<\infty$. Assume $\rho_{1,\ell}<0$,
    we have
    \begin{align}
    \label{eq:aux.SemigroupDifferenceEstimate}
    \big(W^{(2)}(P^*_{0,t} \mu,P^*_{0,t}\nu)\big)^2 
    \le    
   3e^{\rho_2 t} \big(W^{(2)}(\mu,\nu)\big)^2 .   
    \end{align}

    \item  
    Assume further that the functions $b,\sigma$ are independent of time and that $\rho_2  , \rho_{1,2\ell-2}< 0$ with $1+m/2\geq\ell>2q+2$. Then \eqref{eq:aux.SemigroupDifferenceEstimate} yields exponential contraction, \eqref{eq:Bounded Orbit} yields bounded orbits, and there exists a unique invariant measure $\bar \mu \in \cP_{\ell}(\bR^d)$,  such that, for any $t>0$ and $\nu_0\in \cP_{2\ell -2}(\bR^d)$ we have 
    \begin{align*}
       W^{(2)}(P^*_{0,t} \bar \mu ,\bar \mu)
       =
       0
       \qquad \textrm{and}\qquad 
        W^{(2)}(P^*_{0,t} \nu_0 ,\bar \mu)
        \leq 
        e^{ \rho_2 t/2}  W^{(2)}\big(\nu_0, \bar \mu \big)
       .
    \end{align*}
In fact, for any $\nu_0\in \cP_{2\ell -2}(\bR^d)$ we have $\lim_{t\to \infty} W^{(\ell)}(P_{0,t}^* \nu_0,\bar \mu )=0$.
\end{enumerate}

\end{theorem}

The proof is postponed to Section \ref{section: proof of ergodicity}. A quick inspection shows that statement 1 and 2 only need $\ell> 2q+2$.  
Strictly speaking, the requirement for the initial distribution $\nu_0\in \cP_{2\ell -2}(\bR^d)$ is only needed for the final statement. The mechanism of choice for the proof is inspired by \cite[Theorem 3.1]{Wang2018DDSDE-LandauType}. In essence, \eqref{eq:Bounded Orbit} can be interpreted as the existence of a `non-expanding orbit', i.e., there is a `bounded orbit' and the exponential contractivity of the Wasserstein metric (under $\rho_2,\rho_{1,2\ell-2}<0$) in \eqref{Eq:MV-SDE Propagation} yields that all orbits are bounded -- for further considerations see \cite{liu2022ergodicity}. 



\subsection[C-stability and B-consistency for the particle system]{$C$-stability and $B$-consistency for the particle system}
\label{sec:C-B}
Before introducing our numerical scheme and the corresponding strong convergence result, we first present a definition of $C$-stability and $B$-consistency for the particle system. The following definitions and methodologies are modifications of the original work in \cite{2015ssmBandC} tailored to the present particle system setting. The probability space in this section supports (at least) the $N$ driving Brownian motions of the particle system and the filtration corresponds to the enlarged filtration generated by all Brownian motions augmented by a rich enough $\sigma$-algebra $\mathcal{F}_0$. 
\begin{definition}
\label{def:bc:def1:scheme}
Let $h \in(0, T]$ be  the stepsize and $\Psi_i: \mathbb{R}^{d} \times \cP_2(\bR^d) \times [0,T] \times \Omega \rightarrow \mathbb{R}^{d}$ 
for all $i\in \llbracket 1,N \rrbracket$ be a mapping satisfying the following measurability and integrability condition: For every $t,t+h \in [0,T],~h\in(0,1)$ and $X^N=(X^1,\ldots,X^N) \in L^{2}\left(\Omega, \mathcal{F}_{t}, \bP ; \mathbb{R}^{Nd}\right)$, $\mu \in \cP_2(\bR^{d})$
it holds
\begin{align}
\label{eq: def psi is L2}
\Psi_i(X^i,\mu, t, h) \in L^{2}\big(\Omega, \mathcal{F}_{t+h}, \bP  ; \mathbb{R}^{d}\big),\quad 
\Psi=(\Psi_1,\ldots,\Psi_N).
\end{align}
Then, for $M \in \mathbb{N}, M h = T,~k\in \llbracket 0,M-1 \rrbracket $, $t_k=kh$, 
\color{black}
we say that a particle system $\hx_{t_k}^{N}=(\hx_{t_k}^{1,N},\ldots,\hx_{t_k}^{N,N})\in \mathbb{R}^{Nd}  $ 
  is generated by the stochastic one-step method $(\Psi, h, \xi)$ with initial condition $\xi=(\xi^1,\ldots,\xi^N) \in L^{2}\left(\Omega, \mathcal{F}_{0}, \bP  ; \mathbb{R}^{Nd}\right)$, $\Psi=(\Psi_1,\ldots,\Psi_N) $, if 
\begin{align*}
&\hx_{k+1}^{i,N} =\Psi_i (\hx_{k}^{i,N}, \hm_{k}^{X,N}, t_{k}, h  ), 
\quad
\hm_k^{X,N} (\dd x)=\frac{1}{N} \sum_{j=1}^N \delta_{\hx_{k}^{j,N}}(\dd x),
\\
&\hx_{0}^{i,N}=\xi^i,\quad i\in \llbracket 1,N \rrbracket, 
\end{align*}
where $ \hx_{k}:=\hx_{t_k}$ and $\hm_{t_k}^{X,N}:=\hm_{k}^{X,N}$. 
\color{black} 
We call $\Psi$ the one-step map of the method.  
\end{definition}

\begin{definition}
\label{def:bc:def2:C-stable}
A stochastic one-step method $(\Psi, h, \xi)$ is called stochastically $C$-stable if there exists a constant $C>0$ and a parameter $\eta \in(1, \infty)$  
such that for all $t,t+h\in [0,T],~h>0$ and 
\color{black}
all random variables $X_t^{i,N}, Z_t^{i,N} \in L^{2}\left(\Omega, \mathcal{F}_{t}, \bP  ; \mathbb{R}^{d}\right),~ i\in \llbracket 1,N \rrbracket $ 
 -- the components of \textit{identically distributed} particle systems $X^N_t, Z^N_t\in\bR^{dN}$ (i.e., each particle system is exchangeable) with their empirical measures $\mu_t^{X,N},~\mu_t^{Z,N}\in \cP_2(\bR^d)$ -- satisfying that the pairs $( X_t^{i,N}, Z_t^{i,N})_i$ are identically distributed over $i$, it holds
\color{black}  
\begin{align*}
&\bE \Big[    \left|\mathbb{E}\big[\Psi_{i}(X_t^{i,N},\mu_t^{X,N}, t, h)
-\Psi_{i}(Z_t^{i,N},\mu_t^{Z,N}, t, h) \mid \mathcal{F}_{t}\big]\right|^2 \Big]
\\
& \qquad +
 \eta \bE \Big[   
\left|\left(\mathrm{id}-\mathbb{E}\big[\cdot \mid \mathcal{F}_{t}\big]\right)
\big(\Psi_{i}(X_t^{i,N},\mu_t^{X,N}, t, h)-\Psi_{i}(Z_t^{i,N},\mu_t^{Z,N}, t, h)\big)\right|^{2}  \Big]
\\
&
\qquad \qquad
\qquad \qquad
\leq
\left(1+C  h\right)
\bE \big[    |X_t^{i,N}-Z_t^{i,N}|^{2}\big]
+
Ch \big( W^{(2)}(\mu_t^{X,N},\mu_t^{Z,N} )\big)^2 .
\end{align*} 
\end{definition}
Here, and in what follows we denote by $\left(\mathrm{id}-\mathbb{E}\left[\cdot \mid \mathcal{F}_{t}\right]\right) Y=Y-\mathbb{E}\left[Y \mid \mathcal{F}_{t}\right]$ the projection of an $\mathcal{F}_{t+h}$-measurable random variables $Y$ orthogonal to the conditional expectation $\mathbb{E}\left[\cdot \mid \mathcal{F}_{t}\right]$.
\begin{definition}
\label{def:bc:def3:B-consist}
Let $X^{i,N},~i\in \llbracket 1,N \rrbracket$, be the unique strong solution to \eqref{Eq:MV-SDE Propagation},
with $\mu^{X,N}$ being the corresponding empirical measure. A stochastic one-step method $(\Psi, h, \xi)$ is called stochastically $B$-consistent of order $\gamma>0$ if there exists a constant $C>0$ such that  for all $t,t+h\in [0,T],~h\in(0,1)$, it holds
\begin{align*}{
\bE \Big[    \left|\mathbb{E}\big[X_{t+h}^{i,N}-\Psi_i(X_t^{i,N},\mu^{X,N}_{t}, t, h) \mid \mathcal{F}_{t}\big]\right|^2\Big]}
&\leq C  h^{2\gamma+2},
\\{
\bE \Big[    
\left|\left(\mathrm{id}-\mathbb{E}\big[\cdot \mid \mathcal{F}_{t}\big]\right)\big(X_{t+h}^{i,N}-\Psi_i(X_t^{i,N},\mu^{X,N}_{t}, t, h)\big)\right|^2 \Big]}
&\leq C  h^{2\gamma+1}.
\end{align*} 

\end{definition}

Next, we show the convergence results based on the definitions above.
\begin{lemma}
\label{lemma:bc: diff leq sum} 
Let $(\Psi, h, \xi)$ be a stochastically $C$-stable one-step method with some $\eta \in(1, \infty)$.  For the particle system with components $X^{i,N}$, given by \eqref{Eq:MV-SDE Propagation} with its empirical distribution $\mu^{X,N}$,  we have
\begin{align*}
& \sup_{n\in \llbracket 0,M \rrbracket}\sup_{i\in \llbracket 1,N \rrbracket}  \bE \big[   
 |X^{i,N}_n -\hx^{i,N}_n  |^{2} \big] 
 \leq
 e^{CT}\bigg[  
 \color{black}
 \sup_{i\in \llbracket 1,N \rrbracket}
\bE \big[   |X^{i,N}_0-\xi^{i}| ^{2} \big]
\color{black}
\\
&\quad+\sum_{k=1}^{M}\sup_{i\in \llbracket 1,N \rrbracket}  \bigg(
 (1+h^{-1} )
\bE \Big[    
\big|\mathbb{E}\big[X^{i,N}_k-\Psi_i (X^{i,N}_{k-1},\mu^{X,N}_{k-1}, t_{k-1}, h  ) \mid \mathcal{F}_{t_{k-1}}\big]\big|^{2} \Big]
\\
&\quad\qquad\qquad+C_{\eta} ~
\bE \Big[   
\Big|\left(\mathrm{id}-\mathbb{E}\big[\cdot \mid \mathcal{F}_{t_{k-1}}\big]\right)
\big(X^{i,N}_k-\Psi_i (X^{i,N}_{k-1},\mu^{X,N}_{k-1}, t_{k-1}, h  ) \big)\Big|^{2}  
\Big]\bigg) \bigg], 
\end{align*}
where $C_{\eta}=1+(\eta-1)^{-1}$ and $\hx^{i,N}_n $ denotes the particles generated by $(\Psi, h, \xi)$, with  $ X^{i,N}_k=X^{i,N}_{t_k}$, $\mu^{X,N}_k=\mu^{X,N}_{t_k}$, $t_k=kh$ for all $k \in \llbracket 0,M \rrbracket$.
\end{lemma}

\begin{theorem}
\label{theorem:bc: convergence rate }
Let the stochastic one-step method $(\Psi, h, \xi)$ be stochastically $C$-stable and stochastically $B$-consistent of order $\gamma>0$. If $\xi^{i}=X_{0}^{i,N}=\hx_{0}^{i,N}$, then there exists a constant $C$  independent of $N,h$ such that 
\begin{align*}    	
\sup_{n\in \llbracket 0,M \rrbracket}    \sup_{i\in \llbracket 1,N \rrbracket}   
\bE \big[   
 | X_{n}^{i,N}-\hx_{n}^{i,N} |^2 \big]  
\leq Ch^{2\gamma},
\end{align*}
where  $X^{i,N}$ denotes the exact solution to \eqref{Eq:MV-SDE Propagation} and $\hx^{i,N}$ is the particle generated by $(\Psi, h, \xi)$.
In particular, $(\Psi, h, \xi)$ is strongly convergent of order $\gamma$.

\end{theorem}

\subsection{The numerical scheme}
The split-step method (SSM) proposed here follows the steps of \cite{chen2022SuperMeasure} and is re-cast accordingly. The critical difficulty arises from the simultaneous appearance of the convolution component in $v$  \eqref{Eq:General MVSDE} and the super-linear diffusion coefficient. The presence of both nonlinearities is the main hindrance to proving moment bounds of order $p >2$ for the numerical scheme. Therefore, we rely on the $C$-stability and $B$-consistency methodology, as this approach does not require proving moment stability of higher order for the numerical scheme. This is in stark contrast to the techniques used in \cite{chen2022SuperMeasure}, where the time-stepping scheme has stable moments of higher order (depending on the regularity of the initial data) and strong convergence rates are proven without employing the $C$-stability and $B$-consistency procedure. Here, we wish to emphasize that even with the symmetry condition it is unclear how to prove $L^p$-moment bounds of the numerical scheme for $p >2$.

\begin{definition}[Definition of the SSM]
\label{def:definition of the ssm}
Let Assumption \ref{Ass:Monotone Assumption} hold, let $h$ satisfy \eqref{eq:h choice} and let
 $M\in \bN$   such that $Mh=T$.
Define recursively the SSM approximating of \eqref{Eq:MV-SDE Propagation} as: set $\hx_{0}^{i,N}=X^i_0$, for $i\in \llbracket 1,N\rrbracket $; for  $n\in \llbracket 0,M-1\rrbracket$ and $i\in \llbracket 1,N\rrbracket $ (recall Lemma \ref{remark:OSL for the whole function / system V}), $t_n=nh$, we have with $\Delta W_{n}^i=W_{t_{n+1}}^i-W_{t_n}^i$, {\color{black}  and $V$ defined in Lemma \ref{remark:OSL for the whole function / system V},}
\begin{align}
& Y_{n}^{\star,N} =\hat{ X}_{n}^{N}+h V (Y_{n}^{\star,N} ),
\quad  \hat{ X}_{n}^{N}=(\ldots,\hx_{n}^{i,N},\ldots),\quad Y_{n}^{\star,N}=(\ldots,Y_{n}^{i,\star,N},\ldots),
\label{eq:SSTM:scheme 0}
\\
\label{eq:SSTM:scheme 1}
&\textrm{where }~Y_{n}^{i,\star,N} =\hx_{n}^{i,N}+h v (Y_{n}^{i,\star,N},\hm^{Y,N}_n ),  
 \quad 
 \quad 
  \hm^{Y,N}_n(\dd x):= \frac1N \sum_{j=1}^N \delta_{Y_{n}^{j,\star,N}}(\dd x),
 \\
\label{eq:SSTM:scheme 2}
& \hx_{n+1}^{i,N} =Y_{n}^{i,\star,N}
            + b(t_n,Y_{n}^{i,\star,N},\hm^{Y,N}_n) h
            +\hs(t_n,Y_{n}^{i,\star,N},\hm^{Y,N}_n) \Delta W_{n}^i.
\end{align}
The stepsize $h$ satisfies (this constraint is soft, see \cite[Remark 2.7]{chen2022SuperMeasure} for details)
\begin{align}
\label{eq:h choice}
h\in \Big(0, \min\big\{1,\tfrac 1\zeta\big\} \Big)
~  \textrm{where}~
\zeta= \max\Big\{ 2(L^{(1)}_{(f)}+L^{(1)}_{(u\sigma)}),~2(2L^{(1),+}_{(f)}+L^{(1)}_{(u\sigma)}+ L^{(2)}_{(u\sigma)}),~0 \Big\}.
\end{align}
\end{definition}
It is immediate to see that \eqref{eq:SSTM:scheme 0} or \eqref{eq:SSTM:scheme 1} are implicit equations (given $\hx_{n}^{N}$). The solvability of $Y_{n}^{\star,N}$ as a unique implicit map of the input $\hx_{n}^{N}$ is addressed in Remark \ref{rem:SolvabilityImplicitEquation} below.  The choice of $h$ is discussed next. 
\begin{remark}[Choice of $h$]
\label{remark: choice of h}
Let Assumption \ref{Ass:Monotone Assumption} hold (the constraint on $h$ in \eqref{eq:h choice} comes from \eqref{eq:c-stable:y-y: u}, 
\eqref{eq:se1:y-y}, \eqref{eq:prop:yi-yj leq xi-xj}  and \eqref{eq:sum y square leq sum x square} below) and following the notation of these inequalities, under \eqref{eq:h choice} with $\zeta>0$, there exists $\lambda \in(0,1)$  such that $h< {\lambda}/{\zeta}$ and 
\begin{align*}
   &\max \bigg\{ \frac{1}{1-2(L^{(1)}_{(f)}+L^{(1)}_{(u\sigma)})h},~\frac{1}{1-2( 2L^{(1),+}_{(f)}+  L^{(1)}_{(u\sigma)}+L^{(2)}_{(u\sigma)})h} 
   \bigg\}< \frac{1}{1-\lambda}.
\end{align*}
For $\zeta=0$, the result is trivial and we conclude that there exists a constant $C$ independent of $h$ such that
\begin{align*}
  \max \bigg\{ \frac{1}{1-2(L^{(1)}_{(f)}+L^{(1)}_{(u\sigma)})h},~\frac{1}{1-2(2  L^{(1),+}_{(f)}+  L^{(1)}_{(u\sigma)}+ L^{(2)}_{(u\sigma)} )h} 
  \bigg\} \le 1+Ch.
\end{align*}
As argued in \cite[Remark 2.7]{chen2022SuperMeasure},  the constraint on $h$ \textit{can be lifted}. 
\end{remark}

\begin{remark}[Solvability of the implicit equation \eqref{eq:SSTM:scheme 1}]
\label{rem:SolvabilityImplicitEquation}
Recall that the function $V$ (defined in Lemma \ref{remark:OSL for the whole function / system V}) satisfies a one-sided Lipschitz condition in $\bR^{Nd}$, and hence (under \eqref{eq:h choice}) a unique solution $Y_{n}^{\star,N}$ to \eqref{eq:SSTM:scheme 0} as a function of $\hat{ X}_{n}^{N}$ exists. 
\color{black}
This result follows from a well-known argument using results on strongly monotone operators \cite[Theorem 26.A (p.557)]{Zeidler1990B} and is the same case in   \cite[Lemma 4.2]{chen2022SuperMeasure} (where we do not have a measure component for $V$ but view it as a mapping from $\bR^{Nd} \to \bR^{Nd}$ instead), the detailed argument (for mappings from  $\bR^{d} \to \bR^{d}$) is shown in \cite[Lemma 4.1]{2021SSM}.
\color{black}
\end{remark}

After introducing the discrete scheme, we discuss its continuous extension and the main convergence results.  
\begin{definition} [Continuous extension of the SSM]
Under the same choice of $h$ and assumptions in Definition \ref{def:definition of the ssm}, for all $t\in[t_n,\tnp]$, $n\in \llbracket 0,M-1\rrbracket$, $t_n=nh$,   $i\in\llbracket 1,N \rrbracket$,  
$\hx_{0}^{i,N}=X^i_0$, for $X^i_0$  in \eqref{Eq:MV-SDE Propagation}, the continuous extension of the SSM is 
\begin{align*} 
    \dd \hat{X}_{t}^{i,N}
    &
    =
    \big( v(Y_{\kappa(t)}^{i,\star,N},\hm^{Y,N}_{\kappa(t)})
    +b(\kappa(t),Y_{\kappa(t)}^{i,\star,N},\hm^{Y,N}_{\kappa(t)})  \big) \dd t 
    + \hs (\kappa(t),Y_{\kappa(t)}^{i,\star,N},\hm^{Y,N}_{\kappa(t)} ) \dd W_t^i,
    \\
    \nonumber
    & \textrm{where }~ 
    \quad 
   \hm^{Y,N}_n(\dd x):= \frac1N \sum_{j=1}^N \delta_{Y_{n}^{j,\star,N}}(\dd x),\qquad \hm^{Y,N}_{t_n}=\hm^{Y,N}_n, 
\end{align*}
and $\kappa(t)=\sup\big\{t_n: t_n\le t,\ n\in \llbracket 0,M-1 \rrbracket \big\}$.
\end{definition}
\color{black} 
\begin{theorem} [Convergence of the SSM]
\label{theorem:SSM: convergence all}
Let Assumption \ref{Ass:Monotone Assumption} hold for some  $m > 4q+4 >\max\{2(q+1),4\}$.
Choose $h$ as in \eqref{eq:h choice}. Then for the SSM scheme defined in \eqref{eq:SSTM:scheme 0}-\eqref{eq:SSTM:scheme 2}, we have the following properties.
\begin{enumerate}
    \item  The SSM is $C$-stable;
    \item  The SSM is $B$-consistent with $\gamma=1/2$ in Definition \ref{def:bc:def3:B-consist};
    \item  For $i\in\llbracket 1,N\rrbracket$, let $X^{i,N}$ be the solution to \eqref{Eq:MV-SDE Propagation}, 
    then there exists a constant $C>0$ (independent of $N$ and $h$) such that 
\begin{align*} 
     \sup_{i\in \llbracket 1,N \rrbracket} \sup_{t\in [0,T]}
  \bE\big[\,  |X_{t}^{i,N}-\hx_{t}^{i,N} |^2 \big]    &   \le Ch.
\end{align*}
    
\end{enumerate}

\end{theorem}

Lastly, we present a result about long time stability  of the numerical scheme proposed as means to access the invariant distribution of the original MV-SDE by way of simulation. In other words, we provide sufficient conditions for our scheme to be \textit{mean-square contractive} as $T\to \infty$ in the sense of \cite[Definition 2.8]{2021SSM}.
\begin{theorem}
\label{theo:SSTM:stabilty}
Let the Assumptions of Theorem \ref{theorem:SSM: convergence all} and Theorem \ref{Thm:ergodicity}  hold. 
Suppose that $X_0\in L^m(\mathbb{R}^d)\nonumber$ and $Z_0\in L^m(\mathbb{R}^d)\nonumber$ for $m>4q+4$ as in Theorem \ref{theorem:SSM: convergence all}, and let $\hx_0^{i,N}$ and  $\hz_0^{i,N}$ be i.i.d.~copies of $X_0$ and $Z_0$ respectively, for all $i\in \llbracket 1,N \rrbracket$. 

Set $h>0$. For $i\in \llbracket 1,N \rrbracket$ and $n \in \llbracket 1,M \rrbracket$, define $(\hx_n^{i,N},Y^{i,X,N}_n)$ and $(\hz_n^{i,N},Y^{i,Z,N}_n)$ as the output of the SSM \eqref{eq:SSTM:scheme 1}-\eqref{eq:SSTM:scheme 2} (i.e., $\star= X,Z$)
corresponding to the empirical measure pairs $(\hm^{X,N}_n,\hm^{Y,X,N}_n)$ and $(\hm^{Z,N}_n,\hm^{Y,Z,N}_n)$
with initial conditions $X_0^{i,N}$ and $Z_0^{i,N}$ respectively. Then, for any $n \in \llbracket 1,M \rrbracket$, 
\begin{align*} 
\sup_{ i\in \llbracket 1,N \rrbracket}  \bE\big[\, |\hx_{n}^{i,N}-\hz_{n}^{i,N} |^2 \big]
\le    
 (1+\beta h )^n \sup_{ i\in \llbracket 1,N \rrbracket}\bE\big[ |\hx^{i,N}_{0}-\hz^{i,N}_{0} |^2\big],
\end{align*}
where we recall the parameters of Theorem \ref{Thm:ergodicity}, 
\begin{align*} 
\beta=
\frac{\rho_2+2L_{(b)}^{(1)} h }{1- h(4L^{(1),+}_{(f)}+2L^{(1)}_{(u\sigma)} +2L^{(2)}_{(u\sigma)  })},
\quad 
\rho_2= 4L^{(1),+}_{(f)}+2L^{(1)}_{(u\sigma)} +2L^{(2)}_{(u\sigma)}+2L_{(b)}^{(2)}+2L_{(b)}^{(3)}. 
\end{align*}
Under the choice of $h$ stated in Theorem \ref{theorem:SSM: convergence all}, the quantity $1+\beta h$ is always positive. 
If $\rho_2 < 0$ and $h$ sufficiently small then $\beta<0$ and thus the SSM is \textit{mean-square contractive} in the sense of \cite[Definition 2.8]{2021SSM}.
\end{theorem}

 

\section{Examples of interest}
\label{sec:examples}
 We illustrate the performance of the SSM on several numerical examples. 
 As the ``true'' solution of the considered models is unknown, the convergence rates for these examples are calculated in reference to a proxy solution given by an approximation at a smaller timestep $h$. The strong error between the proxy-true solution $X_T$ and approximation $\hat X_T$ is as follows
\begin{align*}
\textrm{root Mean-square error (rMSE)} 
= \Big( \bE\big[\, |X_T-\hat{X}_T|^2\big]\Big)^{\frac12}
\approx \Big(\frac1N \sum_{j=1}^N |X_T^j - \hat{X}_T^j|^2\Big)^\frac12.
\end{align*}
We also consider the path type strong error as follows 
\begin{align*}
\textrm{Strong error (path)} 
= \Big( \bE\big[\,\sup_{t\in[0,T]} |X_t-\hat{X}_t|^2\big]\Big)^{\frac12}
\approx \Big(\frac1N \sum_{j=1}^N \sup_{n\in \llbracket 0,M \rrbracket } |X_n^j - \hat{X}_n^j|^2\Big)^\frac12.
\end{align*} 
The propagation of chaos (PoC) rate between different particle systems $(\hx_T^{i,N_l})_{i,l}$ where $i$ denotes the $i$-th particle and $N_l$ denotes the size of the system, is measured by 
\begin{align}
\label{aux_PoC_rate_Estimator}
\textrm{Propagation of chaos Error (PoC-Error)} 
\approx \Big(\frac{1}{N_l} \sum_{j=1}^{N_l} |\hx_T^{j,N_l} - \hx_T^{j,N_{l+1}} |^2\Big)^\frac12.
\end{align}
Above $N_{l+1}=2N_l$ and the first half of the $N_{l+1}$ particles use the same Brownian motions as the whole $N_l$ particle system. In this section, the rMSE takes $h\in\{10^{-1},5\times10^{-2},2\times10^{-2},10^{-2},5\times10^{-3},2\times10^{-3},10^{-3}\}$ with $N=1000$, the proxy solution takes $h=10^{-4}$. The PoC takes $N\in \{40,80,160,320,640,1280\}$ with $h=10^{-3}$, the proxy solution takes $N=2560$.

\begin{remark}[`Taming' algorithm]
For comparative purposes, we implement the `\textit{Taming}' algorithm \cite{2021SSM,reis2018simulation} -- any convergence analysis of the taming algorithm in the framework of this manuscript is an open question. 
Of the many possible taming variants, we implement the following two cases: taming $f$ (and similarly $f_{\sigma}$) inside the convolution term (`Taming-in') and taming the convolution itself (`Taming-out'). Concretely, set $Mh=T$, then $f$ is replaced by (for $\alpha \in (0,1]$)
\begin{itemize}
    \item `Taming-out': $\int_{\bR^{d}  } f(\cdot-y) \mu(\dd y)$ is replaced by 
        $\int_{\bR^{d}  } f(\cdot-y) \mu(\dd y)/\big(~1+M^\alpha|\int_{\bR^{d}  } f(\cdot-y) \mu(\dd y)|~\big)$. 
    \item `Taming-in': $f$ is replaced by $f/\big(1+M^\alpha|f|\big)$. 
\end{itemize}
\end{remark}
Note that the proxy solution for the SSM is computed using the SSM and analogously for the taming schemes. For each example, the error rates of Taming and SSM are computed using the same Brownian motion paths and same initial data. {To avoid confusion later in the numerical results, we clarify that due to the super-linear convolution kernel, we do not expect the Taming method to converge. However, under mild initial conditions, it is rare to observe the divergence, so we test high variance cases to show the Taming method does not work in general while the SSM works as expected.} We remark that the first step \eqref{eq:SSTM:scheme 0} of the SSM requires to solve an implicit equation in $\bR^{Nd}$, which is done employing Newton's method (see \cite[Appendix B]{chen2022SuperMeasure} for details).

Below, the symbols $\cN(\alpha,\beta)$ denote the normal distribution with mean $\alpha\in \bR$ and variance $\beta\in (0,\infty)$, the symbol $U(a,b)$ denotes the uniform distribution over $[a,b]$ for $-\infty<a<b<\infty$, the symbol $B(c,p)$ denotes the binomial distribution for random variables $X$ such that $X=0$ with probability $p$ and $X=c$ with probability $1-p$.

\subsection{Example: Symmetric double-well type model}
\label{section:example: dw1}
We consider an extension to the symmetric double-well model \cite{2013doublewell} of confinement type with extra super-linearity \cite[Section 5]{2019taming} in the diffusion coefficient,
\begin{align}
\label{eq:example:toy1}
\dd  X_{t} &= \big(  v(X_{t},\mu_{t}^{X})+X_{t} \big) \dd t + (X_t+\tfrac{1}{4}X_t^2) \dd W_{t},~ 
v(x,\mu)= -\tfrac{1}{4} x^3+\int_{\bR  } -\big(x-y  \big)^3  \mu(\dd y).
\end{align}
The corresponding Fokker-Planck equation is 
    $ \partial_t \rho=\nabla   [~\nabla  \frac{\rho}{2}|x+\tfrac{1}{4}x^2|^2+\rho \nabla V+\rho \nabla W * \rho ]$ 
with $W=\tfrac{1}{4}|x|^4$, $V=\tfrac{1}{16}|x|^4-\tfrac{1}{2}|x|^2$, and $\rho$ is the corresponding density map. Due to the structure of the drift term, we expect three cluster states around $x\in\{-2,0,2\}$.

\begin{figure}[h!bt]
    \centering
    \begin{subfigure}{.58\textwidth}
    \setlength{\abovecaptionskip}{-0.02cm}
     \setlength{\belowcaptionskip}{-0.05 cm}  
			\centering
 			\includegraphics[scale=0.51]{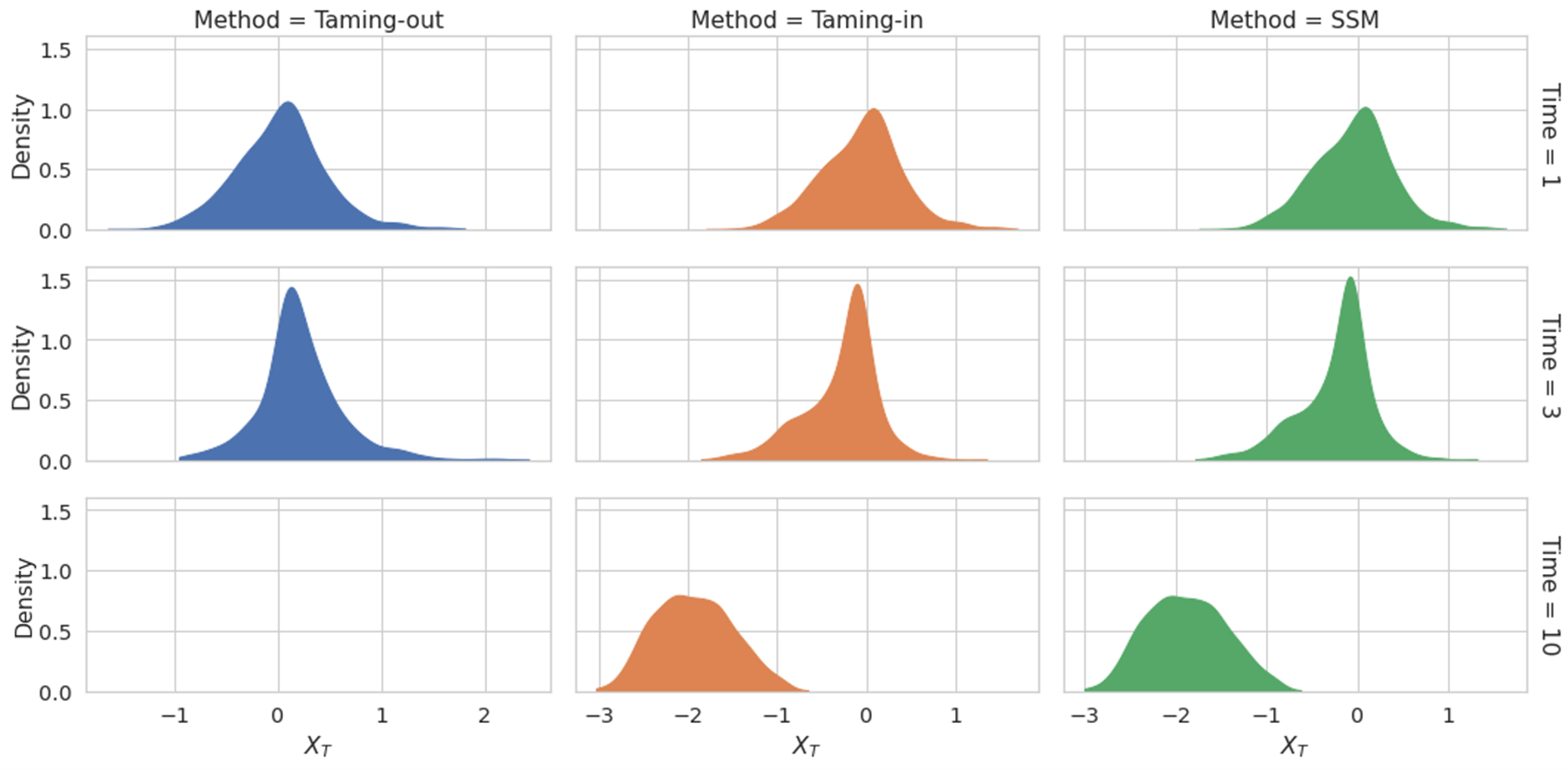}
			\caption{Density with $X_0\sim \cN(0,1)$}
		\end{subfigure}%
		\begin{subfigure}{.36\textwidth}          \setlength{\abovecaptionskip}{-0.02cm}
          \setlength{\belowcaptionskip}{-0.05 cm}
			\centering
 			\includegraphics[scale=0.26]{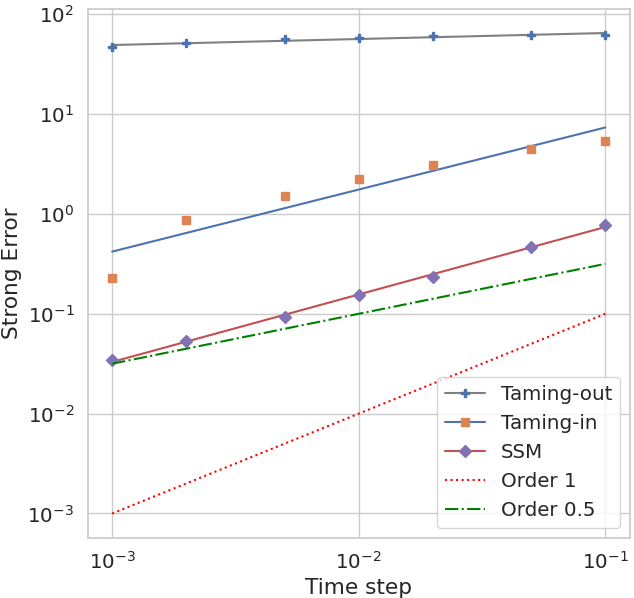}
			\caption{Strong error (rMSE) }
		\end{subfigure}	
		\\
    \centering
    \begin{subfigure}{.58\textwidth}
    \setlength{\abovecaptionskip}{-0.02cm}
     \setlength{\belowcaptionskip}{-0.05 cm}  
			\centering
 			\includegraphics[scale=0.51]{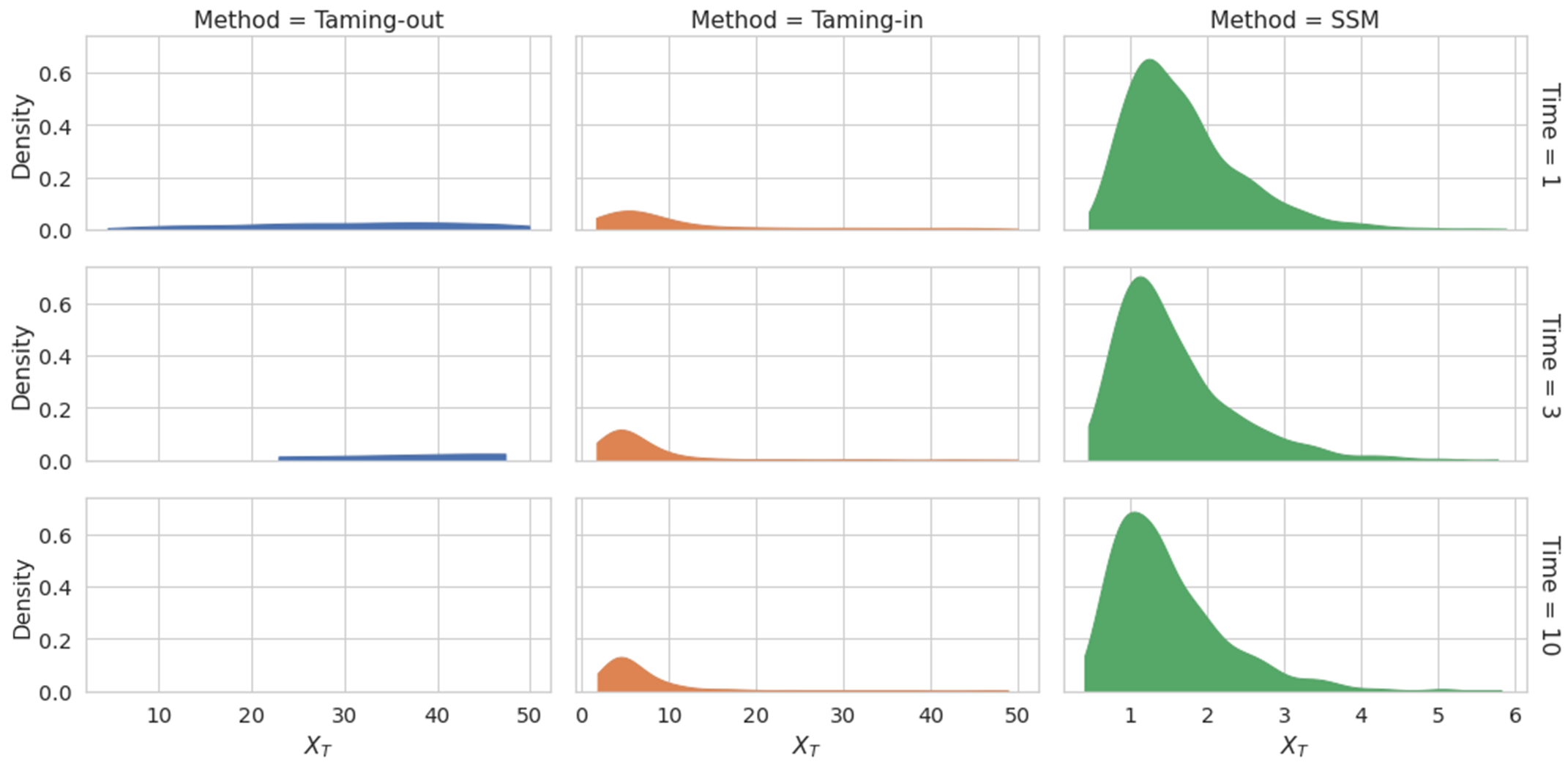}
			\caption{Density with $X_0\sim B(50,0.5)$}
		\end{subfigure}%
		\begin{subfigure}{.36\textwidth}          \setlength{\abovecaptionskip}{-0.02cm}
          \setlength{\belowcaptionskip}{-0.05 cm}
			\centering
 			\includegraphics[scale=0.26]{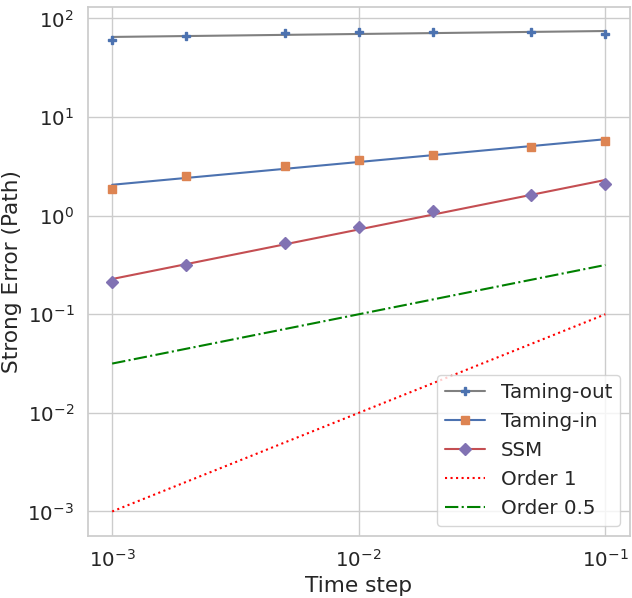}
			\caption{Strong error (Path) }
		\end{subfigure}	
    \caption{Simulation of the double-well model  \eqref{eq:example:toy1} with $N=1000$ particles. All schemes are initialized on the exact same samples. 
     (a)  and (c) show the density map for Taming-out (left),  Taming-in (middle) and SSM (right) with $h=0.01$  at times $T\in\{1,3,10\}$ seen top-to-bottom and with different initial distribution.  
     (b) Strong error (rMSE) of SSM and Taming with $X_0 \sim \cN(3,9)$ in log-scale. (d) Strong error (Path) of SSM and Taming with $X_0 \sim \cN(3,9)$ in log-scale. 
     }
    \label{fig:the toy example33}
\end{figure}

The goal of this example is to simulate the interacting particle system associated to \eqref{eq:example:toy1} up to $T=10$ using the three numerical methods available.  {\color{black} Note that Theorem \ref{Thm:ergodicity} does not apply here for the parameter choice in \eqref{eq:example:toy1}. }  
Figure \ref{fig:the toy example33} (a) and (c)  show the evolution of the density map at $T\in\{1,3,10\}$. In (a) with $X_0 \sim \cN(0,1)$, all three methods yield similar results,  but (c) shows that with $X_0  \sim B(50,0.5)$, Taming-out (blue, left) and  Taming-in fail to produce acceptable results, while the SSM produces the expected results.  
 
Figure \ref{fig:the toy example33} (b)  shows the strong convergence of the methods, Taming-out failed to converge. Taming-in and  the SSM converge under all time step choices (all satisfying \eqref{eq:h choice}) and nearly attain the $1/2$ strong error rate, the error of SSM is one order of magnitude smaller than the error of Taming-in. 
Figure \ref{fig:the toy example33} (d)  shows the path type strong convergence of both methods, and we observe that Taming-out and Taming-in failed to converge or at least converge with a very low rate. The SSM converges under all time step choices but the errors are one order of magnitude greater than the standard strong error.

As mentioned earlier, we do not have any theoretical support for the convergence of the taming methods. This example shows that a convergence proof for Taming-in might be feasible, possibly, under the caveat of an additional condition on the distribution/support of the initial condition -- this was fully unforeseen. These results for Taming-out are discouraging, nonetheless, under strong dissipativity Taming-out seems stable (see next example).

\subsection{Example: Approximating the invariant distribution}
This example aims to illustrate the long-time simulation for the purpose of approximating the invariant distribution of the system  
\begin{align}
\label{eq:example:toy222}
\dd  X_{t} &= \big(  v(X_{t},\mu_{t}^{X})-X_{t} \big) \dd t + \tfrac{1}{4}(1-X_t^2) \dd W_{t},~  
v(x,\mu)= - x^3+\int_{\bR  } -\big(x-y  \big)^3  \mu(\dd y).
\end{align}
The corresponding Fokker-Planck equation is 
    $ \partial_t \rho=\nabla   [~\nabla  \frac{\rho}{32}|1-x^2|^2+\rho \nabla V+\rho \nabla W * \rho ]$ 
with $W=\frac{1}{4}|x|^4$, $V=\frac{1}{4}|x|^4+\frac{1}{2}|x|^2$, and $\rho$ is the corresponding density map. We know that there is a unique invariant distribution, see Theorem \ref{Thm:ergodicity}. Here, the cluster state is $x =0$.  
\begin{figure}[h!bt]
    \centering
    \begin{subfigure}{.57\textwidth}
    \setlength{\abovecaptionskip}{-0.02cm}
     \setlength{\belowcaptionskip}{-0.05 cm}  
			\centering
 			\includegraphics[scale=0.23]{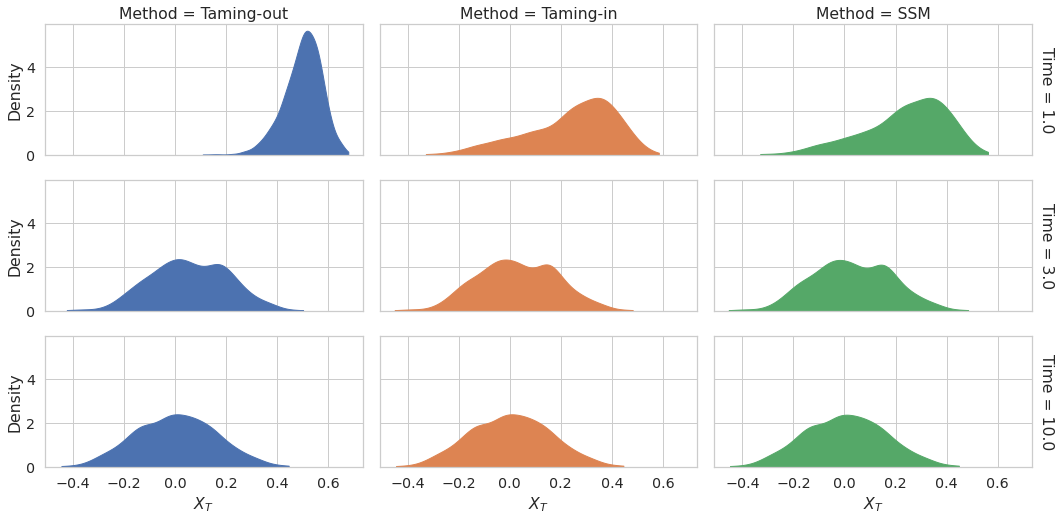}
			\caption{Density with $X_0\sim \cN(2,16)$}
		\end{subfigure}%
		\begin{subfigure}{.41\textwidth}
		\setlength{\abovecaptionskip}{-0.02cm}
		\setlength{\belowcaptionskip}{-0.05 cm} 
			\centering
 			\includegraphics[scale=0.27]{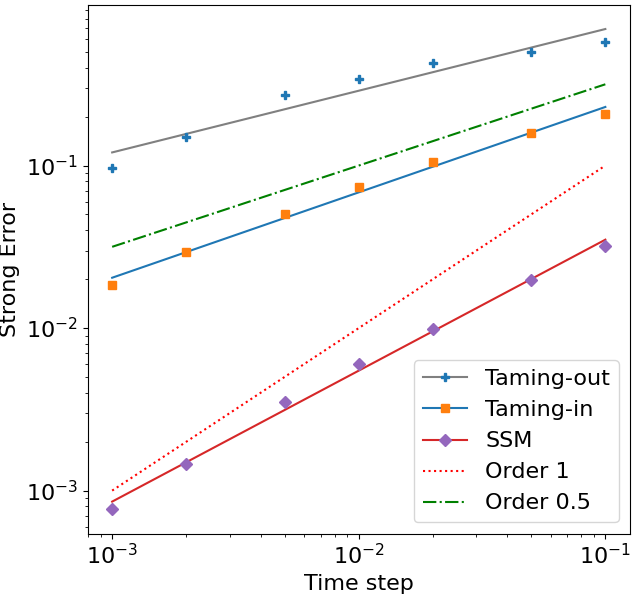}
			\caption{Strong error }
		\end{subfigure}
		\\
        \begin{subfigure}{.57\textwidth}
          \setlength{\abovecaptionskip}{-0.02cm}
          \setlength{\belowcaptionskip}{-0.2 cm}
			\centering
 			\includegraphics[scale=0.23]{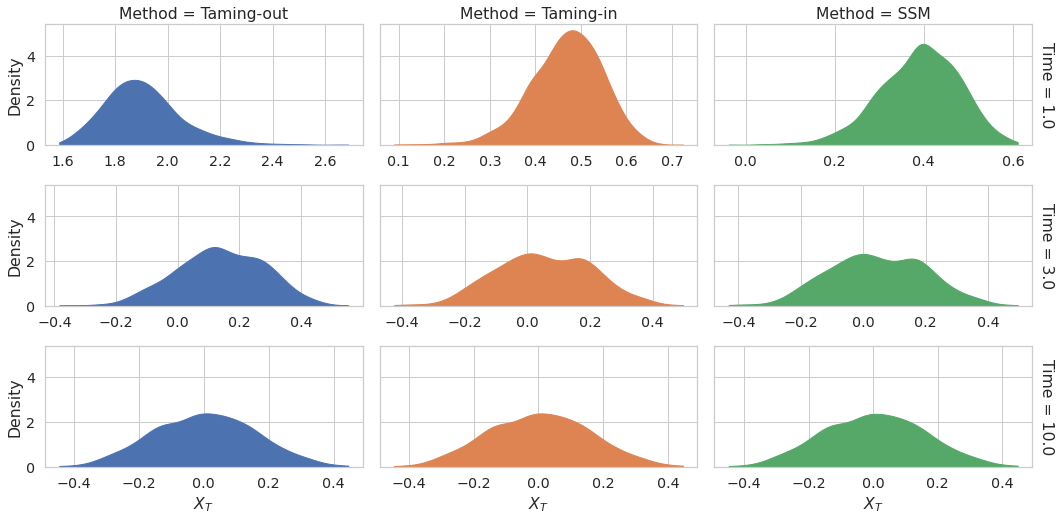}
			\caption{Density with $X_0\sim U(4,12)$}
		\end{subfigure}%
		\begin{subfigure}{.41\textwidth}
		\setlength{\abovecaptionskip}{-0.02cm}
		\setlength{\belowcaptionskip}{-0.2 cm}
			\centering
 			\includegraphics[scale=0.27]{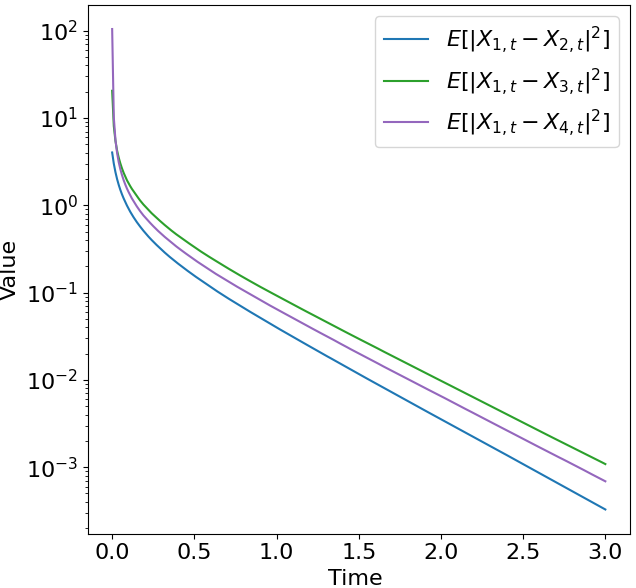}
			\caption{Expected particle distance (SSM)}
		\end{subfigure}
	\setlength{\belowcaptionskip}{-0.2cm} 
    \caption{Approximation of the invariant distribution of   \eqref{eq:example:toy222} with $N=1000$ particles. The simulated Brownian motion paths and initial distribution are the same for all schemes.  
     (a) and (c) show the distribution for Taming-out (left),  Taming-in (middle) and SSM (right) with $h=0.01$  at times $T\in\{1,3,10\}$ seen top-to-bottom and with different initial distribution; $x$- and $y$-scales are fixed.  
     (b) Strong error (rMSE) of SSM and Taming with $X_0 \sim \cN(2,16)$.  (d) Expected distance (in log-scale) between particles under different initial distributions with $h=10^{-3}$ for the SSM.  
     }
    \label{fig:3:the toy example3}
\end{figure}

Figure \ref{fig:3:the toy example3} (a) and (c) show the evolution of the particle distribution under different initial conditions. All three methods produce similar outputs at $T\in\{3,~ 10\}$, with Taming-out taking longer to contract and to converge than the other methods under $X_0 \sim \cN(2,16)$ in (a) and  $X_0 \sim U(4,12)$ in (c). The similar results obtained at $T\in\{3,10\}$ are due to the fact that the model \eqref{eq:example:toy222} has an invariant distribution and the initial distribution is compactly supported around the cluster state $x=0$.  

Figure \ref{fig:3:the toy example3} (b) illustrates the strong convergence of the three methods: they all converge and the rates are of order close to $1/2$, the SSM outperforms the other two methods by $1$ to $2$ orders of magnitude. 
Figure \ref{fig:3:the toy example3} (d) depicts the expected exponential decay rate for the SSM under different initial conditions of Theorem \ref{Thm:ergodicity}: $X_{1,0}\sim \cN(0,1)$, $X_{2,0}\sim U(-3,3)$, $X_{3,0}\sim \cN(2,16)$, $X_{4,0}\sim \cN(2,100)$ (same Brownian motion samples).

\subsection{Example: Kinetic 2d Van der Pol  oscillator and periodic phase-space}
We consider a two-dimensional Van der Pol (VdP) oscillator model with added  super-linearity terms.    
The VdP model was proposed to describe stable oscillation \cite[Section 4.2 and 4.3]{HutzenthalerJentzen2015AMSbook} and for a system of many coupled oscillators in the presence of noise the limit model is a MV-SDE \cite{Spigler2005KuramotoModel}. Here, we build a two-dimensional VdP-type model with mean-field components and super-diffusivity that features a periodicity of phase-space to show that the SSM preserves the theoretical periodic behaviour in simulation scenarios -- see \cite[Section 7.3]{buckwarbrehier2021FHNmodelandsplittingBSTT2020}. 

Set $x=(x_1,x_2)\in\bR^2$ and define the functions $f,u,b,\sigma$ as 
\begin{align}
\label{eq:example:vdp}
f(x)=-x|x|^2
,~
u(x)=
\left[\begin{array}{c}
-\frac13 x_1^3 \\
0
\end{array}\right],
~
 b(x)=
\left[\begin{array}{c}
x_1-x_2 \\
x_1
\end{array}\right],
~
\sigma (x)=
\left[\begin{array}{ccc}
1+1/4 ~x_1^2 & 0 \\
0 &  0
\end{array}\right],
\end{align}
where $f$ satisfies $ (\mathbf{A}^f)$.

\begin{figure}[h!bt]
	    \centering
		\begin{subfigure}{.19\textwidth} 
			\centering
 			\includegraphics[scale=0.24]{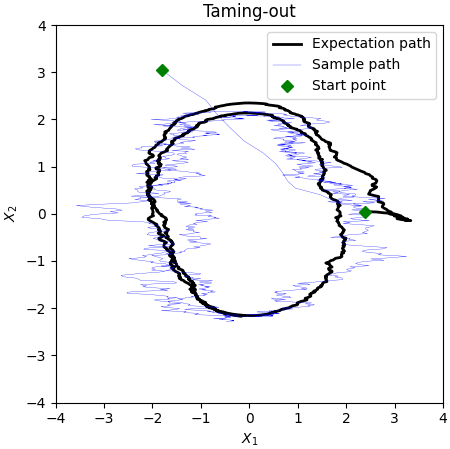}
			\caption{ $N=50$}
		\end{subfigure}
		\begin{subfigure}{.19\textwidth} 
			\centering
 			\includegraphics[scale=0.24]{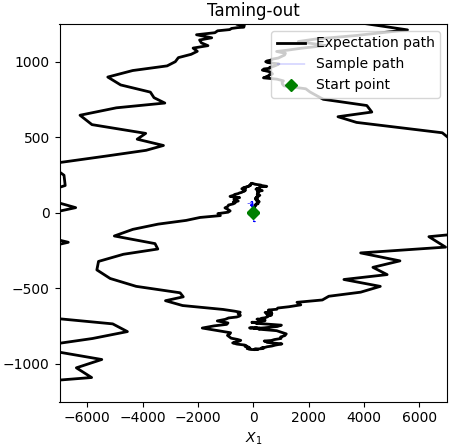}
			\caption{ $N=200$}
		\end{subfigure}
		\begin{subfigure}{.19\textwidth} 
			\centering
 			\includegraphics[scale=0.24]{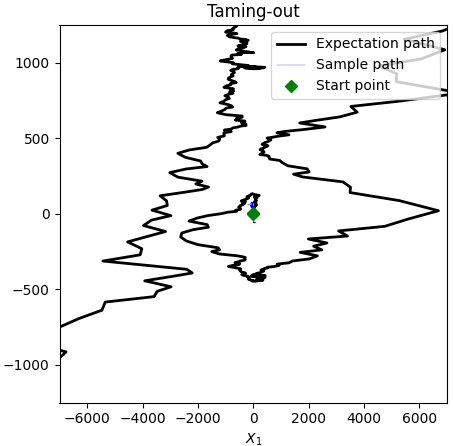}
			\caption{ $N=500$}
		\end{subfigure}
		\begin{subfigure}{.19\textwidth} 
			\centering
 			\includegraphics[scale=0.24]{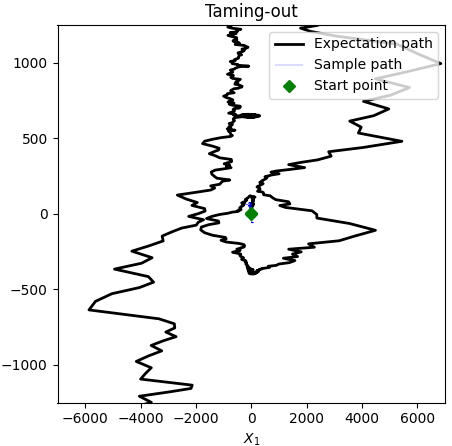}
			\caption{ $N=1000$}
		\end{subfigure}
		\begin{subfigure}{.19\textwidth} 
			\centering
 			\includegraphics[scale=0.24]{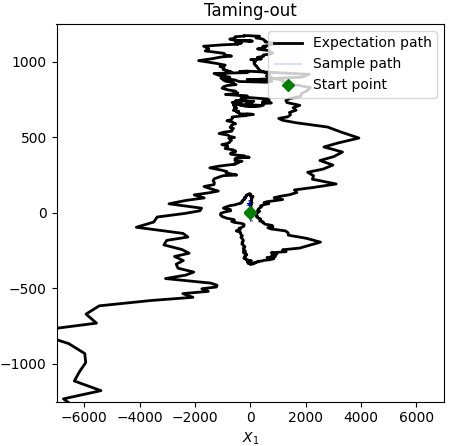}
			\caption{  $N=2000$}
		\end{subfigure}
	\\
	    \centering
		\begin{subfigure}{.19\textwidth} 
			\centering
 			\includegraphics[scale=0.24]{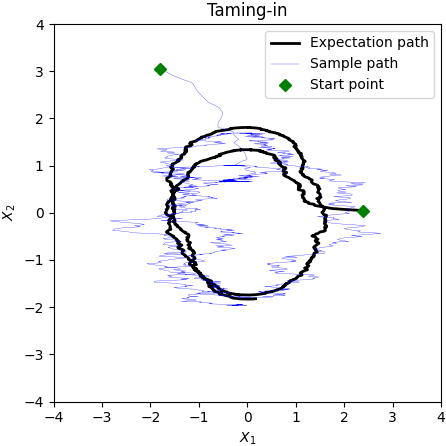}
			\caption{ $N=50$}
		\end{subfigure}
		\begin{subfigure}{.19\textwidth} 
			\centering
 			\includegraphics[scale=0.24]{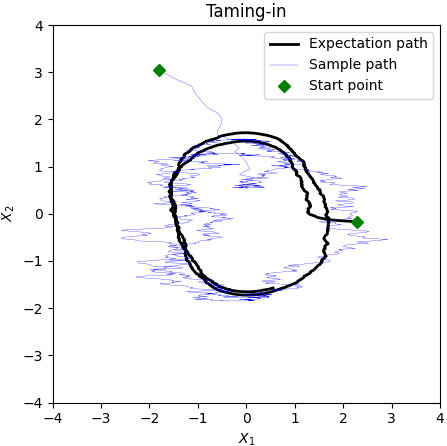}
			\caption{ $N=200$}
		\end{subfigure}
		\begin{subfigure}{.19\textwidth} 
			\centering
 			\includegraphics[scale=0.24]{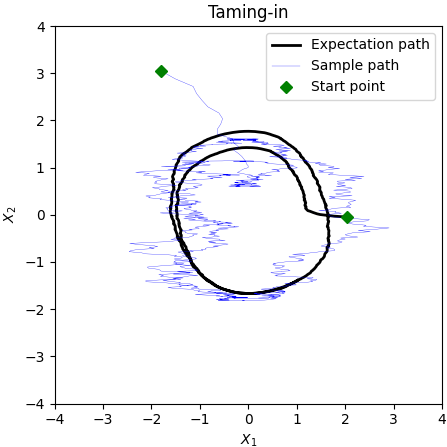}
			\caption{ $N=500$}
		\end{subfigure}
		\begin{subfigure}{.19\textwidth} 
			\centering
 			\includegraphics[scale=0.24]{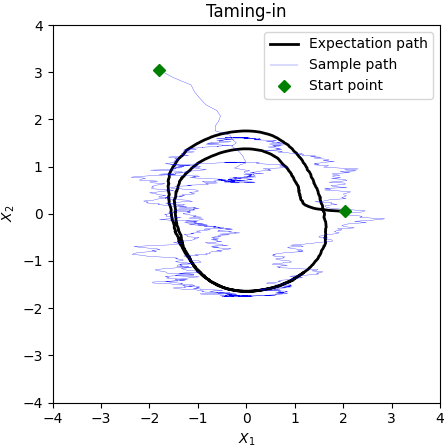}
			\caption{  $N=1000$}
		\end{subfigure}
		\begin{subfigure}{.19\textwidth} 
			\centering
 			\includegraphics[scale=0.24]{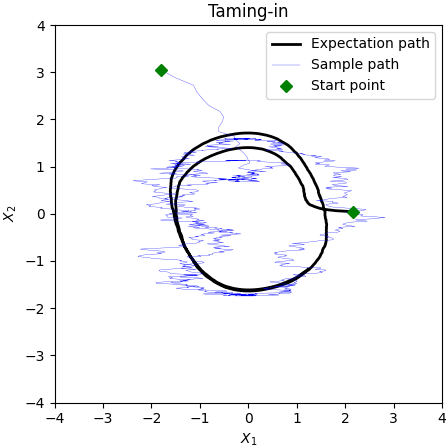}
			\caption{ $N=2000$}
		\end{subfigure}
		\\
		     \centering
		\begin{subfigure}{.19\textwidth} 
			\centering
 			\includegraphics[scale=0.24]{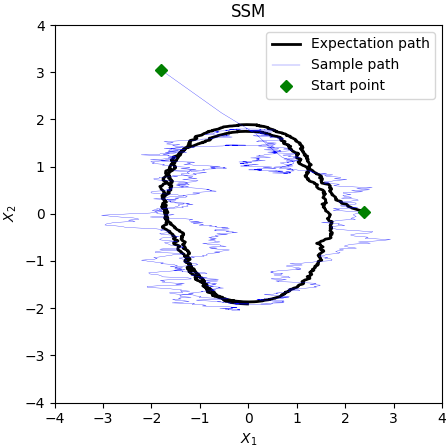}
			\caption{  $N=50$}
		\end{subfigure}
		\begin{subfigure}{.19\textwidth} 
			\centering
 			\includegraphics[scale=0.24]{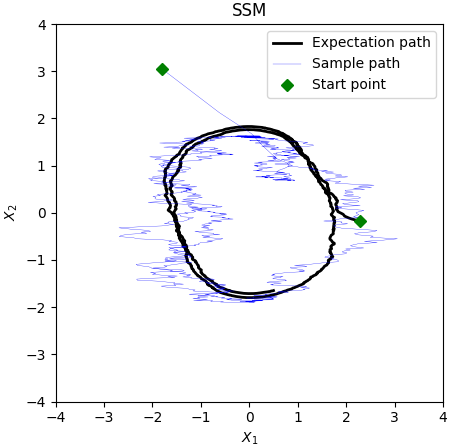}
			\caption{  $N=200$}
		\end{subfigure}
		\begin{subfigure}{.19\textwidth} 
			\centering
 			\includegraphics[scale=0.24]{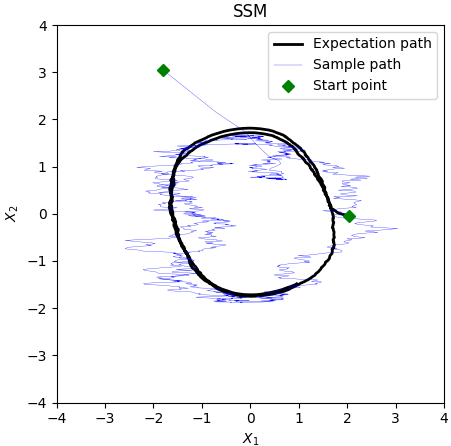}
			\caption{ $N=500$}
		\end{subfigure}
		\begin{subfigure}{.19\textwidth} 
			\centering
 			\includegraphics[scale=0.24]{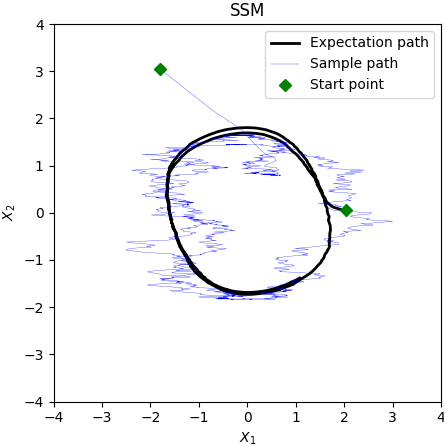}
			\caption{  $N=1000$}
		\end{subfigure}
		\begin{subfigure}{.19\textwidth} 
			\centering
 			\includegraphics[scale=0.24]{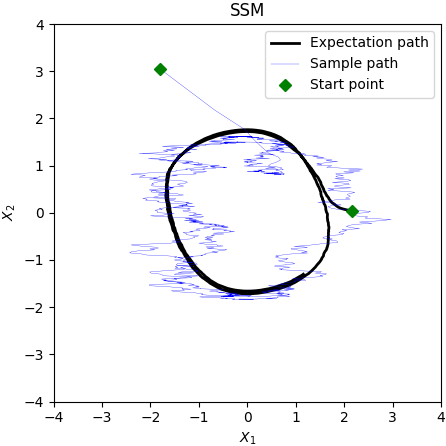}
			\caption{  $N=2000$}
		\end{subfigure}
    \caption{Simulation of the Vdp model \eqref{eq:example:vdp} with a different number of particles and $h=10^{-2}$, $T=12$, $X_{1,0}\sim \cN(2,16),X_{2,0}\sim \cN(0,16)$. 
     (a)(b)(c)(d)(e) are phase portraits of the Taming-out method with different choices of $N$.  (f)(g)(h)(i)(j) are phase portraits of the Taming-in method with different choices of $N$.  (k)(l)(m)(n)(o) are phase portraits of the SSM with different choices of $N$.  
     }
    \label{fig:4:multi d example}
\end{figure}

Figure \ref{fig:4:multi d example}   (a)-(o) show the system's phase-space portraits (i.e., the parametric plot of $t\mapsto (X_{1,t},X_{2,t})$ and $t\mapsto (\bE[X_{1,t}],\bE[X_{2,t}])$) for the three methods with different choices of $N$.

In the first row of Figure \ref{fig:4:multi d example}, (a)-(e) shows the result of the Taming-out method, the system fails to converge for $N>50$. 
The second row and third row of Figure \ref{fig:4:multi d example} show the result of Taming-in and the SSM, both methods converge and the  trajectory becomes smoother as more particles are taken. However, there is a big difference on the expectation trajectories of the SSM and Taming in, the expectation trajectories of the SSM do not cross themselves while the expectation trajectories of Taming-in always cross themselves, which is not expected since the slope fields  of the VdP model are smooth and do not admit the cross.   
Moreover, comparing the first few steps in the sample paths, the particles generated by the SSM concentrate to the expectation path within two steps  while the one generated by Taming-in takes about 10 steps. This is because the SSM preserves the super-linear power from the convolution kernel while the Taming-in turns this power to an asymptotic linear one. 
Thus, the SSM preserves more geometric properties than the taming method even though the approximation obtained via taming may not blow up.

\subsection{Example: Super-linear growth of measure components in diffusion}
This example illustrates the effect of two additional types of measure-nonlinearities included in the diffusion term; Case 1 corresponds to a convolution term in the diffusion and Case 2 is a variance-type term (which is beyond the scope of the paper). Note that the assumptions of the wellposedness result are not satisfied as the estimate \eqref{eq:condition:driftoffsetdiffusion} does not hold (but could readily be achieved by slightly modifying the constants of the coefficients), which indicates that this bound is not sharp.
We consider
\begin{align}
\label{eq:example:supermeasurediff}
\dd  X_{t} &= \big(  v(X_{t},\mu_{t}^{X})+X_{t} \big) \dd t + \big(X_t+\tfrac{1}{4}X_t^2+ f_{\sigma}(X_{t},\mu_{t}^{X}) \big) \dd W_{t},~ 
\\
\nonumber
\textit{with}
~
v(x,\mu)&= - \tfrac{1}{4} x^3+\int_{\bR   } -\big(x-y  \big)^3  \mu(\dd y),
\\\nonumber
f_{\sigma}(x,\mu) &=
\begin{cases}
    \int_{\bR   } \big(x-y  \big)^2  \mu(\dd y), \textit{Case 1,}
    \\
    \int_{\bR   } \int_{\bR   }\big(y-z  \big)^2  \mu(\dd y)\mu(\dd z),\textit{Case 2.}
\end{cases}
\end{align}
For Case 1, we have a nonlinear convolution kernel $f_{\sigma}(x) = x^2$ for all $x \in \mathbb{R}$. Figure \ref{fig:4:the super measure diffusion}, in particular, subplots (a)-(c), illustrates that the SSM converges, in a pointwise sense, with strong order $1/2$ and recovers reasonable density estimates for different choices of the initial distribution. Similar behaviour is not observed for different taming approaches which fail to recover the anticipated strong convergence order of $1/2$ and we observe that taming schemes do not capture the density of the solution well for high-variance initial data. We conducted an analogous test with $v(x,\mu) = -  x^3/4$ in (d), i.e., we removed the convolution term in the drift, and our experiments failed, in the sense that the approximate solutions computed by the SSM did not converge. This supports our theoretical results that a suitable drift compensation for the nonlinear measure component appearing in the diffusion is indeed needed. 

Case 2 corresponds to an example, where the convolution term is again integrated, i.e., resembles a variance-type term. We are not aware of an existing result that yields wellposedness of the underlying MV-SDE including such a term (even without the nonlinear convolution terms). Further, it is not clear which assumptions would be required for a numerical scheme to converge in a strong sense. The expected strong convergence order is observed for the SSM in (e), but no taming approach appears to be a reasonable alternative. We additionally conducted a numerical experiment for Case 2 with $v(x,\mu) = -  x^3/4$, in order to investigate if the variance-type term requires a compensation term (similar to changed Case 1). We also observed that no time-stepping scheme (i.e., taming and SSM) seemed to converge (the result is similar to (d) and we do not present here),  which again indicates that the drift's convolution term can also help to control variance-type terms in the diffusion.
\begin{figure}[h!bt]
    \centering
    \begin{subfigure}{.57\textwidth}
    \setlength{\abovecaptionskip}{-0.02cm}
     \setlength{\belowcaptionskip}{-0.05 cm}  
			\centering
 			\includegraphics[scale=0.51]{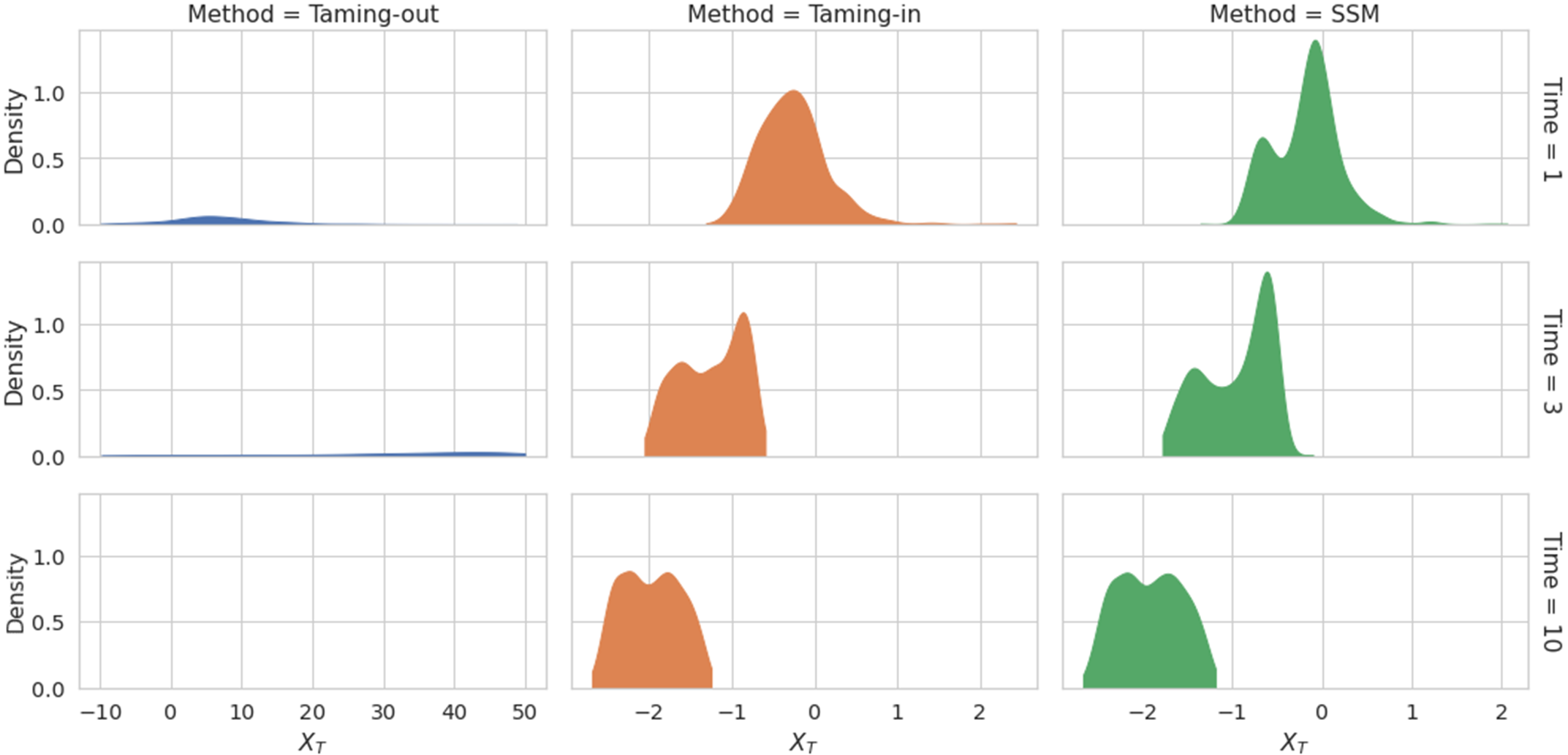}
			\caption{Case 1 with $X_0\sim \cN(0,1)$}
		\end{subfigure}%
		\begin{subfigure}{.41\textwidth}
		\setlength{\abovecaptionskip}{-0.02cm}
		\setlength{\belowcaptionskip}{-0.05 cm} 
			\centering
 			\includegraphics[scale=0.26]{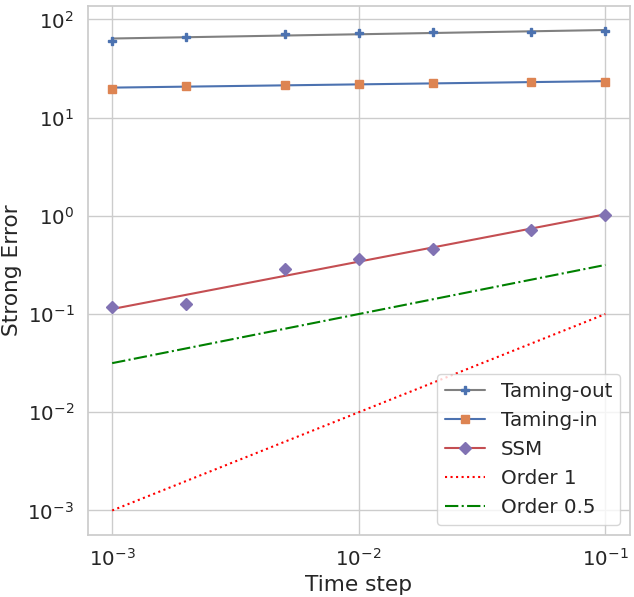}
			\caption{Case 1-Strong error(SE) }
		\end{subfigure}
		\\
            \begin{subfigure}{.40\textwidth}
			\centering
 			\includegraphics[scale=0.22]{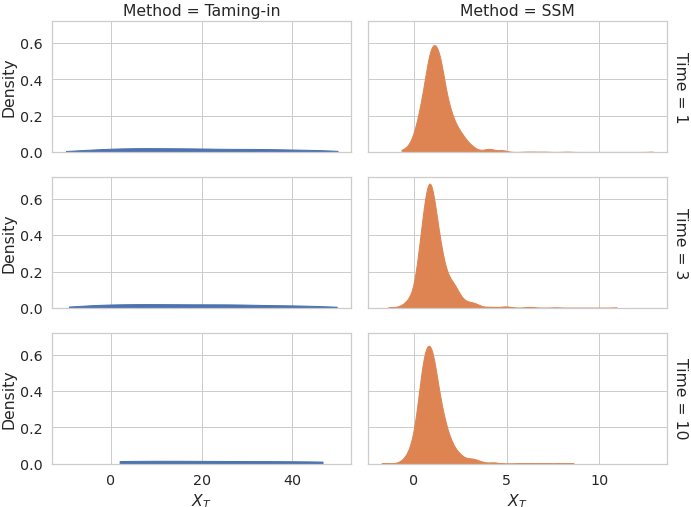}
    \setlength{\abovecaptionskip}{-0.02cm}
 			\caption{  Case 1 with $X_0\sim B(50,0.5)$ }
		\end{subfigure}%
	\begin{subfigure}{.29\textwidth}
			\centering
 			\includegraphics[scale=0.25]{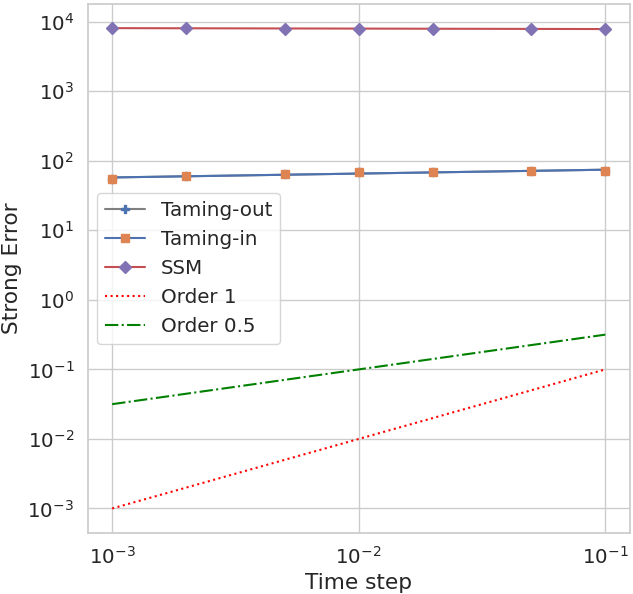}
 			\caption{ Changed Case 1-SE   }
		\end{subfigure}%
	\begin{subfigure}{.29\textwidth}
			\centering
 			\includegraphics[scale=0.25]{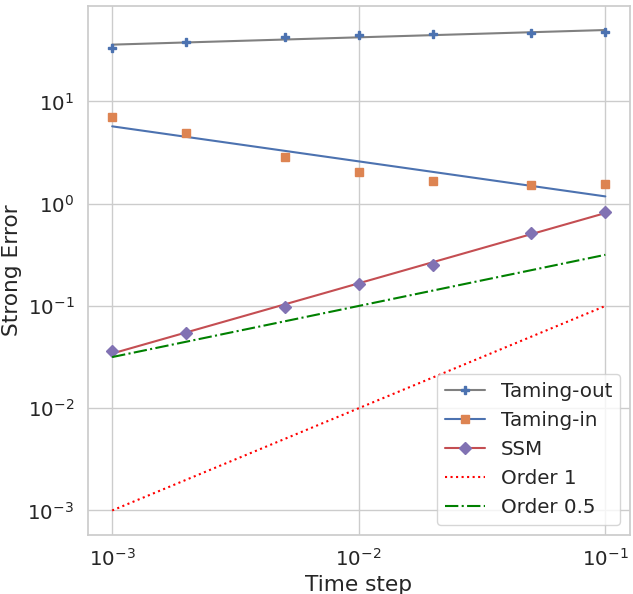}
 			\caption{Case 2-SE }
		\end{subfigure}%
    \caption{Approximation of \eqref{eq:example:supermeasurediff} with $N=1000$ particles. The simulated Brownian motion sample paths and initial distribution are the same for all schemes.  
     (a) and (c)  show the distribution for Taming-out (left),  Taming-in (middle) and SSM (right) with $h=0.01$  at times $T\in\{1,3,10\}$ seen top-to-bottom and with different initial distribution; $x$- and $y$-scales are fixed.  
     (b), (d) and (e) show the strong error (rMSE) of SSM and Taming with $X_0 \sim \cN(1,1)$ for different cases. The changed Case 1 in (d) is Case 1 with $v(x,\mu)=-x^3/4.$ 
     }
    \label{fig:4:the super measure diffusion}
\end{figure}

\subsection{Example: Propagation of Chaos rate across dimensions}

In this example, we estimate the PoC rate depending on the dimension and compare the findings to the theoretical upper bounds established in Theorem \ref{theorem:Propagation of Chaos}. For equation \eqref{Eq:General MVSDE}-\eqref{Eq:General MVSDE shape of v} we make the following choices: Let $d\geq 2$, $x=(x_1,\ldots,x_d)\in \bR^d$, the initial condition $X_0$ is a vector distributed according to $d$-independent $\cN(1,1)$-random variables, and 
\allowdisplaybreaks
\begin{align}
\nonumber
f(x)&=-x|x|^2
,\quad
u (x ) 
=
-\frac13
\left[\begin{array}{cccc}
 x_1^3  , x_2^3 , \cdots , x_d^3 \end{array}\right]^\intercal,
\quad 
b(t,x,\mu)=x,
\quad
\\
\hs (x )
&=
\left[\begin{array}{cccc}
x_1+1/4 ~x_1^2 & x_2 & \cdots & x_d \\
x_1 & x_2+1/4 ~x_2^2  & \cdots & x_d \\
\cdots & \cdots  & \cdots & \cdots   \\
x_1 & x_2  & \cdots & x_d+1/4 ~x_d^2 \\
\end{array}\right].
\label{eq:example:poc}
\end{align}
This is a toy model with a high-dimensional fully coupled convolution kernel and super-linear diffusion term. 
We observe in Figure \ref{fig:4:PoC example} a strong PoC rate, estimated via \eqref{aux_PoC_rate_Estimator}, of order of roughly $1/2$ across dimension $d$. By the ordinary least squares linear regression, for dimension $d\in\{2,3,4,6,10\}$, the corresponding slopes are $\{\textrm{slopes}_d\}_d=\{-0.55,-0.57,-0.5,-0.50,-0.49\}$ and the corresponding $R$-square measure is $\{R^2_d\}_d=\{0.81,0.75,0.92,0.91,0.98\}$.
\begin{figure}[h!bt]
    \centering
    \begin{subfigure}{.33\textwidth}
			\centering
 			\includegraphics[scale=0.28]{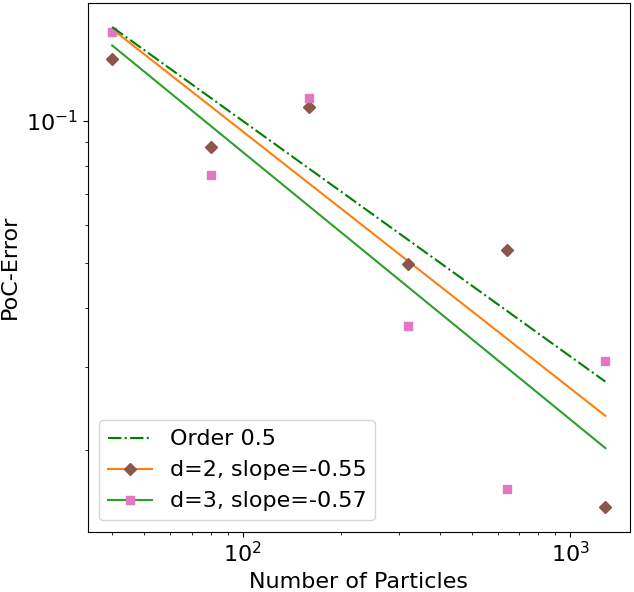}
 			\caption{PoC rates for $d\in\{2,3\}$}
		\end{subfigure}%
	\begin{subfigure}{.33\textwidth}
			\centering
 			\includegraphics[scale=0.28]{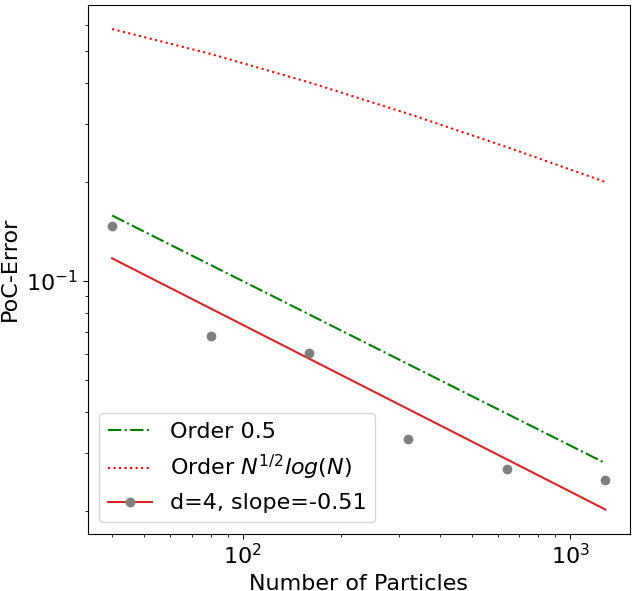}
 			\caption{PoC rates for $d=4$}
		\end{subfigure}%
	\begin{subfigure}{.33\textwidth}
			\centering
 			\includegraphics[scale=0.28]{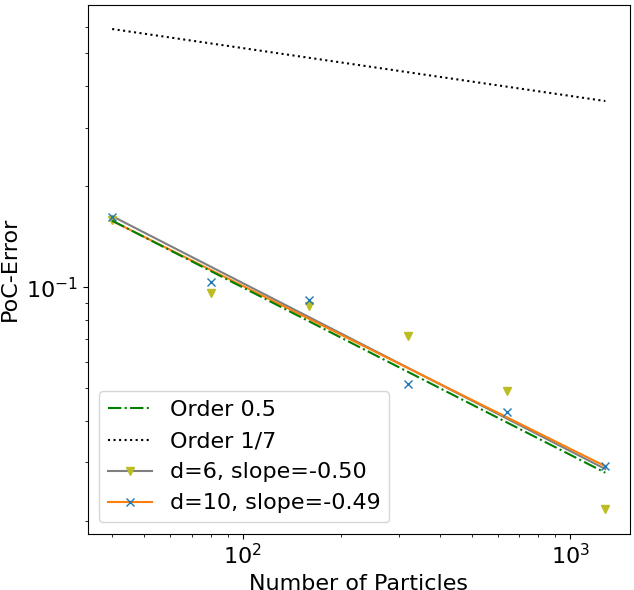}
 			\caption{PoC rates for $d\in\{6,10\}$}
		\end{subfigure}%
    \caption{Estimation of PoC rate for equation \eqref{Eq:General MVSDE}-\eqref{Eq:General MVSDE shape of v} under \eqref{eq:example:poc} using SSM \eqref{eq:SSTM:scheme 0}-\eqref{eq:SSTM:scheme 2} with fixed stepsize $h=10^{-3}$, $T=1$ and number of particles $N\in \{40,80,160,320,640,1280,2560\}$. In all figures the reference rate $0.5$ and the upper bound rate from Theorem \ref{theorem:Propagation of Chaos} are displayed.
 } 
    \label{fig:4:PoC example}
\end{figure}

These findings are inline with those obtained in the one-dimensional example of \cite[Example 4.1]{reisinger2020adaptive}. 
Theorem \ref{theorem:Propagation of Chaos} establishes a strong convergence rate (in terms of number of particles in a pathwise sense) of order $1/4$ for dimensions $d < 4$ only and these results are smaller than the upper bounds of PoC in Theorem \ref{theorem:Propagation of Chaos} -- this highlights a gap in the literature to be explored in future research. 
For perspective, at a theoretical level the rate $1/2$ in $N$ is not new under stronger assumptions. This was obtained in \cite[Lemma 5.1]{delarue2018masterCLT} or  \cite{tse2022} when the drift and diffusion coefficients are assumed to satisfy strong regularity assumptions. Also in \cite{Meleard1996} for linear type MV-SDEs featuring diffusions $\bR^d \ni x\mapsto \hs(x)$ and drifts with structure of the type $\bR^d \ni x\mapsto\int_{\bR^d} b(x,y)\mu(\dd y)$, and requiring that $b,\hs$ are uniformly Lipschitz, the convergence rate $1/2$ in the number of particles is obtained; also in \cite{delarue2021uniform}. 

\subsection{Discussion}

We discuss the advantages of the SSM compared with the taming methods. The SSM converges under all cases, while the two types of taming failed to converge in some cases. The  SSM requires an implicit solver for the convolution kernel but the running time of the SSM compared to the taming methods is only 2 to 3 times longer. From the numerical examples, we see that:
\begin{enumerate}
    \item The two types of strong errors of the SSM are of order 0.5 and consistently outperform that of the proposed taming schemes. In fact, the taming methods are not even expected to converge, however, under a mild initial condition, it is hard to observe the divergence. In the tests with high variance initial distributions, the taming methods diverge while SSM converges consistently. The SSM preserves convergence for larger time steps $h$ (via comparative lower errors) and is also suitable for long-time simulation. 
    
    \item The SSM preserves important geometric properties (the concentration speed of the particles is fast, the expected trajectory coincides with the vector field result), while the taming methods appear to fail to capture these crucial properties. 

   \item We applied the SSM to examples, where the diffusion also involves certain nonlinear measure terms. As long as a suitable monotonicity condition is satisfied the SSM yields promising results.
    
    \item We perform a PoC rate test across dimensions with non-trivial convolution kernel. The rate which we observe numerically is better than the one suggested by the PoC results.
\end{enumerate}

%

%
%
%

\section{Proof of the main results}
\label{section: proof of the main results}

\subsection{Proof of Theorem \ref{Thm:MV Monotone Existence} : Wellposedness and moment stability}
\label{subsection: proof of well-pose and momentbound}

\begin{proof}[Proof of Theorem \ref{Thm:MV Monotone Existence}]
 
The existence and uniqueness follow from modifications of the methodologies used in \cite[Theorem 3.5]{adams2020large}.

\textit{Wellposedness.} 
The proof for existence and uniqueness follows along the same lines as the arguments presented in \cite[Theorem 3.5]{adams2020large}. We repeat here the main steps for convenience. As opposed to more classical approaches, the fixed point argument is carried out over a suitable function space, see \cite{FUNCTIONITER}, instead of a measure space. 

To be precise, one considers the function space $\Lambda_{[0,T],q}$, for $q$ as in Assumption \ref{Ass:Monotone Assumption}, defined as the space of continuous functions $\boldsymbol{g}:[0,T]\times \mathbb{R}^d \to  \mathbb{R}^d \times \mathbb{R}^{d \times l}$, $\boldsymbol{g}(t,x) = (g_1(t,x),g_2(t,x))$ with $g_1:[0,T] \times \mathbb{R}^d \to \mathbb{R}^d$ and $g_2:[0,T] \times \mathbb{R}^d \to \mathbb{R}^{d \times l}$, satisfying 
\begin{align*}
    \| \boldsymbol{g} \|_{[0,T],q}:= \sup_{t \in [0,T]} \bigg( \sup_{x \in \mathbb{R}^d} \frac{|\boldsymbol{g}(t,x) |}{ 1+ | x|^{q+1}} \bigg) < \infty,
\end{align*}
and there exists a constant $L_1  \geq 0$ such that (with $m$ as in Assumption \ref{Ass:Monotone Assumption})
\begin{align}\label{eq:proof_monot}
    \left \langle  x-y,g_1(t,x)-g_1(t,y) \right \rangle + 2(m-1)|g_2(t,x)-g_2(t,y)|^2 
    &\leq L_1 | x-y |^2, 
\end{align}
for all $t \in [0,T]$, $x,y \in \mathbb{R}^d$. In particular, this implies that there exists a constant $L_2 \geq 0$ such that
\begin{equation*}
   \left \langle  x,g_1(t,x) \right \rangle + 2(m-1)|g_2(t,x)|^2 \leq L_2(1 + |\boldsymbol{g}(t,0)|^2 + |x|^2). 
\end{equation*}

For some $K>0$ (chosen below), and a small enough terminal time $T_0$, we define now
\begin{equation*}
    E := \lbrace \boldsymbol{g} \in \Lambda_{[0,T_0],q}: \| \boldsymbol{g} \|_{[0,T_0],q} \leq K \rbrace.
\end{equation*}
We claim that there exist choices for $T_0$ and $K$ such that $\Gamma: E \to E$ defined by
\begin{equation*}
    \Gamma[\boldsymbol{g}](t,x) = (\Gamma[\boldsymbol{g}]_1(t,x),\Gamma[\boldsymbol{g}]_2(t,x) ) := (f \ast \mu_t^{\boldsymbol{g}}(x),f_{\sigma} \ast \mu_t^{\boldsymbol{g}}(x)),
\end{equation*} 
forms a contraction.
Here, $\mu^{\boldsymbol{g}}$ is the law of the solution to the MV-SDE
\begin{align}\label{eq:mv_fixed}
\dd X^{\boldsymbol{g}}_{t} 
&= \big( v^{\boldsymbol{g}}(t,X^{\boldsymbol{g}}_{t},\mu^{\boldsymbol{g}}_t) + b(t,X^{\boldsymbol{g}}_t, \mu^{\boldsymbol{g}}_t) \big)\dd t + \sigma^{\boldsymbol{g}}(t,X^{\boldsymbol{g}}_{t}, \mu_{t}^{\boldsymbol{g}})\dd W_{t},  
\\
\nonumber 
 v^{\boldsymbol{g}}(t,x,\mu) &= g_1(t,x)  +u(x,\mu), \quad \sigma^{\boldsymbol{g}}(t,x,\mu)  = \sigma(t,x,\mu) + g_2(t,x), \quad X^{\boldsymbol{g}}_{0} = X_0 \in L^m( \bR^{d}) .
\end{align}
 {\color{black}The existence of a unique strong solution to $ (X^{\boldsymbol{g}}_{t})_{t \in [0,T]}$ satisfying $\sup_{t \in [0,T]} \mathbb{E}[|X_t^{g}|^m] \leq C$ for some constant $C>0$, is shown in \cite{biswas2020well,kumar2020wellposedness}.} 

We first show that there exist $0<T_0<T$ and $K$ such that $\Gamma$ indeed maps $E$ onto itself. Let $\boldsymbol{g} \in E$. First, we observe that for all $x,y \in \mathbb{R}^d$, $t \in [0,T_0]$
\begin{align*}
   & \left \langle x-y,\Gamma[\boldsymbol{g}]_1(t,x)-\Gamma[\boldsymbol{g}]_1(t,y) \right \rangle + 2(m-1)|\Gamma[\boldsymbol{g}]_2(t,x)-\Gamma[\boldsymbol{g}]_2(t,y)|^2 \\
   & \leq \int_{\mathbb{R}^d} \left( \left \langle x-y,f(x-u)-f(y-u) \right \rangle + 2(m-1)|f_{\sigma}(x-u)-f_{\sigma}(y-u)|^2 \right)\mu_t^{\boldsymbol{g}}(\mathrm{d}u) 
   \\
   & \leq L_1|x-y|^2.
\end{align*} 
Further, we derive, using that $(X^{\boldsymbol{g}}_t)_{t \in [0,T]}$ has finite moments of order $m > 2(q+1)$ that there exist constants $C>0$ and $C(q,\mathbb{E}[|X_0|^{q+1}])>0$ (depending on the moment bounds of the initial data, $q$, and the model parameters) such that 
\begin{align}
   \label{eq: moment bound of gamma g}
    \| &\Gamma[\boldsymbol{g}] \|_{[0,T_0],q} 
    \leq \sup_{t \in [0,T_0]} 
    \bigg( \sup_{x \in \mathbb{R}^d}  \frac{|(f \ast \mu_t^{g})(x)| + |(f_{\sigma} \ast \mu_t^{g})(x)|}{1+|x|^{q+1}} 
    \bigg) 
    \leq C \Big(
    1+\sup_{t\in [0,T_0]}
    \bE \big[ | X^{\boldsymbol{g}}_t |^{q+1}\big]
     \Big) 
    \\
    \nonumber
    & 
      \leq C +
    Ce^{CT_0} \bigg( \mathbb{E}[|X_0|^{q+1}]
    \\
    \nonumber
    & ~ +
    \int_{0}^{T_0} ~
    \big( |b(s,0,\delta_0)|^{q+1}
    +
    |\boldsymbol{g}_1(s,0)|^{q+1}
    + |\boldsymbol{g}_2(s,0)|^{q+1}
        + |u(0,\delta_0)|^{q+1}
        + |\sigma(s,0,\delta_0)|^{q+1} \big)\dd s
    \bigg)
    \\ \nonumber 
    \nonumber
    &
      \leq
     C +
    Ce^{CT_0} \bigg( \mathbb{E}[|X_0|^{q+1}]
    +
    T_0 \| \bdg  \|_{[0,T_0],q}^{q+1}
     \\
    \nonumber
    & ~ +
    \int_{0}^{T_0} ~
    \big( |b(s,0,\delta_0)|^{q+1}
        + |u(0,\delta_0)|^{q+1}
        + |\sigma(s,0,\delta_0)|^{q+1} \big)\dd s
    \bigg)
    \\
    \nonumber
    &
      \leq
     C +
    Ce^{CT_0} \bigg( \mathbb{E}[|X_0|^{q+1}]
    +
    T_0 K^{q+1}
     \\
    \nonumber
    & ~ +
    \int_{0}^{T_0} ~
    \big( |b(s,0,\delta_0)|^{q+1}
        + |u(0,\delta_0)|^{q+1}
        + |\sigma(s,0,\delta_0)|^{q+1} \big)\dd s
    \bigg) \leq K,
\end{align}
for a sufficiently small $T_0>0$ and the choice $K=2C(1+e^{CT}\mathbb{E}[|X_0|^{q+1}])$. 
It remains to show that the map $\Gamma: E \to E$ forms a contraction, i.e., for any $\boldsymbol{g}_1 =(g_{1,1},g_{1,2}),\boldsymbol{g}_2 = (g_{2,1}, g_{2,2}) \in E$, we have
\begin{equation*}
    \| \Gamma[\boldsymbol{g}_1] - \Gamma[\boldsymbol{g}_2] \|_{{[0,T_0],q}} \leq c \|\boldsymbol{g}_1 -\boldsymbol{g}_2 \|_{[0,T_0],q},
\end{equation*}
for $c \in (0,1)$ and a $T_0$ possibly even smaller than chosen above. 
\color{black}
An application of It\^o's formula shows for $t \in [0,T_0]$  
\begin{align*}
    &\mathbb{E}[|X^{\boldsymbol{g}_1}_t - X^{\boldsymbol{g}_2}_t|^2]
    \leq 
    \bE[| X^{\boldsymbol{g}_1}_0 - X^{\boldsymbol{g}_2}_0|^2]
    +
    \int_0^t 
   \bE \Big[    \big| 
     \sigma^\bdg( s, X^{\bdg_1}_s, \mu_s^{\bdg_1}      )
   -
   \sigma^\bdg( s, X^{\bdg_2}_s, \mu_s^{\bdg_2}      )  
     \big|^2  \Big] \dd s
   \\
   &+ 
    2\int_0^t 
   \bE \Big[    \langle X^{\bdg_1}_s - X^{\bdg_2}_s,
     b( s, X^{\bdg_1}_s, \mu_s^{\bdg_1}      )
   -
   b( s, X^{\bdg_2}_s, \mu_s^{\bdg_2}      )  
   \rangle\Big] \dd s 
   \\
   &+ 
    2\int_0^t 
   \bE \Big[    \langle X^{\bdg_1}_s - X^{\bdg_2}_s,
   v^\bdg( s, X^{\bdg_1}_s, \mu_s^{\bdg_1}      )
   -
   v^\bdg( s, X^{\bdg_2}_s, \mu_s^{\bdg_2}      )
   \rangle\Big] \dd s 
     \\
     &\leq 
         \bE[| X^{\boldsymbol{g}_1}_0 - X^{\boldsymbol{g}_2}_0|^2]
    +  \int_0^t C
   \bE \Big[    \big|   X^{\bdg_1}_s - X^{\bdg_2}_s\big|^2       \Big]
   +
   2
   \bE \Big[     \big|        
     g_{1,2}( s, X^{\bdg_1}_s      )
   - g_{1,2}( s, X^{\bdg_2}_s      )  \big|^2
    \Big] \dd s
   \\
   & ~  + 2 \int_0^t \bE \Big[    \langle X^{\bdg_1}_s - X^{\bdg_2}_s,
     g_{1,1}( s, X^{\bdg_1}_s      )
   - g_{1,1}( s, X^{\bdg_2}_s      )  
   \rangle\Big]  \dd s
   \\
   & ~ +2\int_0^t 
   \bE \Big[    \langle X^{\bdg_1}_s - X^{\bdg_2}_s,
       g_{1,1}( s, X^{\bdg_2}_s      )
   -
   g_{2,1}( s, X^{\bdg_2}_s       )  
   \rangle\Big]    \dd s \\
    &~+ 2  \int_0^t
   \bE \Big[     \big|  
   g_{1,2}( s, X^{\bdg_2}_s      )
   -
   g_{2,2}( s, X^{\bdg_2}_s       )  \big|^2
    \Big] \dd s
    \\
    &\leq   
    \bE[| X^{\boldsymbol{g}_1}_0 - X^{\boldsymbol{g}_2}_0|^2]
    + C \int_0^t 
   \bE \Big[    \big|   X^{\bdg_1}_s - X^{\bdg_2}_s\big|^2       \Big] \dd s
   \\
    &~
    +4\int_0^t 
   \bE \Big[     \big|g_{1,1}( s, X^{\bdg_2}_s      )
   -
   g_{2,1}( s, X^{\bdg_2}_s       )  
   \big|^2
   + 
    \big|g_{1,2}( s, X^{\bdg_2}_s      )
   -
   g_{2,2}( s, X^{\bdg_2}_s       )  
   \big|^2
   \Big] \dd s  
   \\
   &\leq  
    \bE[| X^{\boldsymbol{g}_1}_0 - X^{\boldsymbol{g}_2}_0|^2]
    + C \int_0^t 
   \bE \Big[    \big|   X^{\bdg_1}_s - X^{\bdg_2}_s\big|^2       \Big] \dd s
    \\
    &~
    +C\int_0^t \|\boldsymbol{g}_1 -\boldsymbol{g}_2 \|_{[0,T_0],q}^2
   \bE [   1 + |X^{\boldsymbol{g}_2}_s|^{2q+2}
   ] \dd s,
\end{align*}
where we used Young's inequality in the last display. By Gronwall's Lemma, we have 
\begin{align*}
     \sup_{t\in[0,T_0]}\mathbb{E}[&|X^{\boldsymbol{g}_1}_t - X^{\boldsymbol{g}_2}_t|^2]
    \leq C T_0 e^{CT_0}  \sup_{t\in[0,T_0]} \bE [   1 + |X^{\boldsymbol{g}_2}_t|^{2q+2}
   ] \|\boldsymbol{g}_1 -\boldsymbol{g}_2 \|_{[0,T_0],q}^2.
\end{align*}
\color{black}
From the result above,  we have 
\begin{align*}
      & \| \Gamma[\boldsymbol{g}_1] - \Gamma[\boldsymbol{g}_2] \|_{{[0,T_0],q}} 
      \\
      &
       \leq \sup_{t \in [0,T_0]} \bigg( \sup_{x \in \mathbb{R}^d}  \frac{|(f \ast \mu_t^{\boldsymbol{g}_1})(x) - (f \ast \mu_t^{\boldsymbol{g}_2})(x)| 
      + 
      |(f_{\sigma} \ast \mu_t^{\boldsymbol{g}_1})(x) - (f_{\sigma} \ast \mu_t^{\boldsymbol{g}_2})(x)|}{1+|x|^{q+1}} \bigg) \\
      &\leq C \sup_{t \in [0,T_0]} 
      \bigg(
      \sup_{x \in \mathbb{R}^d} \frac{\mathbb{E}\left[|X^{\boldsymbol{g}_1}_t - X^{\boldsymbol{g}_2}_t|(1+|x|^{q+1}) \big( 1 + |X^{\boldsymbol{g}_1}_t|^q + |X^{\boldsymbol{g}_2}_t|^q \big) \right]}{1+|x|^{q+1}} \bigg)
      \\
      & \leq C   \sup_{t \in [0, T_0]}\mathbb{E}\left[|X^{\boldsymbol{g}_1}_t - X^{\boldsymbol{g}_2}_t| \big( 1 + |X^{\boldsymbol{g}_1}_t|^q + |X^{\boldsymbol{g}_2}_t|^q \big) \right] 
      \\
      & \leq C \Big(  \sup_{t \in [0, T_0]}\mathbb{E}[|X^{\boldsymbol{g}_1}_t - X^{\boldsymbol{g}_2}_t|^2] \Big)^{1/2} 
      \Big(  \sup_{t \in [0, T_0]} \mathbb{E} \left[ \left(1 + |X^{\boldsymbol{g}_1}_t|^q + |X^{\boldsymbol{g}_2}_t|^q \right)^2 \right] \Big)^{1/2} 
      \\
      & \leq C  \Big( e^{C T_0} \sqrt{T_0} \Big)\Big(\sup_{t\in[0,T_0]} \bE [   1 + |X^{\boldsymbol{g}_2}_t|^{2q+2}]
      \Big)^{1/2} 
      \\
      &\quad ~\times 
      \Big(  \sup_{t \in [0, T_0]} \mathbb{E} \left[ \left(1 + |X^{\boldsymbol{g}_1}_t|^q + |X^{\boldsymbol{g}_2}_t|^q \right)^2 \right] \Big)^{1/2} \|\boldsymbol{g}_1 -\boldsymbol{g}_2 \|_{[0,T_0],q}
      \\
      & \leq  C \Big( e^{C T_0} \sqrt{T_0} \Big) ~ \Big( 1+  \sup_{t \in [0, T_0]}
       \mathbb{E}  [    |X^{\boldsymbol{g}_1}_t|^{2q+2}    ] + \sup_{t \in [0, T_0]}
       \mathbb{E}  [     |X^{\boldsymbol{g}_2}_t|^{2q+2}   ]
       \Big) 
       \|  \boldsymbol{g}_1 - \boldsymbol{g}_2 \|_{{[0,T_0],q}},
\end{align*}
where we used Young's inequality in the last estimate. 
Performing similar calculations as above for the moments of $X^{\boldsymbol{g}_1}_t$ and $X^{\boldsymbol{g}_2}_t$, which by assumption exist up to order $m > 2q+2$, allows to deduce that $T_0$ can indeed be chosen small enough such that $\Gamma$ maps $E$ onto $E$ and is a contraction operator. We conclude that the sequence $(\boldsymbol{g}^{n})_{n \geq 0}$ defined by $\boldsymbol{g}^{n+1} =  \Gamma [\boldsymbol{g}^{n}]$, for $\boldsymbol{g}^{0} \in E$, is a Cauchy sequence belonging to $E$ and converges with respect to the $\| \cdot \|_{[0,T_0],q}$-norm to $\boldsymbol{g} = \Gamma [\boldsymbol{g}]$ satisfying \eqref{eq:proof_monot}. Thus, for all $t \in [0,T_0]$, we have $$ 
\boldsymbol{g}(t,X_t^{\boldsymbol{g}}) = \Big(f \ast \mu_t^{\boldsymbol{g}}(X_t^{\boldsymbol{g}}),f_{\sigma} \ast \mu_t^{\boldsymbol{g}}(X_t^{\boldsymbol{g}})\Big).
$$
 Substituting this into \eqref{eq:mv_fixed}, yields \eqref{Eq:General MVSDE} and thus $(X_t)_{t \in [0,T_0]}$ with $\sup_{t \in [0,T_0]} \mathbb{E}\left[ |X_t|^{m} \right] < \infty$.

\color{black}
Our challenge now is to find a solution over the whole interval $[0, T ]$. From the above analysis, we observe that the implied constants (and therefore the choice of $T_0$) depend on the moments
of $X_0$. Therefore, we are not immediately able to deduce the existence of a solution on $[0,T]$. We need to ensure that these constants do not explode.

Below, we show pointwise $p$-th moment estimates for $m\geq p> 2$ (the case $p=2$ follows in a straightforward manner from the below arguments where one would use Lemma \ref{AppendixLemma-aux1} and Lemma \ref{AppendixLemma-aux2} instead of the additional symmetry property -- we discuss this in more detail in Section \ref{section: proof of ergodicity} as we prove Theorem \ref{Thm:ergodicity}).
From It\^o's formula, Assumption \ref{Ass:Monotone Assumption} and  \eqref{eq: remark eq1}-\eqref{eq: remark eq7} in Remark \ref{remark:ImpliedProperties}, for all $t \in [0,T_0]$, we deduce
\begin{align}
\nonumber 
    |X_{t}&|^p
    \leq 
      |X_{0} |^p
    +  p \int_0^t |X_{s} |^{p-2}  \langle   X_{s} , v(X_s, \mu^{X}_{s} )  \rangle \dd s
        +
      p \int_0^t |X_{s} |^{p-2} \langle   X_{s} ,  \hs(s,X_s , \mu^{X}_{s} )\dd W_s  \rangle 
    \\\label{eq: ito for X,p}
    &
    +
     p  \int_0^t  |X_{s} |^{p-2} \langle  X_{s},  b(s,X_s, \mu^{X}_{s} )  \rangle \dd s
     \\\nonumber 
    &
    +
     p(p-1)   \int_0^t |X_{s} |^{p-2}  \Big(|\sigma(s,X_s , \mu^{X}_{s} )|^2 +\int_{\bR^d} |\fs( X_s-y)|^2  \mu^{X}_{s}(\mathrm{d}y) \Big) \dd s
    \\ \nonumber 
    \le
    &
     |X_{0} |^p
    +  C \int_0^t \Big(1+ |X_{s} |^{p}  + \big(W^{(2)}(\mu^{X}_{s},\delta_0)\big)^p  \Big)  \dd s
        +
      p \int_0^t |X_{s} |^{p-2} \langle   X_{s} ,  \hs(s,X_s , \mu^{X}_{s} )\dd W_s  \rangle
     \\
    &
    +
     p  \int_0^t  |X_{s} |^{p-2} \Big( \langle  X_{s},  \int_{\bR^d}
     f(X_{s}-y) \mu^{X}_{s}(\dd y)\rangle
     +(p-1)
     \int_{\bR^d} |\fs( X_s-y)|^2  \mu^{X}_{s}(\mathrm{d}y) \Big)
      \dd s.\nonumber
\end{align}
Taking expectation on both sides, using Assumption \ref{Ass:Monotone Assumption}, in particular $(\mathbf{A}^f,~\mathbf{A}^{\fs})$, and \eqref{eq: remark eq1} in Remark \ref{remark:ImpliedProperties}, we derive
\begin{align*}
\nonumber
    \bE   & \big[    |X_{t}|^p    \big]
     \le 
     \bE \big[       |X_{0} |^p  \big] 
    +  
     C \int_0^t (1 + \bE \big[    |X_{s}|^p    \big]) \dd s
     \\
    & \qquad+
    \frac{p}{4}\int_0^t    \int_{\bR^d}  \int_{\bR^d}  (|x|^{p-2} - |y|^{p-2}  )  \langle x+ y
    , f(x-y)  \rangle
    \mu^{X}_{s}(\dd x) \mu^{X}_{s}(\dd y) \dd s
    \\
    & \qquad
    +   \frac{p}{2} \int_0^t    \int_{\bR^d}  \int_{\bR^d} |x|^{p-2}\big( 
        \langle  x-y
    , f(x-y)  \rangle
    +  2(p-1)|\fs( x-y)|^2
    \big)
    \mu^{X}_{s}(\dd x) \mu^{X}_{s}(\dd y) \dd s 
    \\
    \le 
     & 
     \bE \big[       |X_{0} |^p  \big] 
    +  
     C \int_0^t  \bE \big[    |X_{s}|^p    \big] \dd s   +C t.
\end{align*}
Gronwall's lemma yields the pointwise moment estimate, 
    \begin{align}
    \label{eq: moment bound for X}
    \sup_{t\in[0,T_0]} \mathbb E \big[  |X_{t}|^{\widetilde m} \big] 
	&\leq C \left(1+ \bE[|X_0|^{\widetilde m}]\right) e^{C T_0},\qquad \textrm{for any }~ \widetilde m \in [2,m].
\end{align}

Since, we have established a-priori $L^{p}$-moment bounds, for $p \in [2,m]$, (which substitutes for \cite[Proposition 3.13]{adams2020large}), we can repeat the arguments from above to establish the existence of a solution to an arbitrary time interval $[0,T]$. To be more precise, we first show that we can choose constants $K_1,T_1>0$ (independent of $T_0$) such that for $T_0+T_1 \in[0,T]$ we have $ \|\Gamma[\boldsymbol{g}]\|_{[T_0,T_0+T_1],q} \leq K_1$, for $\|\bdg\|_{[T_0,T_0+T_1],q} \leq K_1$. From Equation \ref{eq: moment bound of gamma g}, we get 
\begin{align*}
     \| \Gamma&[\boldsymbol{g}] \|_{[ T_0,T_0+T_1],q} 
    \\     
    &
      \leq C +C
    e^{CT_1} \bigg( \sup_{t\in[0,T_0]}\mathbb{E}[|X_t|^{q+1}]
    +
    T_1 \| \bdg  \|_{[T_0,T_0+T_1],q}^{q+1}
    \\     
    &\qquad \qquad +
    \int_{T_0}^{T_0+T_1} ~
    \Big( |b(s,0,\delta_0)|^{q+1}
        + |u(0,\delta_0)|^{q+1}
        + |\sigma(s,0,\delta_0)|^{q+1} \Big)\dd s
    \bigg)
    \\
      &\leq C +C
    e^{CT_1} \bigg( \sup_{t\in[0,T_0]}\mathbb{E}[|X_t|^{q+1}]
    +
    T_1 K_1^{q+1}
    \\     
    &\qquad \qquad +
    \int_{T_0}^{T_0+T_1} ~
    \Big( |b(s,0,\delta_0)|^{q+1}
        + |u(0,\delta_0)|^{q+1}
        + |\sigma(s,0,\delta_0)|^{q+1} \Big)\dd s
    \bigg) \\
      &\leq C +C
    e^{CT_1} \bigg(  e^{CT_0}
    \Big(1+ \mathbb{E}[|X_0|^{q+1}]\Big)
    +
    T_1 K_1^{q+1}
    \\     
    &\qquad \qquad +
    \int_{T_0}^{T_0+T_1} ~
    \Big( |b(s,0,\delta_0)|^{q+1}
        + |u(0,\delta_0)|^{q+1}
        + |\sigma(s,0,\delta_0)|^{q+1} \Big)\dd s
    \bigg),
\end{align*}
where we used \eqref{eq: moment bound for X}  in the last inequality.

Let now $K_1=2C(1+e^{CT}+e^{CT} \mathbb{E}[|X_0|^{q+1}])$.  Then, we choose $T_1>0$ (independent of $T_0$) small enough such that for any $\| \boldsymbol{g} \|_{[T_0,T_0+T_1],q} \leq K_1$,  we  have  $\|\Gamma[\boldsymbol{g}]\|_{[T_0,T_0+T_1],q}\leq K_1$. Similary as above, we can show that the map $\Gamma: E_1 \to E_1$, where $$
 E_1 := \lbrace \boldsymbol{g} \in \Lambda_{[T_0,T_0 +T_1],q}: \| \boldsymbol{g} \|_{[T_0,T_0 +T_1],q} \leq K_1 \rbrace,
$$
forms a contraction (eventually choosing $T_1$ even smaller as above). The argument from above (choosing $K_2$ etc. as $K_1$) can be repeated to establish the existence of a solution on the time interval $[0,T]$.
\end{proof}

\subsection{Proof of Theorem \ref{theorem:Propagation of Chaos}: Propagation of chaos}
\label{section: proof of poc}

\begin{proof}

Due to Lemma \ref{remark:OSL for the whole function / system V} and conditions $(\mathbf{A}^u,~\mathbf{A}^\sigma, \mathbf{A}^f,~\mathbf{A}^{\fs} )$, \color{black} we observe that the drift and diffusion of the interacting particle system (viewed as an SDE in $\mathbb{R}^{Nd}$ ) satisfy a monotonicity condition as in \cite[Section 2]{mao2008stochastic} which allow us to deduce that the interacting particle system has a unique strong solution. 
Critically, the wellposedness result therein does not yield moment estimates that are independent of $N$, as we interpreted the particle system as one single SDE in $\bR^{Nd}$. In the next step, we prove moment bounds independent of $N$. By It\^o's formula, Assumption \ref{Ass:Monotone Assumption}, \eqref{eq: remark eq1}-\eqref{eq: remark eq7} in Remark \ref{remark:ImpliedProperties}  and Jensen's inequality, we have, for all $t\in[0,T],~i\in \llbracket 1,N \rrbracket$, $2\leq p\leq m$, 
\allowdisplaybreaks
\begin{align*}
   &\bE \big[       |X_{t}^{i,N}|^p \big]
     \leq 
    \bE \big[      |X_{0}^{i,N}|^p  \big] 
    +  p  \bE \Big[ \int_0^t  |X_{s}^{i,N}|^{p-2} \langle      X_{s}^{i,N}, v(X_s^{i,N}, \mu^{X,N}_{s} ) 
    + b(s,X_s^{i,N}, \mu^{X,N}_{s} )
    \rangle   \Big]  \dd s
     \\
    &
   \quad +
     p\bE \Big[  \int_0^t |X_{s}^{i,N}|^{p-2}  
     \Big( \langle   X_{s}^{i,N},  \hs(s,X_s^{i,N}, \mu^{X,N}_{s} )\dd W_s^i \rangle     
     +
     \tfrac{ (p-1)}{2} 
      |\hs(s,X_s^{i,N}, \mu^{X,N}_{s} )|^2 \dd s \Big) \Big] 
    \\
    &
    \le  
    \bE \big[      |X_{0}^{i,N}|^p  \big] 
    + C
   \int_0^t    \bE \big[     |X_{s}^{i,N}|^p  \big]  \dd s
   +CT 
   \\
    &
   \quad +
     p   \int_0^t  \bE \bigg[|X_{s}^{i,N}|^{p-2} \Big(
     \langle      X_{s}^{i,N}, v(X_s^{i,N}, \mu^{X,N}_{s} )  \rangle   
     + (p-1) |\sigma(s,X_s^{i,N}, \mu^{X,N}_{s} )|^2  
   \\
    &
   \quad \quad \quad \quad
   +     
    \langle      X_{s}^{i,N}, \frac1N\sum_{j=1}^N f(X_s^{i,N}-X_s^{j,N} )  \rangle   
   +
   (p-1)  \frac1N\sum_{j=1}^N \big| \fs(X_s^{i,N}-X_s^{j,N} )  \big|^2
   \Big)\bigg]  \dd s 
   \\
   \color{black}
   &
   \leq
     \bE \big[      |X_{0}^{i,N}|^p  \big]   
   + 
   CT+C \int_0^t    \bE \big[    |X_{s}^{i,N}|^p  \big]  \dd s
    \\
    &
 \quad  + 
        \frac{p}{2N}\sum_{j=1}^N
   \int_0^t\bE \Big[ 
   |X_{s}^{i,N}|^{p-2}  \langle      X_{s}^{i,N}-X_s^{j,N},  f(X_s^{i,N}-X_s^{j,N} )  \rangle   
   \Big]  \dd s
   \\
    &
 \quad + 
        \frac {p(p-1)}{N}\sum_{j=1}^N
   \int_0^t\bE \Big[ 
   |X_{s}^{i,N}|^{p-2}  ~
    \big| \fs(X_s^{i,N}-X_s^{j,N} )  \big|^2
    \Big]  \dd s
   \\
    &
  \quad +
       \frac{p}{4N} \sum_{j=1}^N
   \int_0^t \bE \Big[ 
   (|X_{s}^{i,N}|^{p-2}- |X_{s}^{j,N}|^{p-2})  
   \langle      X_{s}^{i,N}+X_s^{j,N},  f(X_s^{i,N}-X_s^{j,N} )  \rangle   
   \Big]  \dd s
   \\
   &
   \le 
     \bE \big[      |X_{0}^{i,N}|^p  \big] 
    + 
   C\int_0^T     \bE \big[      |X_{s}^{i,N}|^p  \big]  \dd s
   +C T.
\end{align*}
\color{black}
In the last estimate, we used the following chain of equalities
\begin{align*}
&\sum_{j=1}^N\bE 
    \Big[ 
   |X_{s}^{i,N}|^{p-2}  \langle      X_{s}^{i,N},  f(X_s^{i,N}-X_s^{j,N} )  \rangle   
   \Big]
   \\
   & = 
    \frac12 \sum_{j=1}^N\bE 
    \Big[ 
    \langle   |X_{s}^{i,N}|^{p-2}    X_{s}^{i,N}-  |X_{s}^{j,N}|^{p-2} X_s^{j,N},  f(X_s^{i,N}-X_s^{j,N} )  \rangle   
   \Big]
   \\
   &=  \frac12 \sum_{j=1}^N\bE 
    \Big[ 
    \langle   |X_{s}^{i,N}|^{p-2}    X_{s}^{i,N}
    -|X_{s}^{i,N}|^{p-2}    X_{s}^{j,N}
    + |X_{s}^{i,N}|^{p-2}    X_{s}^{j,N}
    \\
   &\qquad\qquad \qquad 
    -  |X_{s}^{j,N}|^{p-2} X_s^{j,N},  f(X_s^{i,N}-X_s^{j,N} )  \rangle   
   \Big]
   \\
   &= \frac12 \sum_{j=1}^N\bE 
    \Big[ 
      |X_{s}^{i,N}|^{p-2}   \langle  X_{s}^{i,N}-    X_s^{j,N},  f(X_s^{i,N}-X_s^{j,N} )  \rangle   
   \Big]
   \\
   &\quad +
   \frac{1}{4} 
    \sum_{j=1}^N\bE 
    \Big[ 
    \big \langle    
    \big( |X_{s}^{i,N}|^{p-2}    X_{s}^{j,N}
    -  |X_{s}^{j,N}|^{p-2} X_s^{j,N} \big) 
    \\
   &\qquad\qquad \qquad 
    -
    \big( |X_{s}^{j,N}|^{p-2}    X_{s}^{i,N}
    -  |X_{s}^{i,N}|^{p-2} X_s^{i,N} \big)
    ,  f(X_s^{i,N}-X_s^{j,N} )  \big\rangle   
   \Big]
   \\
   &= \frac12 \sum_{j=1}^N\bE 
    \Big[ 
      |X_{s}^{i,N}|^{p-2}   \langle  X_{s}^{i,N}-    X_s^{j,N},  f(X_s^{i,N}-X_s^{j,N} )  \rangle   
   \Big]
   \\
   &\quad +
   \frac{1}{4} 
    \sum_{j=1}^N\bE 
    \Big[ (   |X_{s}^{i,N}|^{p-2} - |X_{s}^{j,N}|^{p-2}  )
     \langle    
     X_s^{i,N} +  X_s^{j,N} 
    ,  f(X_s^{i,N}-X_s^{j,N} )  \rangle   
   \Big].
\end{align*}
\color{black} 
Taking supremum over $i$ and $t$, shows the claim using Gronwall's lemma.

\color{black}
The estimate (\ref{eq:propogation of chaos, poc result}) is then a consequence of \cite[Theorem 3.14]{adams2020large}. We provide some key differences here. \color{black}Using  It\^o's formula, we have 
\begin{align} 
    \label{eq: poc diff 0}
     \sum_{i=1}^N & \bE \big[   |X_{t}^{i, N}-X_{t}^{i}|^2   \big]
    =
    2\sum_{i=1}^N\int_0^t     \bE \big[    \langle    X_{s}^{i, N}-X_{s}^{i} , v(X_{s}^{i, N}, \mu^{X,N}_{s} )-v(X_{s}^{i}, \mu^{X}_{s} )  \rangle  \big]\dd s
    \\
    \nonumber 
    & \quad+ \sum_{i=1}^N
     \int_0^t    \bE \big[   2 \langle     X_{s}^{i, N}-X_{s}^{i} , b(s,X_{s}^{i, N}, \mu^{X,N}_{s} )-b(s,X_{s}^{i}, \mu^{X}_{s} )  \rangle 
     \\\nonumber
   &\qquad\qquad \qquad 
   + 
     | \hs(s,X_{s}^{i, N}, \mu^{X,N}_{s} )-\hs(s,X_{s}^{i}, \mu^{X}_{s} )   |^2  \big] \dd s
    \\ 
    \nonumber 
   & \le  
   \int_0^t   \sum_{i=1}^N 2\bE \big[    \langle    X_{s}^{i, N}-X_{s}^{i}, b(s,X_{s}^{i, N}, \mu^{X,N}_{s} )-b(s,X_{s}^{i}, \mu^{X}_{s} )  \rangle  
   \big] \dd s
     \\
     \nonumber 
     & \quad + 
     2 \sum_{i=1}^N
     \int_0^t     \bE \big[   
     \langle    X_{s}^{i, N}-X_{s}^{i}, \frac1N \sum_{j=1}^N f( X_{s}^{i, N}-X_{s}^{j, N} ) 
     -\frac1N \sum_{j=1}^N f( X_{s}^{i}-X_{s}^{j} )\rangle\big]\dd s 
     \\
     \nonumber 
     & \quad + 
     2 \sum_{i=1}^N
     \int_0^t     \bE \big[   
     \langle    X_{s}^{i, N}-X_{s}^{i},  
     \frac1N \sum_{j=1}^N f( X_{s}^{i}-X_{s}^{j} )
     - \int_{\bR^d} f(  X_{s}^{i} -y ) \mu_{s}^X(\mathrm{d}y) 
     \rangle\big]\dd s 
      \\
      \nonumber 
     & \quad + 
     4 \sum_{i=1}^N
     \int_0^t     \bE \big[  \big|  
    \frac{1}{N} \sum_{j=1}^N  \fs( X_{s}^{i, N}-X_{s}^{j, N}) -  \frac{1}{N} \sum_{j=1}^N \fs(X_{s}^{i}-X_{s}^{j} ) \big|^2  
     \big]\dd s
      \\
      \nonumber 
     & \quad + 
     4 \sum_{i=1}^N
     \int_0^t     \bE \big[ \big|
      \frac{1}{N} \sum_{j=1}^N  \fs(X_{s}^{i}-X_{s}^{j} )  -  \int_{\bR^d} \fs(X_{s}^{i}-y ) \mu_{s}^X(\mathrm{d}y)  ~  \big|^2
     \big]\dd s 
           \\
           \nonumber 
    & \quad + 
     2 \sum_{i=1}^N \int_0^t     \bE \Big[  \langle    X_{s}^{i, N}-X_{s}^{i}, u(X_{s}^{i, N}, \mu^{X,N}_{s} )-u(X_{s}^{i}, \mu^{X}_{s} )  \rangle
     \\  \nonumber 
    &\qquad\qquad \qquad \qquad \qquad 
     +| \sigma(s,X_{s}^{i, N}, \mu^{X,N}_{s} )-\sigma(s,X_{s}^{i}, \mu^{X}_{s} )   |^2  \Big] \dd s 
     \\
     \nonumber 
     &\leq 
     C  \sum_{i=1}^N \int_0^t  \left(  \bE \big[   |    X_{s}^{i, N}-X_{s}^{i} |^2 
   \big]  
   +  (W^{(2)}( \mu^{X}_{s}, \mu^{X,N }_{s}))^2 \right)
   \dd s
   \\
   \nonumber 
   & \quad +  \frac{1}{N}\sum_{i=1}^N \sum_{j=1}^N  \int_0^t  \Big(  
     \big\langle    (X_{s}^{i, N}-X_{s}^{i}) -  (X_{s}^{j, N}-X_{s}^{j}) ,  f( X_{s}^{i, N}-X_{s}^{j, N} ) - f( X_{s}^{i}-X_{s}^{j} ) \big\rangle  
    \\  \nonumber 
    &\qquad\qquad \qquad \qquad \qquad  +
    4 |   \fs ( X_{s}^{i, N}-X_{s}^{j, N} ) - \fs ( X_{s}^{i}-X_{s}^{j} )  |^2 
    \Big)
    \dd s
     \\
     \nonumber 
     & \quad + 
     2 \sum_{i=1}^N
     \int_0^t     \bE \big[   
     \langle    X_{s}^{i, N}-X_{s}^{i},  
     \frac1N \sum_{j=1}^N f( X_{s}^{i}-X_{s}^{j} )
     - \int_{\bR^d} f(  X_{s}^{i} -y ) \mu_{s}^X(\mathrm{d}y) 
     \rangle\big]\dd s 
     \\
     \nonumber 
     & \quad + 
     4 \sum_{i=1}^N
     \int_0^t     \bE \big[ \big|
      \frac{1}{N} \sum_{j=1}^N  \fs(X_{s}^{i}-X_{s}^{j} )  -  \int_{\bR^d} \fs(X_{s}^{i}-y ) \mu_{s}^X(\mathrm{d}y)  ~  \big|^2
     \big]\dd s 
     \\
     \nonumber 
     &\leq 
     C  \sum_{i=1}^N \int_0^t   \left( \bE \big[   |    X_{s}^{i, N}-X_{s}^{i} |^2 
   \big]  
   +  (W^{(2)}( \mu^{X}_{s}, \mu^{X,N }_{s}))^2 \right)
   \dd s  + \frac C N \sum_{i=1}^N\int_0^t       
   \Big(    \bE    \big[ \big|    X_{s}^{i, N}-X_{s}^{i}      \big|^2\big]  \Big)^{1/2}
   \\
    \label{eq: poc diff 1}
   & \quad \qquad \qquad \cdot 
   \Big(    \bE    \big[ \big|       \sum_{j=1}^N f( X_{s}^{i}-X_{s}^{j} )
     - \int_{\bR^d} f(  X_{s}^{i} -y ) \mu_{s}^X(\mathrm{d}y)    \big|^2\big]  \Big)^{1/2}
   \dd s 
   \\ \label{eq: poc diff 2}
   & \quad + 
   \frac{C}{N^2}   \sum_{i=1}^N\int_0^t 
    \bE    \big[ \big|       \sum_{j=1}^N \big(  \fs ( X_{s}^{i}-X_{s}^{j} )
     - \int_{\bR^d} \fs (  X_{s}^{i} -y ) \mu_{s}^X(\mathrm{d}y)  \big)  \big|^2\big] 
     ~ \dd s. 
\end{align}

Now, to further estimate the terms \eqref{eq: poc diff 1} and \eqref{eq: poc diff 2}, we use similar arguments as in \cite[Equation (3.25)]{adams2020large}. Regarding  \eqref{eq: poc diff 1}, we have 
\begin{align*}
    \bE  &  \big[ \big|       \sum_{j=1}^N \big(  f( X_{s}^{i}-X_{s}^{j} )
     - \int_{\bR^d} f(  X_{s}^{i} -y ) \mu_{s}^X(\mathrm{d}y)    \big) \big|^2\big] 
     \\
     = & \sum_{j,k=1}^N 
       \bE    \big[ \big\langle     f( X_{s}^{i}-X_{s}^{j} )
     - \int_{\bR^d} f(  X_{s}^{i} -y ) \mu_{s}^X(\mathrm{d}y) 
     , 
      f( X_{s}^{i}-X_{s}^{k} )
     - \int_{\bR^d} f(  X_{s}^{i} -y ) \mu_{s}^X(\mathrm{d}y) 
     \big\rangle\big]. 
\end{align*}
For $i\neq j \neq k$, $X_{s}^{i},X_{s}^{j},X_{s}^{k}$ are independent and identically distributed, and consequently, we have 
\begin{align*}
    \bE[ & \big\langle    f(X_{s}^{i}-X_{s}^{j}  )     -   \int_{\bR^d}  f(X_{s}^{i}-y ) \mu_{s}^X(\mathrm{d}y),         f(X_{s}^{i}-X_{s}^{k}  )    -  \int_{\bR^d} f(X_{s}^{i}-y ) \mu_{s}^X(\mathrm{d}y)         \big\rangle     ] 
    \\
    & =  
    \int_{  {\bR^d} } \int_{  {\bR^d} } \int_{  {\bR^d} } \int_{  {\bR^d} } \int_{  {\bR^d} } \langle           f(x - y_1) - f(x-z_1), 
    \\
    & ~   \qquad  \qquad 
    f(x-y_2)-f(x-z_2)
             \big\rangle      \mu_s^X ( \mathrm{d}x )\mu_s^X ( \mathrm{d}y_1 )\mu_s^X ( \mathrm{d}y_2 ) \mu_s^X ( \mathrm{d}z_1 ) \mu_s^X ( \mathrm{d}z_2 )
    \\
    &= (1-1+1-1)\int_{  {\bR^d} } \int_{  {\bR^d} } \int_{  {\bR^d} }  \langle           f(x - y_1)  , f(x-y_2) 
             \big\rangle      \mu_s^X ( \mathrm{d}x )\mu_s^X ( \mathrm{d}y_1 )\mu_s^X ( \mathrm{d}y_2 ) 
             = 0. 
\end{align*}
Thus, only the cases $j=k$ yield a non-zero contribution. Therefore, we deduce  
\begin{align}
    \nonumber 
    \bE   \big[ \big|       \sum_{j=1}^N \big(  f( X_{s}^{i}
    &
    -X_{s}^{j} )
     - \int_{\bR^d} f(  X_{s}^{i} -y ) \mu_{s}^X(\mathrm{d}y)    \big) \big|^2\big] 
     \\
     \label{eq: poc summation square}
     = & \sum_{j=1}^N 
       \bE    \big[ \big|     f( X_{s}^{i}-X_{s}^{j} )
     - \int_{\bR^d} f(  X_{s}^{i} -y ) \mu_{s}^X(\mathrm{d}y) 
     \big|^2\big]
     \leq CN
     , 
\end{align}
where we additionally used the growth for the function $f$ in Assumption \ref{Ass:Monotone Assumption} along with the stability results in Theorem \ref{Thm:MV Monotone Existence}.
Similar arguments apply for $\fs$ in \eqref{eq: poc diff 2}.
By gathering the inequalities from above, we obtain the following estimate for \eqref{eq: poc diff 0}: 
\begin{align*}
    \sum_{i=1}^N \bE \big[ &   |X_{t}^{i, N}-X_{t}^{i}|^2   \big]
     \leq 
     C  \sum_{i=1}^N \int_0^t   \Big(  \bE \big[   |    X_{s}^{i, N}-X_{s}^{i} |^2 
   \big]  
   +  (W^{(2)}( \mu^{X}_{s}, \mu^{X,N }_{s}))^2 
   + \sqrt{N}
   \Big)
   \dd s.
\end{align*}
The  estimate \eqref{eq:propogation of chaos, poc result} in the theorem's statement follows as in the proof of \cite[Theorem 3.14]{adams2020large}. 
\color{black}
\end{proof}

\color{black}

\subsection{Proof of Theorem \ref{Thm:ergodicity} : Exponential contraction and the ergodic property}
\label{section: proof of ergodicity}

For improved readability, we prove each statement of Theorem \ref{Thm:ergodicity} separately but the proof is articulated as a whole in the sense that notations and arguments used in proving a statement will carry into the proof of the following statement (as to avoid repetitions). 

%


We prove the statements by the order they were stated.

\begin{proof}[Proof of Theorem \ref{Thm:ergodicity}]
\textit{Proof of statement 1.} 
Let the corresponding assumptions hold and let $X_0\sim \mu$, $\mu \in \cP_\ell(\bR^d)$ with  $2q+2 <\ell\leq m  $ be given. { \color{black}Applying similar calculations as in \eqref{eq: ito for X,p} for $e^{-\rho_{1,\ell} t} |X_{t}|^\ell$} (with $\rho_{1,\ell} \neq 0$), we deduce that there exists a constant $C>0$ depending on $\ell,~L^{(1)}_{(bu\sigma)}$ and $\sup_t|b(t,0,\delta_0)|$ such that 
\color{black}  
\begin{align}
    \notag 
    &e^{-\rho_{1,\ell} t} \bE \big[    |X_{t}|^\ell \big]
    \\\notag 
    & \quad 
     \leq 
     \bE \big[       |X_{0} |^\ell  \big] 
    +   \int_0^t    \ell e^{-\rho_{1,\ell} s}    \bE \big[    |X_{s}|^{\ell-2} \big( \langle    X_{s}, f( X_{s}-\bar{X}_{s} )   \rangle 
    + (\ell-1)|\fs( X_{s}-\bar{X}_{s} )|
    \big)
    \big]\dd s
    \\\notag 
     & \qquad 
     +
       \big(
       L^{(2)}_{(bu\sigma)}\ell+ L^{(3)}_{(bu\sigma)}\ell -\rho_{1,\ell}  \big) \int_0^t e^{-\rho_{1,\ell} s} \bE \big[    |X_{s}|^\ell  \big] \dd s
       +
       \ell\int_0^t e^{-\rho_{1,\ell} s}
       L^{(1)}_{(bu\sigma)}
       \bE \big[    |X_{s}|^{\ell-2} \big] \dd s
       \\\notag 
        & \quad \leq 
         \bE \big[       |X_{0} |^l  \big]
         +
           \frac{2(\ell  L^{(1)}_{(bu\sigma)})^\ell}{\ell}\int_0^t e^{-\rho_{1,\ell} s} \dd s
        \\\notag 
        &\qquad + 
        \Big(
       \ell( L^{(2)}_{(bu\sigma)}+ L^{(3)}_{(bu\sigma)}+2L^{(1),+}_{(f)}+L^{(3)}_{(f)}/2)+\frac{\ell-2}\ell{}  -\rho_{1,\ell}  \Big) \int_0^t e^{-\rho_{1,\ell} s} \bE \big[    |X_{s}|^\ell  \big] \dd s
       \\
       \label{eq: moment bound ergodicity solution}
       &  \quad 
     \le 
       \bE \big[       |X_{0} |^l  \big] + \frac{C}{\rho_{1,\ell}}(1-e^{-\rho_{1,\ell} t}),
\end{align}
\color{black}
where we used  Assumption \ref{Ass:Monotone Assumption} $(\mathbf{A}^f,~\mathbf{A}^{\fs})~$, \eqref{eq: lp moment expectation result} and { \color{black}  Young's inequality for the last term in the first inequality.} 
Similarly, for $\rho_{1,\ell}= 0 $, we have
\begin{align*}
      \bE \big[    |X_{t}|^\ell    \big]
     \le 
       \bE \big[       |X_{0} |^\ell  \big] + Ct . 
\end{align*}
Using the properties of the Wasserstein metric we have
\begin{align}
\nonumber 
\big(W^{(\ell)}(P^*_{0,t} \mu,\delta_0)\big)^\ell
&=
\big(W^{(\ell)}(\mu^{X}_{t},\delta_0)\big)^\ell 
\le    
\bE \big[    |X_{t} |^\ell    \big]
\\
& \nonumber
\le   
\bE \big[    |X_{0}|^\ell    \big] e^{\rho_{1,\ell} t}
  +\frac{C}{\rho_{1,\ell}}(e^{\rho_{1,\ell} t}-1)  \1_{ \rho_{1,\ell} \neq 0 }+ Ct  \1_{ \rho_{1,\ell} = 0 } 
\\ \label{aux:contratcioninequality3}
&  \leq  
  e^{\rho_{1,\ell} t}  \big(W^{(\ell)}( \mu,\delta_0)\big)^\ell
  +\frac{C}{\rho_{1,\ell}}(e^{\rho_{1,\ell} t}-1)  \1_{ \rho_{1,\ell} \neq 0 }+ Ct  \1_{ \rho_{1,\ell} = 0 }
  .
\end{align}
\end{proof}

\begin{proof}[Proof of Theorem \ref{Thm:ergodicity}]
\textit{Proof of statement 2.} 
\color{black}
Consider two  solutions $X, Y$ of \eqref{Eq:General MVSDE}, driven by the same Brownian motion but with different initial conditions $X_0\sim \mu,Y_0\sim \nu$, $\mu,\nu \in \cP_{\ell}(\bR^d),~\ell>2q+2$. 
Consider the corresponding non interacting particle systems $(X^i_t,Y^i_t)_{i \in \llbracket 1,N\rrbracket}$, for $t \geq 0$, satisfying \eqref{Eq:Non interacting particles}, where each of the initial conditions $(X^i_0)_{i \in \llbracket 1,N\rrbracket}$, $ (Y^i_0)_{i \in \llbracket 1,N\rrbracket}$ are i.i.d.
The corresponding interacting particle couples satisfying \eqref{Eq:MV-SDE Propagation} are denoted by $( X^{i,N}_t,Y^{i,N}_t)_{i \in \llbracket 1,N\rrbracket}$.

The direct study of the difference $X_t-Y_t$ is not a feasible avenue to prove this result. It leads to problems with the estimates involving the convolution term  
\begin{align}
    \label{eq: prob in s2}
     \bE\Big[   
     \Big\langle     X_t-Y_t , \int_{ \bR^d} f( X_t-x ) \mu_{t}^X(\mathrm{d}x) - \int_{ \bR^d} f( Y_{s}-y ) \nu_{t}^Y(\mathrm{d}y) \Big\rangle   \Big],
\end{align}
where $X_t$ and $Y_t$ are not necessarily independent, and hence  Lemma \ref{AppendixLemma-aux2} cannot be applied. 

The route here, instead of the direct analysis of $X_t-Y_t$, is to rely on the triangle inequality and PoC results. To be precise, we will subsequently prove the following chain of inequalities:  
\begin{align*}
\big(W^{(2)}(P^*_{0,t} \mu,P^*_{0,t}\nu)\big)^2
& \leq 
 \bE \big[    |X_{t}^i-Y_{t}^i|^2   \big]
\nonumber 
\\
& \leq 
3\lim_{N\rightarrow \infty } \Big( \bE \big[    |X_{t}^i-X_{t}^{i,N}|^2   \big]
+
\bE \big[    |X_{t}^{i,N}-Y_{t}^{i,N}|^2   \big]
+
\bE \big[    |Y_{t}^i-Y_{t}^{i,N}|^2   \big]
\Big)
\\
& 
= 3  e^{\rho_2 t} \big(W^{(2)}(\mu,\nu)\big)^2.
\end{align*}
%
%
The \textit{first and third term} in above estimate will be analysed via a PoC result, and the \textit{middle term} is carefully estimated in \eqref{eq: middle estimate XiN-yiN} onwards. 
Note that by the established wellposedness, the processes $( X^i_t,Y^i_t)_{i \in \llbracket 1,N\rrbracket}$ and the interacting particle systems $( X^{i,N}_t,Y^{i,N}_t)_{i \in \llbracket 1,N\rrbracket}$ have finite moments up to order $\ell$.  \\


\textit{Part 1: the 1st and 3rd term.} 
We start with the analysis of the first (and third term), similar to the steps used in \eqref{eq: poc diff 0}. Applying It\^o's formula to $|X_{t}^{i, N}-X_{t}^{i}|^2$, taking expectations and summing over $i$ yields 

\begin{align}
    \label{eq: diff second moment bound }
     & \sum_{i=1}^N \bE \big[    |X_{t}^{i, N}-X_{t}^{i}|^2   \big]
    \\
     \nonumber 
     & \leq 
       \sum_{i=1}^N \int_0^t  \Big\{ \big(   2 L_{(b)}^{(2)}  + 2L^{(1)}_{(u\sigma)} +4L_{(f)}^{(1),+} \big)
       \bE \big[   |    X_{s}^{i, N}-X_{s}^{i} |^2 
   \big]  
   \\
   \nonumber 
   & 
   \qquad \qquad \qquad \qquad + 
   \big( 2 L_{(b)}^{(3)} +2L^{(2)}_{(u\sigma)}  \big)
   \bE\big[\big(W^{(2)}( \mu^{X}_{s}, \mu^{X,N }_{s})\big)^2 \big]
   \Big\}\dd s
   \\
    \label{eq: erg poc diff1}
   & + \frac 2 N \sum_{i=1}^N\int_0^t       
   \Big(    \bE    \big[ \big|    X_{s}^{i, N}-X_{s}^{i}      \big|^2\big]  \Big)^{1/2}
   \Big(    \bE    \big[ \big|       \sum_{j=1}^N f( X_{s}^{i}-X_{s}^{j} )
     - \int_{\bR^d} f(  X_{s}^{i} -y ) \mu_{s}^X(\mathrm{d}y)    \big|^2\big]  \Big)^{1/2}
   \dd s 
   \\ \label{eq: erg poc diff2}
   & + 
   \frac{4}{N^2}   \sum_{i=1}^N\int_0^t 
    \bE        \big[ \big|       \sum_{j=1}^N \big(  \fs ( X_{s}^{i}-X_{s}^{j} )
     - \int_{\bR^d} \fs (  X_{s}^{i} -y ) \mu_{s}^X(\mathrm{d}y)  \big)  \big|^2\big] 
     ~ \dd s.
\end{align} 
Further, from the calculations in  \eqref{eq: poc summation square} along with the uniform
moment bound result \eqref{eq: moment bound ergodicity solution}, we have (recall    $\rho_{1,\ell}<0$ for $2q+2 <\ell\leq m $) 
\begin{align*}
\bE \big[    |X_{s}^i |^{2q+2}    \big]
\leq 1+ 
     \bE \big[    |X_{s}^i |^\ell    \big]
&\le   
\bE \big[    |X_{0}|^\ell    \big] e^{\rho_{1,\ell} t}
  +\frac{C}{\rho_{1,\ell}}(e^{\rho_{1,\ell} t}-1)  
  \leq C_{\ell},
\end{align*}
for some constant $C_\ell$ independent of $t,N,i$.
Therefore, for the terms in \eqref{eq: erg poc diff1} and \eqref{eq: erg poc diff2}, we have 
\begin{align*}
     &\bE        \big[ \big|       \sum_{j=1}^N\big(  f ( X_{s}^{i}-X_{s}^{j} )
     - \int_{\bR^d} f (  X_{s}^{i} -y ) \mu_{s}^X(\mathrm{d}y)  \big)  \big|^2\big]
     \leq C N \bE  \big[ \big|   X_{s}^{i}  \big|^{2q+2}\big] 
     \leq C_{\ell} N ,
  \\
     &\bE        \big[ \big|       \sum_{j=1}^N\big(  \fs ( X_{s}^{i}-X_{s}^{j} )
     - \int_{\bR^d} \fs (  X_{s}^{i} -y ) \mu_{s}^X(\mathrm{d}y)  \big)  \big|^2\big]
     \leq C N \bE  \big[ \big|   X_{s}^{i}  \big|^{2q+2}\big] 
     \leq C_{\ell} N.
\end{align*} 
To summarise, we derive for \eqref{eq: diff second moment bound }

\begin{align}
    \notag 
     & \frac{1}{N}\sum_{i=1}^N \bE \big[    |X_{t}^{i, N}-X_{t}^{i}|^2   \big]
     \leq 
     \frac{1}{N}\sum_{i=1}^N
       \int_0^t  
       \bigg( 
        \big(    2 L_{(b)}^{(2)}  + 2L^{(1)}_{(u\sigma)} +4L_{(f)}^{(1),+} \big)
       \bE \big[   |    X_{s}^{i, N}-X_{s}^{i} |^2 
   \big]  
   \\
   \label{eq: summarize formula for xiN-xi}
    & 
    \qquad + 
   \big( 2 L_{(b)}^{(3)} +2L^{(2)}_{(u\sigma)}  \big)
   \bE\big[\big(W^{(2)}( \mu^{X}_{s}, \mu^{X,N }_{s})\big)^2 \big]
   +
   \frac{C_\ell}{N} \sum_{i=1}^N \sqrt{N}
   + \frac{C_\ell}{N^2} \sum_{i=1}^N  
    N 
   \bigg)   \dd s.
\end{align}

Having obtained estimates for the average, we now go back to It\^o's formula applied to $e^{-\rho_2 t}|X_{t}^{i, N}-X_{t}^{i}|^2$. Taking expectation and taking \eqref{eq: summarize formula for xiN-xi} into account, we have
\begin{align}
    \nonumber 
       e^{ -\rho_{2} t}
       &
       \bE \big[    |X_{t}^{i, N} -X_{t}^{i}|^2   \big]
     \\      \nonumber 
     &
     \leq 
       \int_0^t  
       e^{ -\rho_{2} s }\bigg( 
        \big(  -\rho_{2}+  2 L_{(b)}^{(2)}  + 2L^{(1)}_{(u\sigma)} +4L_{(f)}^{(1),+} \big)
       \bE \big[   |    X_{s}^{i, N}-X_{s}^{i} |^2 
   \big]  
   \\
     \label{eq: ii poc diff} 
    & 
    \qquad \qquad \qquad \qquad \qquad + 
   \big( 2 L_{(b)}^{(3)} +2L^{(2)}_{(u\sigma)}  \big)
   \bE\big[\big(W^{(2)}( \mu^{X}_{s}, \mu^{X,N }_{s})\big)^2 \big]
   + \frac{C_\ell     }{\sqrt{N}}
   \bigg)   \dd s.
\end{align} 

Similar to the arguments  in  \cite[Equation (3.26)]{adams2020large},  we have 
\begin{align}
    \notag 
    \bE\Big[\Big(W^{(2)}( \mu^{X}_{s}, \mu^{X,N }_{s})\Big)^2 
       \Big]
    \leq 
    \frac{2}{N} \sum_{j=1}^N \bE \big[    |X_{s}^{j, N}-X_{s}^{j}|^2   \big]
    +2 \bE\Bigg[\bigg(W^{(2)}\Big( \mu^{X}_{s}, \frac{1}{N} \sum_{j=1}^N \delta_{X_{s}^{j}} \Big)~\bigg)^2
        \Bigg].
\end{align}
Now, applying the uniform $\ell$-moment bound in \eqref{eq: moment bound ergodicity solution} to $ X_{t}^{j}$, we have $ \mu_t^X \in \cP_\ell(\bR^d)$ for any $t \geq 0$. Consequently, from the result in \cite[Theorem 5.8]{CarmonaDelarue2017book1}, we have 
\begin{align}
    \label{eq: result of poc carmona}
     \bE\Bigg[\bigg(W^{(2)}\Big( \mu^{X}_{s}, \frac{1}{N} \sum_{j=1}^N \delta_{X_{s}^{j}} \Big)~\bigg)^2
     \Bigg]
     \leq C_{N,d} := 
     C \begin{cases}
         N^{-1/2}, &\quad d < 4,
         \\
         N^{-1/2} \ln N, &\quad d =4,
         \\
         N^{-2/d}, &\quad d >  4, 
     \end{cases}
\end{align}
for some positive constant $C$ independent of $s,N$. 

Injecting the result above into \eqref{eq: ii poc diff}, and recalling that $\rho_2= 2L_{(b)}^{(2)}+ 4L_{(b)}^{(3)}+2 L^{(1)}_{(u\sigma)} +4 L^{(2)}_{(u\sigma)} + 4L^{(1),+}_{(f)}$, we obtain 
\begin{align}
    \notag 
     e^{ -\rho_{2} t}\bE \big[    |X_{t}^{i, N}-X_{t}^{i}|^2   \big]
     &\leq  \int_0^t  
       e^{ -\rho_{2} s }\bigg(   \frac{C_\ell    }{\sqrt{N}}
       + 4\big(  L_{(b)}^{(3)} +L^{(2)}_{(u\sigma)}  \big) C_{N,d}
   \bigg)   \dd s
    \\
    \notag 
     &
    \leq   
    \frac{ C_\ell}{  \rho_2}   (1-e^{-\rho_2 t})\big(  \frac{1}{\sqrt{N}} +C_{N,d} \big)
    \1_{ \rho_2 \neq 0 }
    +
    C_\ell\big(  \frac{1}{\sqrt{N}} +C_{N,d} \big) t  \1_{ \rho_2 = 0 }
           ,
\end{align}
for some positive constant $C_\ell $ independent of $t,N,i$. 

The final step of these calculations is to conclude the PoC result (uniform in time if $\rho_2 <0$). Concretely, from the last estimate, it follows that 
\begin{align}
    \label{eq: poc result uni }
    \lim_{N\rightarrow \infty} \bE \big[    |X_{t}^{i, N}-X_{t}^{i}|^2   \big] = 0. 
\end{align}

\textit{Part 2: the middle term.} 
We now proceed with the second part of the proof and tackle, estimates for differences between the two particle systems $|X_t^{i,N}-Y_t^{i,N}|$. First, we obtain estimates for the averages, then estimates for the particles $i$ themselves and finally draw the conclusion. 

Applying It\^o's formula to $|X_{t}^{i, N}-Y_t^{i,N}|^2$, taking expectations and summing over $i$ yields 
\begin{align} 
    \label{eq: middle estimate XiN-yiN}
     & \sum_{i=1}^N  \bE \big[    |X_{t}^{i, N}-Y_{t}^{i,N}|^2   \big]
    =  \sum_{i=1}^N \bE \big[    |X_{0}^{i, N}-Y_{0}^{i,N}|^2   \big]
     \\
    \nonumber 
    & \quad+ 
    2\sum_{i=1}^N\int_0^t     \bE \big[    \langle    X_{s}^{i, N}-Y_{s}^{i,N} , v(X_{s}^{i, N}, \mu^{X,N}_{s} )-v(Y_{s}^{i,N}, \mu^{Y,N}_{s} )  \rangle  \big]\dd s
    \\
    \nonumber 
    & \quad+ 
     2\sum_{i=1}^N\int_0^t    \bE \big[    \langle     X_{s}^{i, N}-Y_{s}^{i,N} , b(s,X_{s}^{i, N}, \mu^{X,N}_{s} )-b(s,Y_{s}^{i,N}, \mu^{Y,N}_{s}  )  \rangle \big] \dd s
     \\
    \nonumber 
    & \quad+ \sum_{i=1}^N
      \int_0^t    \bE \big[ 
     | \hs(s,X_{s}^{i, N}, \mu^{X,N}_{s} )-\hs(s,Y_{s}^{i,N}, \mu^{Y,N}_{s}  )   |^2  \big] \dd s
     \\
      \nonumber 
   & \le \sum_{i=1}^N \bE \big[    |X_{0}^{i, N}-Y_{0}^{i,N}|^2   \big]
     \\
    \label{eq: psx-psy term 1 } 
    & \quad+    
   2
   \int_0^t   \sum_{i=1}^N \bE \big[    \langle    X_{s}^{i, N}-Y_{s}^{i,N}, b(s,X_{s}^{i, N}, \mu^{X,N}_{s} )-b(s,Y_{s}^{i,N}, \mu^{Y,N}_{s}  )  \rangle  
   \big] \dd s
     \\
     \label{eq: psx-psy term 2 }
     & \quad + 
     2 \sum_{i=1}^N
     \int_0^t     \bE \big[   
     \langle    X_{s}^{i, N}-Y_{s}^{i,N}, \frac1N \sum_{j=1}^N f( X_{s}^{i, N}-X_{s}^{j, N} ) 
     -\frac1N \sum_{j=1}^N f( Y_{s}^{i,N}-Y_{s}^{j,N} )\rangle\big]\dd s 
      \\
     \label{eq: psx-psy term 3 }
     & \quad + 
     2 \sum_{i=1}^N
     \int_0^t     \bE \big[  \big|  
    \frac{1}{N} \sum_{j=1}^N  \fs( X_{s}^{i, N}-X_{s}^{j, N}) -  \frac{1}{N} \sum_{j=1}^N \fs(X_{s}^{i}-X_{s}^{j} ) \big|^2  
     \big]\dd s
      \\
     \label{eq: psx-psy term 4 } 
    & \quad + 
     2 \sum_{i=1}^N \int_0^t     \bE \Big[  \langle   X_{s}^{i, N}-Y_{s}^{i,N}, u(X_{s}^{i, N}, \mu^{X,N}_{s} )-u(Y_{s}^{i,N},  \mu^{Y,N}_{s} )  \rangle \Big] \dd s 
     \\
      \label{eq: psx-psy term 5 }
    & \quad+2 \sum_{i=1}^N \int_0^t 
     \bE \Big[  
    | \sigma(s,X_{s}^{i, N}, \mu^{X,N}_{s} )-\sigma(s,Y_{s}^{i,N},  \mu^{Y,N}_{s} )   |^2  \Big] \dd s. 
\end{align} 
Now, from Assumption \ref{Ass:Monotone Assumption}, in particular 
$ (\mathbf{A}^b,\mathbf{A}^u,~\mathbf{A}^\sigma)$, we have 
\begin{align}
    \label{eq: psx-psy term 1 45 result }
    \nonumber 
    \eqref{eq: psx-psy term 1 } 
    &
    + \eqref{eq: psx-psy term 4 } +
    \eqref{eq: psx-psy term 5 } 
    \\
    & 
    \leq 
    \big(   2 L_{(b)}^{(2)}  + 2 L_{(b)}^{(3)}  + 2L^{(1)}_{(u\sigma)}  + 2L^{(2)}_{(u\sigma)} \big)
        \sum_{i=1}^N \int_0^t    \bE \big[   |    X_{s}^{i, N}-Y_{s}^{i,N} |^2 
   \big]    
   \dd s.
\end{align}
 From $(\mathbf{A}^f,~\mathbf{A}^{\fs})$, we have 
\begin{align}
    \notag 
     & \eqref{eq: psx-psy term 2 }  + \eqref{eq: psx-psy term 3 }  
     \\  \notag 
     & \leq  
     \frac{2}{N}\sum_{i=1}^N \sum_{j=1}^N  \int_0^t \bE \Big[   
     \big\langle    (X_{s}^{i, N}-Y_{s}^{i, N})   ,  f( X_{s}^{i, N}-X_{s}^{j, N} ) - f( Y_{s}^{i, N}- Y_{s}^{j, N} ) \big\rangle  
    \\   \notag  
    &\qquad\qquad \qquad \qquad \qquad  +
      |   \fs ( X_{s}^{i, N}-X_{s}^{j, N} ) - \fs (Y_{s}^{i, N}- Y_{s}^{j, N} )  |^2 
     \Big]
    \dd s
     \\  
    \nonumber
     & = 
     \frac{1}{N}\sum_{i=1}^N \sum_{j=1}^N  \int_0^t 
     \bE \Big[  
     \big\langle    (X_{s}^{i, N}-Y_{s}^{i, N}) -  (X_{s}^{j, N}-Y_{s}^{j, N}) ,  f( X_{s}^{i, N}-X_{s}^{j, N} ) - f( Y_{s}^{i, N}- Y_{s}^{j, N} ) \big\rangle  
    \\    \label{eq: psx-psy term symmetry } 
    &\qquad\qquad \qquad \qquad \qquad  +
    2 |   \fs ( X_{s}^{i, N}-X_{s}^{j, N} ) - \fs (Y_{s}^{i, N}- Y_{s}^{j, N} )  |^2 
    \Big] 
    \dd s,
    \\ 
    \notag 
    &\leq 
    \frac{L_{(f)}^{(1),+} }{N}\sum_{i=1}^N \sum_{j=1}^N
    \int_0^t 
    \bE\big[     |(X_{s}^{i, N}-Y_{s}^{i, N}) -  (X_{s}^{j, N}-Y_{s}^{j, N})   |^2  \big] 
    \dd s
     \\ \label{eq: psx-psy term 2 3 result }
    &\leq
      4 L_{(f)}^{(1),+}  \sum_{i=1}^N  
    \int_0^t 
    \bE\big[     |X_{s}^{i, N}-Y_{s}^{i, N} |^2  \big] 
    \dd s.
\end{align}

Substituting the results of \eqref{eq: psx-psy term 1 45 result } and      \eqref{eq: psx-psy term 2 3 result } into \eqref{eq: psx-psy term 1 }-\eqref{eq: psx-psy term 5 }, we conclude
\begin{align*}   
     \sum_{i=1}^N \bE \big[ &   |X_{t}^{i, N}-Y_{t}^{i,N}|^2   \big]
    \leq
    \sum_{i=1}^N \bE \big[    |X_{0}^{i, N}-Y_{0}^{i,N}|^2   \big]
     \\
    \nonumber 
    & \quad+       
      \big(   2 L_{(b)}^{(2)}  + 2 L_{(b)}^{(3)}  + 2L^{(1)}_{(u\sigma)} + 2L^{(2)}_{(u\sigma)}+4L_{(f)}^{(1),+} \big) \sum_{i=1}^N \int_0^t    \bE \big[   |  X_{s}^{i, N}-Y_{s}^{i,N} |^2 
   \big]   
   \dd s . 
\end{align*}

As in the previous step, having obtained estimates for the average, we now go back to It\^o's formula applied to $e^{-\rho_2 t}|X_{t}^{i, N}-Y_{t}^{i, N}|^2$. In particular, we have 
\begin{align}
    \notag 
     e^{-\rho_2 t}&
     \sum_{i=1}^N \bE \big[    |X_{t}^{i, N}-Y_{t}^{i,N}|^2   \big]
    \leq   \sum_{i=1}^N \bE \big[    |X_{0}^{i, N}-Y_{0}^{i,N}|^2   \big]
      \\
    \nonumber 
    & +       
      \big(- \rho_2 + 2 L_{(b)}^{(2)}  + 2 L_{(b)}^{(3)}  + 2L^{(1)}_{(u\sigma)} + 2L^{(2)}_{(u\sigma)}+4L_{(f)}^{(1),+} \big) \sum_{i=1}^N \int_0^t   e^{\rho_2 s} \bE \big[   |  X_{s}^{i, N}-Y_{s}^{i,N} |^2 
   \big]   
   \dd s
   \\
   \label{eq: ps diff result }
   &\leq \sum_{i=1}^N \bE \big[    |X_{0}^{i, N}-Y_{0}^{i,N}|^2   \big]. 
\end{align}

Combining the results  in   \eqref{eq: poc result uni } and \eqref{eq: ps diff result }, we have 
\begin{align}
\label{eq: diff methodology}
\bE \big[    |X_{t}^i-Y_{t}^i|^2   \big]
  &= 
  \lim_{N\rightarrow \infty}   \frac{1}{N}\sum_{i=1}^N\bE \big[     |X_{t}^i-Y_{t}^i|^2   \big]
\\
\notag 
&\leq 
\lim_{N\rightarrow \infty} 
\frac{3}{N}\sum_{i=1}^N
\bigg(
\bE \big[     |X_{t}^i-X_{t}^{i,N}|^2   \big]
+
\bE \big[     |Y_{t}^i-Y_{t}^{i,N}|^2   \big]
+
\bE \big[     | X_{t}^{i,N}-Y_{t}^{i,N}|^2   \big]
\bigg)
\\
\notag
&\leq 
\lim_{N\rightarrow \infty}
\frac{3}{N}\sum_{i=1}^N
\bE \big[     | X_{t}^{i,N}-Y_{t}^{i,N}|^2   \big]
=
3\bE \big[    |X_{0}-Y_{0}|^2   \big] e^{ \rho_2 t }. 
\end{align}

\textit{Part 3: conclusion.} 
Finally, using the properties of the Wasserstein metric, we have
\begin{align}
\label{aux:contratcioninequality}
\big(W^{(2)}(P^*_{0,t} \mu,P^*_{0,t}\nu)\big)^2
\leq 
 \bE \big[    |X_{t}-Y_{t}|^2   \big]
& \nonumber 
\le    3 
\bE \big[    |X_{0}-Y_{0}|^2   \big] e^{ \rho_2 t }
\\
& 
= 3  e^{\rho_2 t} \big(W^{(2)}(\mu,\nu)\big)^2,
\end{align}
where in the last inequality we took the infimum on both sides over all couplings between $\mu$ and $\nu$. This concludes the proof of the second statement.

\color{black}
\end{proof}

\begin{proof}[Proof of Theorem \ref{Thm:ergodicity}]
\textit{Proof of statement 3.} 
In the previous two statements we worked on the finite time interval $[0,T]$ and this statement extends the work to $ [0,\infty)$. We also emphasize that the reason why we work with $\cP_{2\ell-2}$ instead of $\cP_{\ell}$ with $1+m/2\geq\ell>2q+2$ will become apparent later in the proof. Let $X_0\sim \mu_0,Y_0\sim \nu_0$ with $\mu_0,\nu_0 \in \cP_{2\ell-2}(\bR^d)$ be given.  

From Theorem \ref{Thm:MV Monotone Existence} and the flow property on $\cP_{\ell}(\bR^d)$ of \eqref{Eq:General MVSDE} described by the semigroup operator $(P^*_{s,t})$ (defined above Theorem \ref{Thm:ergodicity}), we  extend $(\mu^{X}_{t})_{t \geq 0}$, $(\mu^{Y}_{t})_{t \geq 0}$ (e.g., via patching up solutions inductively over intervals $[nT,(n+1)T]$ for $n\in\bN$). 
Further, since $\rho_2 <0$, we have a contraction in \eqref{aux:contratcioninequality} and hence $\lim_{t\rightarrow \infty}W^{(2)}(\mu^{X}_{t},\mu^{Y}_{t})=0$.  
By using $\rho_{1,2\ell-2} < 0$, we have $\sup_{t \geq 0}W^{(2\ell-2)}(\mu^{X}_{t},\delta_0)< \infty$, which guarantees that $\mu^{X}_{t}\in\cP_{2\ell-2}(\bR^d)$ for all $t\ge 0$. 
The main proof follows via a shift-coupling argument and the properties shown so far under $\rho_2  , \rho_{1,2\ell-2}< 0$, but with a critical additional element regarding establishing contraction and higher order moments for the candidate invariant measure so that the wellposedness result applies.  

%
We start by showing that $(P^*_{0,t}\nu_0)_{t\geq 0}$ is a Cauchy-sequence in $(\cP_2(\mathbb{R}^d),W^{(2)})$, and use this result to show that $(P^*_{0,t}\nu_0)_{t\geq 0}$ is also Cauchy-sequence in $(\cP_{\ell}(\mathbb{R}^d),W^{({\ell})})$ for a given $\nu_0\in \cP_{2\ell-2}(\mathbb{R}^d)$. These arguments suffice to first find a candidate invariant distribution and then to characterize it as an ergodic limit (see below).

\textit{Using the $W^{(2)}$-contraction.} Given $\nu_0\in \cP_{2\ell-2}(\bR^d)$, from \eqref{eq:aux.SemigroupDifferenceEstimate} with $\rho_2<0$, we have exponential contraction and hence for any $0\leq s <t<\infty$
\begin{align*}
    W^{(2)}\big(P^*_{0,t} \nu_0,P^*_{0,t+s}\nu_0 \big)
    =
    W^{(2)}\big( P^*_{0,t} \nu_0 , P^*_{0,t}\big(P^*_{0,s}\nu_0) \big)
    \le    
   e^{\rho_2 t/2} W^{(2)}\big(\nu_0,P^*_{0,s}\nu_0\big),
\end{align*}
where we used the semigroup property that $P^*_{s,t}=P^*_{0,t-s}$ (since $b,\sigma$ are independent of $t$; see \cite{Wang2018DDSDE-LandauType,hu2021long,liu2022ergodicity}).

\textit{The bounded orbit argument.} From \eqref{eq:Bounded Orbit} with $\rho_{1,2\ell-2}< 0$ and $m \geq 2\ell-2> 4q+2$, we have via the  triangle inequality

\begin{align}
\nonumber
    \sup_{t\geq 0}&
    \big(W^{(2\ell-2)}\big(P^*_{0,t} \nu_0, \nu_0\big)  \big)^{2\ell-2}
    \le   
    C \Big(\big(W^{(2\ell-2)}\big(\nu_0, \delta_0\big)\big)^{2\ell-2}+
    \sup_{t\geq 0}
    \big(W^{(2\ell-2)}\big(P^*_{0,t} \nu_0, \delta_0\big)\big)^{2\ell-2} \Big)
     \\\nonumber
    & \le   
    C \bigg(\big(W^{(2\ell-2)}\big(\nu_0, \delta_0\big)\big)^{2\ell-2}+
    \sup_{t\geq 0} \Big(e^{\rho_{1,2\ell-2}t}
    \big(W^{(2\ell-2)}\big( \nu_0, \delta_0\big)\big)^{2\ell-2}\Big)
     \\\nonumber
    & \qquad \qquad \qquad \qquad \qquad \qquad  \qquad 
    +\sup_{t\geq 0}\frac{1}{\rho_{1,2\ell-2}}(e^{\rho_{1,2\ell-2}t}-1  ) \bigg)
    \\\label{eq:proof of egordicity:bounded orbit}
    & \le C \Big(  
    \big(W^{(2\ell-2)}\big(\nu_0, \delta_0\big)\big)^{2\ell-2}
    - \tfrac{1}{ \rho_{1,2\ell-2}} \Big)
    <\infty.
\end{align}
In other words, the orbit of $t\mapsto P^{*}_{0,t}\nu_0$ remains within a sufficiently large $W^{(2\ell-2)}$-ball, which also shows the finiteness of 
$ \sup_{t\geq 0}  W^{(2)}\big(P^*_{0,t} \nu_0, \nu_0\big)$. 

\textit{A $W^{(2)}$-Cauchy-sequence and the completeness argument.} Combining the two previous elements we have
\begin{align*}
\lim_{s\to \infty} 
    W^{(2)}\big(P^*_{0,t} \nu_0,P^*_{0,t+s}\nu_0 \big)
    &
    =
    \lim_{s\to \infty} 
    W^{(2)}\big( P^*_{0,t} \nu_0 , P^*_{0,t} (P^*_{0,s}\nu_0) \big)
    \\
    &
    \le    
   e^{\rho_2 t/2} \lim_{s\to \infty} W^{(2)}\big(\nu_0,P^*_{0,s}\nu_0\big) 
   \leq C e^{\rho_2 t/2}.
\end{align*}
This shows the sequence to be Cauchy and since $(\cP_2(\mathbb{R}^d),W^{(2)})$ is complete, there exists a limiting measure $\bar \mu \in \cP_2(\mathbb{R}^d)$ to the sequence, i.e., we have
\begin{align*}
    \lim_{t\to \infty} W^{(2)}( P^*_{0,t} \nu_0 , \bar \mu) = 0.
\end{align*}

\textit{The candidate invariant measure $\bar \mu$ has sufficiently high moments.}  The current issue with $\bar \mu \in \cP_2({\bR^d})$ is that we cannot guarantee, via Theorem \ref{Thm:MV Monotone Existence}, that $P^*_{0,t} \bar \mu$ has meaning (although we have convergence in $\cP_2(\mathbb{R}^d)$). Thus, we need to show that $(P^*_{0,t}\nu_0)_{t\geq 0}$ also has the Cauchy-sequence property in $(\cP_{\ell}(\mathbb{R}^d),W^{({\ell})})$ so that $\bar \mu \in \cP_{\ell}({\bR^d})$.  Set $X_0\sim \nu_0 \in  \cP_{2\ell-2}(\bR^d)$, $Y_0 \sim P^*_{0,s}\nu_0\in  \cP_{2\ell-2}(\bR^d)$ for $s\ge 0$, then for any $t\geq 0$ we have via Cauchy–Schwarz inequality 
\begin{align*}
 & \bE \big[    |X_{t}-Y_{t}|^{\ell}   \big]
 =
 \bE \big[    |X_{t}-Y_{t}|~|X_{t}-Y_{t}|^{\ell-1}   \big]
   \leq 
   \sqrt{  \bE \big[    |X_{t}-Y_{t}|^{2}   \big] 
            \bE \big[    |X_{t}-Y_{t}|^{2\ell-2}   \big]  }
    \leq 
    C e^{\rho_2 t/2},
    \end{align*}
where $C$ is uniformly bounded in $t$ and depends on $\nu_0,~P^*_{0,s}\nu_0$ due to \eqref{aux:contratcioninequality3} and \eqref{eq:proof of egordicity:bounded orbit}. 
Therefore, 
\begin{align*}
    \big(W^{(\ell)} (P^*_{0,t} \nu_0,P^*_{0,t+s}\nu_0  ) \big)^{\ell}
    &\leq  
     \bE \big[    |X_{t}-Y_{t}|^{\ell}   \big]
     \leq 
    C e^{\rho_2 t/2}.
\end{align*} 
We are then able to recognize $(P^*_{0,t} \nu_0)_{t\geq 0}$ as a $W^{(\ell)}$ Cauchy-sequence in $\cP_{\ell}(\mathbb{R}^d)$, and by completeness of the space $(\cP_{\ell}(\mathbb{R}^d),W^{(\ell)})$ we conclude that the sequence converges to $\bar \mu \in \cP_{\ell}(\mathbb{R}^d)$.

\textit{Invariance argument.} To show the invariance property, it suffices to argue in $W^{(2)}$. From here, using \cite[Lemma 4.2]{Villani2009oldnew}, we obtain for any $t\geq 0$ that
\begin{align*}
    W^{(2)}( P_{0,t} \bar \mu , \bar \mu)
    \leq 
    \liminf_{s\to \infty}  W^{(2)}\big( P^*_{0,t}( P^*_{0,s} \nu_0  )  , \bar \mu\big)=0.
\end{align*}
We then conclude that $\bar \mu$ is an invariant measure. 

\textit{The ergodicity property of the system.}
The contraction inequality  \eqref{eq:aux.SemigroupDifferenceEstimate} with $\rho_2<0$ yields the exponential ergodicity of the invariant measure $\bar \mu$ in the following sense, 
\begin{align*}
     W^{(2)}\big(P^*_{0,t} \nu_0, \bar \mu \big)
     & =
     \lim_{s\to \infty}
      W^{(2)}\big(P^*_{0,t} \nu_0, P^*_{0,t} (P^*_{0,s} \nu_0) \big)
      \\
    &
    \leq 
     e^{\rho_2 t/2}
     \lim_{s \to \infty}  W^{(2)}\big(\nu_0,P^*_{0,s}\nu_0\big)
=
     e^{\rho_2 t/2}  W^{(2)}\big(\nu_0, \bar \mu \big).
     \end{align*}
Via a straightforward application of the same arguments as above, we have 
\begin{align*}
\textrm{for any } \nu_0\in \cP_{2\ell-2}(\bR^d) \qquad \lim_{t\to \infty} W^{(\ell)}(P_{0,t}^* \nu_0,\bar \mu )=0.
\end{align*}

\end{proof}


\subsection[Proof of the Stochastic C-Stability Lemma]{Proof of Lemma \ref{lemma:bc: diff leq sum}: Stochastic $C$-Stability}
The proof shown in this section is an extension of the results for classical SDEs in \cite{2015ssmBandC} to the particle system considered in this paper.
\begin{proof}
For every $n\in \llbracket 0,M \rrbracket$, we denote the difference of the two particles by
$$
e^{i,N}_{n}:=X^{i,N}_{n}-\hx^{i,N}_{n}.
$$
By the orthogonality of the conditional expectation it holds

\begin{align}
\label{eq:ehi different}
\bE \big[    
 |e^{i,N}_{n} |^{2}\big]
=
\bE \Big[    
\left|\mathbb{E}\big[e^{i,N}_{n} \mid \mathcal{F}_{t_{n-1}}\big]\right|^{2}\Big]
+
\bE \Big[   
\left|e^{i,N}_{n}-\mathbb{E}\big[e^{i,N}_{n} \mid \mathcal{F}_{t_{n-1}}\big]\right|^{2}     \Big].
\end{align}
The  term $e^{i,N}_{n}$ can be expressed as follows  
\begin{align*}
e^{i,N}_{n}=X^{i,N}_{n}
+\Psi_i (X^{i,N}_{n-1},\mu^{X,N}_{n-1}, t_{n-1}, h )
-\Psi_i (X^{i,N}_{n-1},\mu^{X,N}_{n-1}, t_{n-1}, h )
-\hx^{i,N}_{n}.    
\end{align*}
 Thus, for the first term in \eqref{eq:ehi different},   it follows from the inequality $(a+b)^{2}=a^{2}+2 a b+b^{2} \leq$ $\left(1+h^{-1}\right) a^{2}+\left(1+h\right) b^{2}$ that, we have 
\begin{align}
\nonumber
\bE \Big[   &
\left|\mathbb{E}\big[e^{i,N}_{n} \mid \mathcal{F}_{t_{n-1}}\big]\right|^2 \Big] 
\leq 
   (1+\tfrac{1}{h}) \bE \Big[    
\left|\mathbb{E}\big[X^{i,N}_{n}
-\Psi_i (X^{i,N}_{n-1},\mu^{X,N}_{n-1}, t_{n-1}, h ) 
\mid 
\mathcal{F}_{t_{n-1}}\big]\right|^2 \Big]
\\
\label{eq:c-stable result eq1}
& \quad \quad+(1+h) \bE \Big[    
\left|\mathbb{E}\big[\Psi_i (X^{i,N}_{n-1},\mu^{X,N}_{n-1}, t_{n-1}, h )
-\hx^{i,N}_{n}
\mid 
\mathcal{F}_{t_{n-1}}\big]\right|^2\Big].
\end{align}
Similarly, for the second term in \eqref{eq:ehi different}, choose $\eta$ such that $1<\eta\leq (m-1)$ in order {\color{black}   to use \eqref{eq:condition:driftoffsetdiffusion} in  Assumption \ref{Ass:Monotone Assumption}} $(\mathbf{A}^u,~\mathbf{A}^\sigma)$, we have

\begin{align}
\nonumber
 \bE \Big[    
\Big| e^{i,N}_{n} &-\mathbb{E}\big[e^{i,N}_{n} | \mathcal{F}_{t_{n-1}}\big] \Big|^{2} \Big]
\\
\nonumber
\leq 
&   (1+\tfrac{1}{\eta -1})
\bE \Big[    
\Big|
 \big(\mathrm{id}-\mathbb{E}[\cdot \mid \mathcal{F}_{t_{n-1}}] \big)
\Big( 
X^{i,N}_{n}
-\Psi_i (X^{i,N}_{n-1},\mu^{X,N}_{n-1}, t_{n-1}, h ) \Big)
\Big|^2 \Big]
\\
\label{eq:c-stable result eq2}
&+\eta ~
\bE \Big[    
\Big|\big(\mathrm{id}-\mathbb{E} [\cdot \mid \mathcal{F}_{t_{n-1}} ]\big)
\Big( 
\Psi_i (X^{i,N}_{n-1},\mu^{X,N}_{n-1}, t_{n-1}, h ) 
-\hx^{i,N}_{n} \Big)
\Big|^2 \Big].
\end{align}
Using the fact $\hx^{i,N}_{n}=\Psi_i (\hx^{i,N}_{n-1},\hm^{X,N}_{n-1}, t_{n-1}, h ) $, and the definition of $C$-stability for the terms  \eqref{eq:c-stable result eq1}, \eqref{eq:c-stable result eq2} ({\color{black}note that $h\in (0,1)$}) 
  \color{black}  
\begin{align*}
(1+h)& \bE \Big[    
\left|\mathbb{E}\big[\Psi_i (X^{i,N}_{n-1},\mu^{X,N}_{n-1}, t_{n-1}, h )
-\hx^{i,N}_{n}
\mid 
\mathcal{F}_{t_{n-1}}\big]\right|^2\Big]
\\
&
    +\eta ~
    \bE \Big[    
    \Big|\big(\mathrm{id}-\mathbb{E} [\cdot \mid \mathcal{F}_{t_{n-1}} ]\big)
    \Big( 
    \Psi_i (X^{i,N}_{n-1},\mu^{X,N}_{n-1}, t_{n-1}, h ) 
    -\hx^{i,N}_{n} \Big)
    \Big|^2 \Big]
    \\
    &\leq 
    (1+h)\Big( (1+Ch)  \bE \big[    
    |e_{n-1}^{i,N} |^{2} \big]
    +C h \bE \big[     | W^{(2)}(\hm^{X,N}_{n-1},\mu^{X,N}_{n-1} )|^2  \big]\Big).
\end{align*}
\color{black}   
We then further estimate \eqref{eq:ehi different} by 
\begin{align*}
    \bE \big[|e^{i,N}_{n} |^{2} \big]
    &\ \le
    (1+\tfrac{1}{h}) \bE \Big[    
    \left|\mathbb{E}\big[X^{i,N}_{n}
    -\Psi_i (X^{i,N}_{n-1},\mu^{X,N}_{n-1}, t_{n-1}, h) 
    \mid 
    \mathcal{F}_{t_{n-1}}\big]\right|^2  \Big]
    \\
    &\quad   +
    (1+\tfrac{1}{\eta  -1})\bE \Big[    
    \left|
    \left(\mathrm{id}-\mathbb{E}\left[\cdot \mid \mathcal{F}_{t_{n-1}}\right]\right)
    \big( 
    X^{i,N}_{n}
    -\Psi_i (X^{i,N}_{n-1},\mu^{X,N}_{n-1}, t_{n-1}, h ) \big)
    \right|^2  \Big]
    \\
    &\quad +
    (1+Ch)\bE \big[    
    |e_{n-1}^{i,N} |^{2} \big]
    +C h \bE \big[     | W^{(2)}(\hm^{X,N}_{n-1},\mu^{X,N}_{n-1} )|^2  \big].
\end{align*}
Using the fact that the particles are identically distributed
\begin{align*}
 \bE \big[    | W^{(2)}(\hm^{X,N}_{n-1},\mu^{X,N}_{n-1} )|^2 \big]
 \le \frac{1}{N}
 \sum_{j=1}^N \bE[|e_{n-1}^{j,N}|^2] 
 =
  \bE \big[   |e_{n-1}^{i,N} |^{2} \big].
\end{align*}
By induction, with $C_\eta=1+(\eta-1)^{-1}$, we have 
\begin{align*}
 \sup_{n\in \llbracket 0,M \rrbracket} \bE &\big[   
 |X^{i,N}_{n}-\hx^{i,N}_{n} |^{2}  \big]
\leq  \bE \big[    |\hx^{i,N}_{0}-\xi^i|^{2} \big]
\\
&\quad+\sum_{k=1}^{M}\left(1+h^{-1}\right)
\bE \Big[    
\left|\mathbb{E}\big[X^{i,N}_{k}-\Psi_i (X^{i,N}_{k-1},\mu^{X,N}_{k-1}, t_{k-1}, h ) \mid \mathcal{F}_{t_{k-1}}\big]\right|^{2} \Big]
\\
&\quad+C_{\eta} \sum_{k=1}^{M} 
\bE \Big[   
\left|\left(\mathrm{id}-\mathbb{E}\left[\cdot \mid \mathcal{F}_{t_{k-1}}\right]\right)\big(X^{i,N}_{k}-\Psi_i (X^{i,N}_{k-1},\mu^{X,N}_{k-1}, t_{k-1}, h )\big)\right|^{2} \Big]
\\
&\quad+ C h \sum_{k=1}^{M} \bE \big[    |X^{i,N}_{k}-\hx_{k}^{i,N} |^{2}  \big]
+  \frac{Ch}{N}\sum_{k=1}^{M} \sum_{j=1}^N 
\bE \big[    |X^{j,N}_{k}-\hx_{k}^{j,N} |^{2}\big].
\end{align*}
Taking supremum over $ i\in \llbracket 1,N \rrbracket$ and applying the discrete Gronwall's Lemma yields the result.

\end{proof}

\subsection{Proof of Theorem \ref{theorem:bc: convergence rate }}

\begin{proof}
Using Definition \ref{def:bc:def2:C-stable}, Definition \ref{def:bc:def3:B-consist}  and the result in Lemma \ref{lemma:bc: diff leq sum}, we obtain 
\begin{align*}
& \sup_{n\in \llbracket 0,M \rrbracket}\sup_{i\in \llbracket 1,N \rrbracket} \bE \big[    
 |X_{n}^{i,N}-\hx^{i,N}_{n} |^{2} \big]
\leq 
\mathrm{e}^{C T}\Bigg[\sup_{i\in \llbracket 1,N \rrbracket}
\bE \big[    |X_{0}^{i,N}-\hx^{i,N}_{0}|^{2}\big]
\\
&\quad+\sum_{k=1}^{M}\sup_{i\in \llbracket 1,N \rrbracket}  \bigg( 
\left(1+h^{-1}\right)
\bE \Big[    
\left|\mathbb{E}\big[X_{k}^{i,N}-\Psi_i\left(X_{k-1}^{i,N},\mu^{X,N}_{k-1}, t_{k-1}, h\right) \mid \mathcal{F}_{t_{k-1}}\big]\right|^{2} \Big]
\\
&\quad\qquad\qquad+C_{\eta}~ \bE \Big[    
\left|\left(\mathrm{id}-\mathbb{E}\left[\cdot \mid \mathcal{F}_{t_{k-1}}\right]\right)\big(X_{k}^{i,N}-
\Psi_i (X_{k-1}^{i,N},\mu^{X,N}_{k-1}, t_{k-1}, h )\big)\right|^{2} \Big] \bigg) \Bigg]
\\
& 
\qquad \qquad
\le C \mathrm{e}^{C T} 
 ~
\sum_{k=1}^{M} 
\Big( (1+h^{-1})h^{2+2\gamma}
+  C_{\eta} h^{1+2\gamma}
\Big) 
\le C h^{2\gamma},
\end{align*}
where in the second last estimate we used $Mh=T$.
\end{proof}

\subsection{Proof of Theorem \ref{theorem:SSM: convergence all}: Convergence of the SSM scheme} 
\subsubsection[The SSM is C-stable]{The SSM is $C$-stable}
\label{subsubsection: SSM is of C-stable }

We first need to prove  \eqref{eq: def psi is L2}, i.e., $\hx_{n+1}^{i,N} \in L^{2}\big(\Omega, \mathcal{F}_{t_n+h}, \bP  ; \mathbb{R}^{d}\big)$ for all $ n\in \llbracket 0,M-1 \rrbracket$ and $i\in \llbracket 1,N \rrbracket$ given $ \hx_{n}^{i,N}\in L^{2}\big(\Omega, \mathcal{F}_{t_n}, \bP  ; \mathbb{R}^{d}\big)$, where $\hx^{i,N}$ is constructed by the SSM scheme defined in \eqref{eq:SSTM:scheme 1} and \eqref{eq:SSTM:scheme 2}. We first provide the following useful result for the later proof.

\begin{proposition} [Summation relationship]
\label{prop:sum y square leq sum x square}
Let Assumption \ref{Ass:Monotone Assumption}  hold and choose $h$ as in \eqref{eq:h choice}. Then there exists a  constant $C>0$ such that, for all $n\in \llbracket 0,M-1 \rrbracket$,
\begin{align}\label{eq:sum y square leq sum x square}
    \frac{1}{N}\sum_{j=1}^N |Y_{n}^{j,\star,N}|^2 
    &\le  
    Ch+(1+Ch) ~\frac{1}{N}\sum_{j=1}^N  |\hx_{n}^{j,N}|^2. 
\end{align}
\end{proposition}
\begin{proof}
See \cite[Proposition 4.4]{chen2022SuperMeasure}. 
\end{proof}

\begin{proposition}[Second order moment bounds of SSM] 
\label{prop: discrete second moment bound}
Let the setting of Theorem \ref{theorem:SSM: convergence all} hold. Then there exists a constant $C>0$ independent of $h,N,M$ such that 
\begin{align*} 
    \sup_{i\in \llbracket 1,N \rrbracket } \sup_{n\in \llbracket 0,M \rrbracket }
    \bE \big[ |\hx_{n}^{i,N}|^{2}  \big]
    +
     \sup_{i\in \llbracket 1,N \rrbracket } \sup_{n\in \llbracket 0,M-1 \rrbracket }
    \bE \big[ |Y_{n}^{i,\star,N}|^{2}  \big] 
    &\le C \big(  1+ \bE\big[\, |\hx_{0}^{N}|^{2}\big] \big). 
\end{align*}

\end{proposition}

\begin{proof}
    The proof is similar to \cite[Section 4.1]{chen2022SuperMeasure}. By Assumption \ref{Ass:Monotone Assumption},   Proposition \ref{prop:sum y square leq sum x square}  
    , and the fact that the particles are identically distributed, we deduce that there exists a constant $C>0$ such that for any $i\in\llbracket 1,N\rrbracket$, $n\in\llbracket 0,M-1\rrbracket$
    \color{black}  		
    \begin{align*}
    \bE \big[ 1+|Y_{n}^{i,\star,N}|^2 \big] 
    & = \frac1N \sum_{j=1}^N\bE \big[ 1+|Y_{n}^{j,\star,N}|^2 \big] 
    \\& 
    \leq 
    1+Ch+(1+Ch) \frac1N \sum_{j=1}^N\bE \big[ |\hx_{n}^{j,N}|^2 \big]
    \leq 
    (1+Ch) \bE\big[  1+|\hx_{n}^{i,N}|^2 \big] .
\end{align*}
 \color{black}	  %
From  \eqref{eq:SSTM:scheme 1} and Jensen's inequality, we have
\begin{align*}
\nonumber
      |Y_{n}^{i,\star,N}|^2
      &=
      \big\langle Y_{n}^{i,\star,N}, \hx_{n}^{i,N}+h v (Y_{n}^{i,\star,N},\hm^{Y,N}_n ) \big\rangle
      \\
      &\leq 
      \frac{1}{2} |Y_{n}^{i,\star,N}|^2 + \frac12  |\hx_{n}^{i,N}|^2
      +
      h \big\langle Y_{n}^{i,\star,N},   v (Y_{n}^{i,\star,N},\hm^{Y,N}_n ) \big\rangle,
\end{align*}
and hence,
     \begin{align} \label{eq: proof of y2 leq x2+yvh}
          |Y_{n}^{i,\star,N}|^2
      &\leq
      |\hx_{n}^{i,N}|^2 
      + 2 h \big\langle Y_{n}^{i,\star,N},   v (Y_{n}^{i,\star,N},\hm^{Y,N}_n ) \big\rangle.
\end{align}
Also, from \eqref{eq:SSTM:scheme 2} and using the result above, we have 
\begin{align*}
      |\hx_{n+1}^{i,N}|^2    =
      \big|Y_{n}^{i,\star,N} + b(t_n,Y_{n}^{i,\star,N},\hm^{Y,N}_n) h
        +\hs(t_n,Y_{n}^{i,\star,N},\hm^{Y,N}_n) \Delta W_{n}^i  \big|^2.
\end{align*}
Taking expectation on both sides, by Jensen's inequality,  \eqref{eq: proof of y2 leq x2+yvh}, Assumption \ref{Ass:Monotone Assumption} and \eqref{eq: remark eq1}-\eqref{eq: remark eq3} in Remark \ref{remark:ImpliedProperties}, we have 
\begin{align*}
 \bE \big[& 1+|\hx_{n+1}^{i,N}|^2\big]
 \\
 & \leq (1+Ch)
 \bE \big[ 1+|\hx_{n}^{i,N}|^2\big] 
 + h
 \bE \Big[   2 \big\langle Y_{n}^{i,\star,N},   v (Y_{n}^{i,\star,N},\hm^{Y,N}_n ) \big\rangle   
 +      |\hs(t_n,Y_{n}^{i,\star,N},\hm^{Y,N}_n) |^2
 \Big]
 \\
 &\leq   (1+Ch) \bE \big[ 1+|\hx_{n}^{i,N}|^2\big]
 + 2 h \bE \Big[    \big\langle Y_{n}^{i,\star,N},   u (Y_{n}^{i,\star,N},\hm^{Y,N}_n ) \big\rangle   
 +      |\sigma(t_n,Y_{n}^{i,\star,N},\hm^{Y,N}_n) |^2
 \Big]
 \\
 & \quad
 +    \frac{h}{N} \sum_{j=1}^N \bE   \Big[    \big\langle Y_{n}^{i,\star,N} - Y_{n}^{j,\star,N},   f (Y_{n}^{i,\star,N}  - Y_{n}^{j,\star,N} ) \big\rangle   
 +      2|\fs (Y_{n}^{i,\star,N}  - Y_{n}^{j,\star,N} ) |^2
 \Big]
 \\
 &\leq   (1+Ch) \bE \big[ 1+|\hx_{n}^{i,N}|^2\big]
 + C h \bE \Big[ 1+ |Y_{n}^{i,\star,N}|^2+ \frac{1}{N} \sum_{j=1}^N|Y_{n}^{j,\star,N}|^2  \Big]
  \\
 &\quad + \frac{Ch}{N} \sum_{j=1}^N \bE \Big[  |Y_{n}^{i,\star,N}-Y_{n}^{j,\star,N}|^2  \Big] 
  \\
 &\leq   (1+Ch) \bE \big[ 1+|\hx_{n}^{i,N}|^2\big]
 + C h \bE \Big[ 1+ |Y_{n}^{i,\star,N}|^2+ \frac{1}{N} \sum_{j=1}^N|Y_{n}^{j,\star,N}|^2  \Big]
 \\
 &\leq   (1+Ch) \bE \big[ 1+|\hx_{n}^{i,N}|^2\big].
\end{align*} 
\end{proof}
Proposition \ref{prop: discrete second moment bound}  shows that the one-step map of the SSM, $\Psi=(\Psi_1,\ldots,\Psi_N) $ in Definition \ref{def:bc:def1:scheme}, $ \Psi_i(\hx_{n}^{i,N},\hm^{X,N}_n, t_n, h)=\hx_{n+1}^{i,N}$ is indeed an $L^2$-operator. We now prove the SSM is $C$-stable.
\begin{proof}[Proof of statement 1 in Theorem \ref{theorem:SSM: convergence all}]

We use \eqref{eq:SSTM:scheme 1} and \eqref{eq:SSTM:scheme 2} to define the mapping $\Psi=(\Psi_1,\ldots,\Psi_N)$  and consequently to generate the following two processes $\hx_{n}^{i,N}$ and $\hz_{n}^{i,N}$
for all $i\in \llbracket 1,N \rrbracket$, $n\in \llbracket 0,M-1 \rrbracket $, with the corresponding empirical measures $\hm^{X,N}_n,~\hm^{Z,N}_n \in \cP_2(\bR^d)$ and $\Delta W_{n}^i= W_{t_{n+1}}^i-W_{t_{n}}^i$
\begin{align*}
Y_{n}^{i,X,N} &=\hx_{n}^{i,N}+h v (Y_{n}^{i,X,N},\hm^{Y,X,N}_n ),  
\quad \quad 
  \hm^{Y,X,N}_n(\dd x):= \frac1N \sum_{j=1}^N \delta_{Y_{n}^{j,X,N}}(\dd x)
  ,
 \\ 
\hx_{n+1} ^{i,N} &=Y_{n}^{i,X,N}
            + b(t_n,Y_{n}^{i,X,N},\hm^{Y,X,N}_n) h
            +\hs(t_n,Y_{n}^{i,X,N},\hm^{Y,X,N}_n) \Delta W_{n}^i,
\\
Y_{n}^{i,Z,N} &=\hz_{n}^{i,N}+h v (Y_{n}^{i,Z,N},\hm^{Y,Z,N}_n ),  
\quad \quad
  \hm^{Y,Z,N}_n(\dd x):= \frac1N \sum_{j=1}^N \delta_{Y_{n}^{j,Z,N}}(\dd x),
 \\ 
\hz_{n+1}^{i,N} &=Y_{n}^{i,Z,N}
            + b(t_n,Y_{n}^{i,Z,N},\hm^{Y,Z,N}_n) h
            +\hs(t_n,Y_{n}^{i,Z,N},\hm^{Y,Z,N}_n) \Delta W_{n}^i.
\end{align*}
Thus, $\hx_{n+1}^{i,N}= \Psi_i(\hx_{n}^{i,N},\hm^{X,N}_n, t_n, h)$ and $\hz_{n+1}^{i,N}= \Psi_i(\hz_{n}^{i,N},\hm^{Z,N}_n, t_n, h)$.
We need to prove
\begin{align}
\notag 
\bE \Big[    
\Big|\mathbb{E}\big[ & \Psi_i(\hx_{n}^{i,N},\hm^{X,N}_n, t_n, h)-\Psi_i(\hz_{n}^{i,N},\hm^{Z,N}_n, t_n, h) \mid \mathcal{F}_{t_n}\big]\Big|^{2} \Big]
\\
\label{eq:c-stable proof: second term }
+&
\eta~\bE \Big[    \left|\left(\mathrm{id}-\mathbb{E}\left[\cdot \mid \mathcal{F}_{t}\right]\right)
(\Psi_i(\hx_{n}^{i,N},\hm^{X,N}_n, t_n, h)-\Psi_i(\hz_{n}^{i,N},\hm^{Z,N}_n, t_n, h))\right|^{2} \Big]
\\
\leq&
\left(1+C h\right) \bE \big[    
|\hx_{n}^{i,N}-\hz_{n}^{i,N}|^{2} \big]
+Ch \bE \big[   | W^{(2)}(\hm^{X,N}_n,\hm^{Z,N}_n )|^2 \big]. \nonumber
\end{align}
For the first term in \eqref{eq:c-stable proof: second term }, {\color{black} note that the Brownian motion $W_t$ is $\cF_t$-measurable, }using the Lipschitz continuity of $b$, we get 
\begin{align*}  
 \bE \Big[    &
\left|\mathbb{E}\big[\Psi_i(\hx_{n}^{i,N},\hm^{X,N}_n, t_n, h)-\Psi_i(\hz_{n}^{i,N},\hm^{Z,N}_n, t_n, h) \mid \mathcal{F}_{t_n}\big]\right|^{2} \Big]
\\
& =   \bE \big[        
\big|
Y_{n}^{i,X,N}+ b(t_n,Y_{n}^{i,X,N},\hm^{Y,X,N}_n) h
-Y_{n}^{i,Z,N}- b(t_n,Y_{n}^{i,Z,N},\hm^{Y,Z,N}_n) h
\big| ^{2}\big]
\\
& \le (1+Ch) \bE \big[        
 |
Y_{n}^{i,X,N}-Y_{n}^{i,Z,N}
 | ^{2}\big] 
+
Ch\bE \big[  
  |   
 W^{(2)}(\hm^{Y,X,N}_n,\hm^{Y,Z,N}_n )
  |^{2} \big].
\end{align*}
\color{black}
From Lemma \ref{remark:OSL for the whole function / system V}, we observe that 
\begin{align*}
|&
Y_{n}^{i,X,N}-Y_{n}^{i,Z,N}|^2 
\\
\quad &
= \big\langle 
Y_{n}^{i,X,N}-Y_{n}^{i,Z,N}, \hx_{n}^{i,N}-\hz_{n}^{i,N}+
v(Y_{n}^{i,X,N},\hm^{Y,X,N}_n) h-v(Y_{n}^{i,Z,N},\hm^{Y,Z,N}_n ) h
\big\rangle 
\\\nonumber
\quad &\le \frac{1}{2} \big( |Y_{n}^{i,X,N}-Y_{n}^{i,Z,N}|^2 +  |\hx_{n}^{i,N}-\hz_{n}^{i,N}|^2 \big)
\\\nonumber
&\qquad + h\big\langle 
Y_{n}^{i,X,N}-Y_{n}^{i,Z,N}, 
v(Y_{n}^{i,X,N},\hm^{Y,X,N}_n) -v(Y_{n}^{i,Z,N},\hm^{Y,Z,N}_n )  
\big\rangle,
\nonumber
\end{align*}
and therefore
\begin{align}
\nonumber 
|&
Y_{n}^{i,X,N}-Y_{n}^{i,Z,N}|^2  
\\ \nonumber 
&\quad
\le |\hx_{n}^{i,N}-\hz_{n}^{i,N}|^2  + 2h\big\langle 
Y_{n}^{i,X,N}-Y_{n}^{i,Z,N}, 
v(Y_{n}^{i,X,N},\hm^{Y,X,N}_n) -v(Y_{n}^{i,Z,N},\hm^{Y,Z,N}_n )  
\big\rangle \nonumber
\\ \nonumber 
&\quad \le
|\hx_{n}^{i,N}-\hz_{n}^{i,N}|^2 
\\ \nonumber 
&\qquad + 2h \big\langle 
Y_{n}^{i,X,N}-Y_{n}^{i,Z,N}, 
\frac{1}{N} \sum_{j=1}^N \big( f(Y_{n}^{i,X,N}-Y_{n}^{j,X,N} )- f(Y_{n}^{i,Z,N}-Y_{n}^{j,Z,N} )  \big) 
\big\rangle  \nonumber
\\
&\qquad +
2h\big\langle 
Y_{n}^{i,X,N}-Y_{n}^{i,Z,N}, 
u(Y_{n}^{i,X,N},\hm^{Y,X,N}_n)  -u(Y_{n}^{i,Z,N},\hm^{Y,Z,N}_n ) 
\big\rangle.  \label{eq:c-stable:y-y: u}
\end{align}
\color{black}  	
For  the second term in \eqref{eq:c-stable proof: second term }, by Jensen's inequality, we have
\begin{align*}    	
&
 \bE \Big[        
\left|\left(\mathrm{id}-\mathbb{E}\left[\cdot \mid \mathcal{F}_{t_n}\right]\right)
(\Psi_i(\hx_{n}^{i,N},\hm^{X,N}_n, t_n, h)-\Psi_i(\hz_{n}^{i,N},\hm^{Z,N}_n, t_n, h))\right|^{2}\Big]
\\
&= \bE \big[       
|  
\hs(t_n,Y_{n}^{i,X,N},\hm^{Y,X,N}_n) \Delta W_{n}^i
-
\hs(t_n,Y_{n}^{i,Z,N},\hm^{Y,Z,N}_n) \Delta W_{n}^i
|^2  \big]
\\
&\leq 
2h \bE \Big[       
|  
\sigma(t_n,Y_{n}^{i,X,N},\hm^{Y,X,N}_n)  
-
\sigma(t_n,Y_{n}^{i,Z,N},\hm^{Y,Z,N}_n)  
|^2   
\\
&\quad +
\frac{1}{N} \sum_{j=1}^N | \fs(Y_{n}^{i,X,N}-Y_{n}^{j,X,N} )- \fs(Y_{n}^{i,Z,N}-Y_{n}^{j,Z,N} )  |^2
\Big].
\end{align*}
From Assumption \ref{Ass:Monotone Assumption} 
 and \eqref{eq: remark eq1}  in Remark \ref{remark:ImpliedProperties}, we derive, for some $\eta>1$,
\begin{align}
\nonumber
   &\bE\Big[ \big\langle 
Y_{n}^{i,X,N}-Y_{n}^{i,Z,N}, 
\frac{1}{N} \sum_{j=1}^N \big( f(Y_{n}^{i,X,N}-Y_{n}^{j,X,N} )- f(Y_{n}^{i,Z,N}-Y_{n}^{j,Z,N} )  \big) \big\rangle\Big]
\\\nonumber
& \quad + \eta
\bE \Big[       
\frac{1}{N} \sum_{j=1}^N | \fs(Y_{n}^{i,X,N}-Y_{n}^{j,X,N} )- \fs(Y_{n}^{i,Z,N}-Y_{n}^{j,Z,N} )  |^2
\Big]
\\\nonumber
& =
\frac{1}{2 N^2}  \sum_{i=1}^N \sum_{j=1}^N\bE\Big[ \big\langle 
(Y_{n}^{i,X,N}-Y_{n}^{j,X,N} )- (Y_{n}^{i,Z,N}-Y_{n}^{j,Z,N} ), 
\\\nonumber
& \qquad \qquad \qquad \qquad \qquad 
   f(Y_{n}^{i,X,N}-Y_{n}^{j,X,N} )- f(Y_{n}^{i,Z,N}-Y_{n}^{j,Z,N} )  \big\rangle \Big]
\\\nonumber
& \quad  + \frac{\eta }{N^2} \sum_{i=1}^N \sum_{j=1}^N
\bE \Big[       
  | \fs(Y_{n}^{i,X,N}-Y_{n}^{j,X,N} )- \fs(Y_{n}^{i,Z,N}-Y_{n}^{j,Z,N} )  |^2
\Big]
   \\ 
   \nonumber 
&\leq 
\frac{1}{2 N^2}  \sum_{i=1}^N \sum_{j=1}^N\bE\Big[  L^{(1)}_{(f)} \big|(Y_{n}^{i,X,N}-Y_{n}^{j,X,N} )- (Y_{n}^{i,Z,N}-Y_{n}^{j,Z,N} )\big|^2 \Big]
 \\ 
   \label{eq: c-stable y-y f-f}
&\leq  2L^{(1),+}_{(f)} \bE\big[  |Y_{n}^{i,X,N}-Y_{n}^{i,Z,N}|^2 \big].
\end{align}
\color{black}
Collecting the above estimates and using \eqref{eq: remark eq4}-\eqref{eq: remark eq5} in Remark \ref{remark:ImpliedProperties}, we have 
\begin{align}
&  \bE \Big[      
\left|\mathbb{E}\big[\Psi_i(\hx_{n}^{i,N},\hm^{X,N}_n, t_n, h)-\Psi_i(\hz_{n}^{i,N},\hm^{Z,N}_n, t_n, h) \mid \mathcal{F}_{t}\big]\right|^{2}  \Big] 
\nonumber
\\
& \quad +    
\eta
\bE \Big[     \left|\left(\mathrm{id}-\mathbb{E}\left[\cdot \mid \mathcal{F}_{t}\right]\right)
\big(\Psi_i(\hx_{n}^{i,N},\hm^{X,N}_n, t_n, h)-\Psi_i(\hz_{n}^{i,N},\hm^{Z,N}_n, t_n, h)\big)\right|^{2} \Big]
\nonumber
\\
\leq& \bE \Big[        
|\hx_{n}^{i,N}-\hz_{n}^{i,N}|^2 + 
4L^{(1),+}_{(f)} h|Y_{n}^{i,X,N}-Y_{n}^{i,Z,N}|^2  
\nonumber 
\\
&\quad +  	2\eta h   
|  
\sigma(t_n,Y_{n}^{i,X,N},\hm^{Y,X,N}_n)  
-
\sigma(t_n,Y_{n}^{i,Z,N},\hm^{Y,Z,N}_n)  
|^2   
\nonumber
\\
&\quad+ 
2h\langle 
Y_{n}^{i,X,N}-Y_{n}^{i,Z,N}, 
u(Y_{n}^{i,X,N},\hm^{Y,X,N}_n)  -u(Y_{n}^{i,Z,N},\hm^{Y,Z,N}_n ) 
\rangle \Big] (1+Ch)
\nonumber
\\
\nonumber 
\le
& (1+Ch) \bE \big[  |\hx_{n}^{i,N}-\hz_{n}^{i,N}|^2      \big]
\\
\label{eq:proof 1 final result}
&\qquad \qquad 
+ Ch \left(
\bE \big[ |Y_{n}^{i,X,N}-Y_{n}^{i,Z,N}|^2      \big] + 
\bE \big[ | W^{(2)}(\hm^{Y,X,N}_n, \hm^{Y,Z,N}_n)|^2      \big] \right),
\end{align}
\color{black} 
where we used that the particles are identically distributed and the following inequality:  for $\eta\in \big(1,2(m-1)\big)$, we have 
\color{black} 
\begin{align}
\nonumber
  \bE \Big[  &       \langle 
Y_{n}^{i,X,N}-Y_{n}^{i,Z,N}, 
u(Y_{n}^{i,X,N},\hm^{Y,X,N}_n)  -u(Y_{n}^{i,Z,N},\hm^{Y,Z,N}_n ) 
\rangle 
\\
\nonumber 
&\qquad \qquad \qquad+ 
\eta h   
|  
\sigma(t_n,Y_{n}^{i,X,N},\hm^{Y,X,N}_n)  
-
\sigma(t_n,Y_{n}^{i,Z,N},\hm^{Y,Z,N}_n)  
|^2   \Big] 
\\
\label{eq:se1:y-y}%
&\le    (L^{(1)}_{(u\sigma)} +  L^{(2)}_{(u\sigma)} ) \bE \big[ |Y_{n}^{i,X,N}-Y_{n}^{i,Z,N}|^2  \big] .
\end{align}
Substituting the estimates from above   into \eqref{eq:c-stable:y-y: u}, and take  Remark \ref{remark: choice of h} into account, we get 
\begin{align*}%
    \bE \big[ |Y_{n}^{i,X,N}-Y_{n}^{i,Z,N}|^2  \big] 
    \le (1+Ch)
    \bE \big[|\hx_{n}^{i,N}-\hz_{n}^{i,N}|^2   \big] .
\end{align*}
Further, we note that 
\begin{align*}    	
\bE \big[    |    
 W^{(2)}(\hm^{Y,X,N}_n,\hm^{Y,Z,N}_n )
 |^{2}\big]
\le    
\frac{1}{N} \sum_{j=1}^N \bE [| Y_{n}^{j,X,N}-&Y_{n}^{j,Z,N} |^2]  
\le (1+Ch)
\bE \big[     
|\hx_{n}^{i,N}-\hz_{n}^{i,N} |^2\big].
\end{align*}
Substituting these estimates in \eqref{eq:proof 1 final result}, allows one to deduce the claim.
\end{proof}

\subsubsection[The SSM is B-consistent]{The SSM is $B$-consistent}
\label{section: proof of B-consistent}
We first state the following auxiliary results and recall  that the constant $C$  is  positive and independent of $h,N,M$.
\begin{proposition} [Difference relationship]
\label{prop:yi-yj leq xi-xj}
Let Assumption \ref{Ass:Monotone Assumption} hold and choose $h$ as in \eqref{eq:h choice}. For any $n\in \llbracket 0,M \rrbracket$, let $Y^{\star,N}_n$ defined as in \eqref{eq:SSTM:scheme 0} and \eqref{eq:SSTM:scheme 1}. Then, there exists a constant $C>0$ such that for all $i,~j\in \llbracket 1,N \rrbracket$,  
\begin{align}\label{eq:prop:yi-yj leq xi-xj}
    |Y_{n}^{i,\star,N}-Y_{n}^{j,\star,N}|^2 
    &\le 
     \frac{1}{1-2(L^{(1)}_{(f)}+L^{(1)}_{(u\sigma)})h}  | \hx_{n}^{i,N}-\hx_{n}^{j,N}|^2  \le (1+Ch) | \hx_{n}^{i,N}-\hx_{n}^{j,N}|^2.
\end{align}
\end{proposition}

\begin{proof}
See \cite[Proposition 4.3]{chen2022SuperMeasure}.
\end{proof}

Now, we state the following moment relationship for the first  step of the SSM.

\begin{proposition} [Moment relationship]
\label{prop:  y  leq   x }
Let Assumption \ref{Ass:Monotone Assumption}  hold and choose $h$ as in \eqref{eq:h choice}, then there exist a constant $C>0$ independent of $N$,   
such that for all $i\in \llbracket 1,N \rrbracket,~n\in \llbracket 0,M \rrbracket,~p\geq 1 $ we have 
\begin{align*} 
   \bE\big[|Y_n^{i, \star ,N}|^{2p}\big] 
    &\le C \Big( 
      \frac{1}{N} \sum_{j=1}^N \bE[|X_{n}^{i,N}- X_{n}^{j,N} |^{2p} ]
    +
      \bE \Big[     \Big|\frac{1}{N}\sum_{j=1}^N ( 1+  |X_{n}^{j,N}|^2) \Big|^{p}     \Big]+1 \Big). 
\end{align*}
\end{proposition}

\begin{proof}
By Young's inequality and Jensen's inequality 
\begin{align*}
     \bE\big[ |Y_{n}^{i,\star,N} |^{2p}  \big]  
    &\le  \bE \Big[ \Big|      \frac{1}{N} \sum_{j=1}^N 
    \Big(
    2|Y_{n}^{i,\star,N}- Y_{n}^{j,\star,N} |^{2}  
    +
    2 |Y_{n}^{j,\star,N}|^{2}    \Big) \Big|^p\Big]
    \\ 
    &\le \frac{4^p}{N}  \sum_{j=1}^N\bE \big[        |Y_{n}^{i,\star,N}- Y_{n}^{j,\star,N} |^{2p} \big]
    +
    4^p \bE \Big[     \Big| \frac{1}{N} \sum_{j=1}^N    |Y_{n}^{j,\star,N}|^{2}    \Big|^{p}  \Big].
\end{align*}
Combining Propositions \ref{prop:yi-yj leq xi-xj} and \ref{prop:sum y square leq sum x square} allows to conclude the claim.

\end{proof}

The main goal of this section is to prove that $X_{t}^{i,N}$ defined by \eqref{Eq:MV-SDE Propagation} satisfies for all $t\in[0,T]$, $i\in \llbracket 1,N\rrbracket$  the following estimates with $\gamma=1/2$. 
\begin{align}
\label{eq:b-consistency:term 1}
\bE \Big[    
\left|\mathbb{E}\big[X_{t+h}^{i,N}-\Psi_i(X_t^{i,N},\mu^{X,N}_{t}, t, h) \mid \mathcal{F}_{t}\big]\right|^2 \Big]
&\leq C h^{2\gamma+2},
\\
\label{eq:b-consistency:term 2}
\quad\bE \Big[    \left|\left(\mathrm{id}-\mathbb{E}\left[\cdot \mid \mathcal{F}_{t}\right]\right)(X_{t+h}^{i,N}-\Psi_i(X_t^{i,N},\mu^{X,N}_{t}, t, h))\right|^2 \Big]
&\leq C  h^{2\gamma+1}.
\end{align}
 
\begin{proof}[Proof of statement 2 in Theorem \ref{theorem:SSM: convergence all}]
Recall \eqref{Eq:MV-SDE Propagation} and the SSM given in \eqref{eq:SSTM:scheme 0}-\eqref{eq:SSTM:scheme 2}. Then, we introduce the following quantities, for all $t\in[0,T]$, $i\in \llbracket 1,N\rrbracket$,  
\begin{align}
    &X_{t+h}^{i,N}=X_{t}^{i,N}+ \int_t^{t+h} \Big( v( X_{s}^{i,N},\mu_s^{X,N})
    +b(s,X_{s}^{i,N},\mu_s^{X,N}) \Big) \dd s
    +
    \int_t^{t+h} \hs(s,X_{s}^{i,N},\mu_s^{X,N}) \dd W_s^i,
    \label{eq:b-consistency:xt+s integration form}
    \\
    \label{eq:def of Y i N from solution X}
    &Y_t^{i,N}
    =X_{t}^{i,N}+v( Y_t^{i,N},\mu_t^{Y,N})h,
    \qquad
    \mu_t^{Y,N}(\dd x):= \frac1N \sum_{j=1}^N \delta_{Y_t^{j,N}}(\dd x),
    \\
    \nonumber
    &\Psi_i(X_t^{i,N},\mu^{X,N}_{t}, t, h)
    \\
    \nonumber
    &\quad =
    X_{t}^{i,N}+
    \int_t^{t+h} \Big( v(  Y_t^{i,N},\mu_t^{Y,N})
    +b(t, Y_t^{i,N},\mu_t^{Y,N}) \Big) \dd s
    +
    \int_t^{t+h} \hs(t, Y_t^{i,N},\mu_t^{Y,N}) \dd W_s^i,
\end{align}
where the last equation is the integration form for the one-step map of SSM. 
Therefore, the first term \eqref{eq:b-consistency:term 1} can be estimated by Jensen's inequality

\begin{align}
\label{eq:b-consistency:first block}
 \bE \Big[    
\big|\mathbb{E}\big[X_{t+h}^{i,N}
&
-\Psi_i(X_t^{i,N},\mu^{X,N}_{t}, t, h) \mid \mathcal{F}_{t}\big]\big|^2 \Big]
\\
\label{eq:b-consistency:first block:v-v term}
\leq& 2h
\int_t^{t+h} \bE \big[   |
    v( X_{s}^{i,N},\mu_s^{X,N})
    - v(  Y_t^{i,N},\mu_t^{Y,N}) |^2 \big]~ \dd s
\\ 
\nonumber
&\qquad \qquad + 2h 
\int_t^{t+h} \bE \big[    |
    b(s,X_{s}^{i,N},\mu_s^{X,N})
    -b(t, Y_t^{i,N},\mu_t^{Y,N}) |^2\big]~ \dd s.
\end{align}
For the second term \eqref{eq:b-consistency:term 2}, we get 
\begin{align}
\nonumber
    &\quad \bE \Big[   
    \left|\left(\mathrm{id}-\mathbb{E}\big[\cdot \mid \mathcal{F}_{t}\big]\right)
    (X_{t+h}^{i,N}-\Psi_i(X_t^{i,N},\mu^{X,N}_{t}, t, h))\right|^2 \Big] 
    \\ 
    \label{eq:b-consistency:second block}
    & \qquad \quad 
    \le C \int_t^{t+h} \bE \big[    |
    \hs(s,X_{s}^{i,N},\mu_s^{X,N})
    -\hs(t, Y_t^{i,N},\mu_t^{Y,N}) |^2 \big]~ \dd s   .
\end{align}
By Young's inequality and Jensen's inequality, Assumption \ref{Ass:Monotone Assumption} and Proposition \ref{prop:yi-yj leq xi-xj},  for $s\in[t,t+h]$, we have 
\begin{align*}
     |&X_{s}^{i,N}-Y_t^{i,N} |^2 
     \le
     2  |X_{s}^{i,N}-X_{t}^{i,N} |^2+
     2 |X_{t}^{i,N}-Y_t^{i,N} |^2,
     \\
     |&X_{t}^{i,N}-Y_t^{i,N} |^2
     =|v(  Y_t^{i,N},\mu_t^{Y,N}) h |^2
     \le
     \frac{2h^2}{N} \sum_{j=1}^N |f(  Y_t^{i,N}-Y_t^{j,N})|^2
     + 2h^2| u(  Y_t^{i,N},\mu_t^{Y,N})  |^2
     \\
     &\quad \le 
     \frac{C h^2   }{N} \sum_{j=1}^N 
     \Big( 1+  |Y_t^{i,N}-Y_t^{j,N}|^{2q+2}\Big)
     +Ch^2\Big(1+|Y_t^{i,N}|^{2q+2}+ \frac{1}{N} \sum_{j=1}^N |Y_t^{j,N}|^2\Big)
    \\
      &\quad \le 
     \frac{C h^2  }{N} \sum_{j=1}^N 
     \Big( 1+  |X_t^{i,N}-X_t^{j,N}|^{2q+2}\Big) +Ch^2\Big(1+|Y_t^{i,N}|^{2q+2}+ \frac{1}{N} \sum_{j=1}^N |Y_t^{j,N}|^2\Big).
\end{align*}
Similarly,  we have 
\begin{align*} 
    \nonumber
     |X_{s}^{i,N}&-Y_t^{i,N} |^4 
     \le
     16  |X_{s}^{i,N}-X_{t}^{i,N} |^4+
     16 |X_{t}^{i,N}-Y_t^{i,N} |^4,
     \\
     \nonumber
     |X_{t}^{i,N}&-Y_t^{i,N} |^4
     \\
     &\le
      C h^4 \Big( 1+  |Y_t^{i,N}|^{4q+4}  + \frac{1 }{N} \sum_{j=1}^N |Y_t^{j,N}|^{4}\Big) 
      +
     \frac{C h^4  }{N} \sum_{j=1}^N 
     \Big( 1+  |X_t^{i,N}-X_t^{j,N}|^{4q+4}\Big).
\end{align*}
Using the moment stability of $X^{i,N}$ (note $m > 4q+4 >\max\{2(q+1),4\}$) and Jensen's inequality, we get 
\begin{align*}
    \frac{C h^2  }{N} \sum_{j=1}^N \bE \Big[ 
     \Big( 1+  |X_t^{i,N}-X_t^{j,N}|^{2q+2}\Big)    \Big]
     &\leq
     Ch^2, 
     \\
    \frac{C h^4  }{N} \sum_{j=1}^N \bE \Big[   \Big| 
     \Big( 1+  |X_t^{i,N}-X_t^{j,N}|^{2q+2}  \Big) \Big|^2   \Big]
     &\leq
     Ch^4.
\end{align*}
By \eqref{eq:b-consistency:xt+s integration form} and another application of Jensen's inequality
\begin{align*}
   \bE \big[     |X_{s}^{i,N}-X_{t}^{i,N} |^2 \big]
     \le&
    Ch \int_t^{s} \bE \big[    |v( X_{u}^{i,N},\mu_u^{X,N})
    +b(u,X_{u}^{i,N},\mu_u^{X,N})|^2 \big] \dd u
    \\
    & + C \int_t^{s}\bE \big[     |\hs(u,X_{u}^{i,N},\mu_u^{X,N})|^2 \big]\dd u ~\le C h.
\end{align*}
Similarly, we have  
\begin{align*}
    \bE \big[    |X_{s}^{i,N}-X_t^{i,N} |^4     \big]
     & 
    \le C h^2.
\end{align*}
Using the above results and we have sufficient moment bounds for $ Y_t^{i,N}$ from Proposition \ref{prop:  y  leq   x }, we conclude that
\begin{align*}
     \bE \big[    |X_{s}^{i,N}-Y_t^{i,N} |^2     \big] 
      &\le Ch, \quad
     \bE \big[    |X_{s}^{i,N}-Y_t^{i,N} |^4     \big] 
     \le Ch^2,
     \\
      \bE \big[    | W^{(2)}(\mu_s^{X,N},\mu_t^{Y,N} )  |^2 \big]
      &\le
      \frac{1}{N} \sum_{j=1}^N   \bE \big[    
      | X_{s}^{j,N}-Y_t^{j,N}  |^2 \big] 
      \le
      Ch.
\end{align*}

Thus, for the term \eqref{eq:b-consistency:first block:v-v term}, taking  Assumption \ref{Ass:Monotone Assumption} into account, following the arguments in \cite[Section 4.2]{chen2022SuperMeasure}, Jensen's inequality, Cauchy-Schwarz inequality and Young's inequality yield   
\begin{align*}
    \nonumber
     \bE \big[ &  | 
    v( X_{s}^{i,N},\mu_s^{X,N})
    - v(  Y_t^{i,N},\mu_t^{Y,N}) |^2  \big] 
      \\
       &\le   C\sqrt{
       \bE \big[    1+|X_{s}^{i,N}|^{4q}    +|Y_{t}^{i,N}|^{4q}         \big]
       \bE \big[  |X_{s}^{i,N}-Y_{t}^{i,N}|^{4}     \big]
      } 
      +C \bE \big[  |X_{s}^{i,N}-Y_{t}^{i,N}|^2      \big] 
      \le Ch.
\end{align*}
Also, from Assumption \ref{Ass:Monotone Assumption}, we have 
\begin{align*}
    \bE \big[   |&
    b(s,X_{s}^{i,N},\mu_s^{X,N})
    -b(t, Y_t^{i,N},\mu_t^{Y,N}) |^2 \big]
    \\
    &\le C\big(       h+   \bE \big[   |X_{s}^{i,N}-Y_t^{i,N} |^2 \big]
    + \bE \big[    
     | W^{(2)}(\mu_s^{X,N},\mu_t^{Y,N} )  |^2 \big] \big)
     \le C h,
\end{align*}
and similarly, from Jensen's inequality and Assumption \ref{Ass:Monotone Assumption}, we have  
\color{black}  	 
\begin{align*}
    \bE \big[    |&
    \hs(s,X_{s}^{i,N},\mu_s^{X,N})
    -\hs(t, Y_t^{i,N},\mu_t^{Y,N}) |^2 \big]
     \\
     &\leq C \bE \Big[  h+ \big(1+|X_{s}^{i,N}|^{2q}+|Y_t^{i,N}|^{2q}      \big) |X_{s}^{i,N}-Y_t^{i,N} |^2  + 
    \frac{1}{N} \sum_{j=1}^N |X_{s}^{j,N}-Y_t^{j,N} |^{2q+2}  \Big]
    \le C h.
\end{align*}
\color{black}
Substituting the results above back to  \eqref{eq:b-consistency:first block} and \eqref{eq:b-consistency:second block}, we have  
\begin{align*}
    \bE \Big[     \left|\mathbb{E}\big[X_{t+h}^{i,N}-\Psi_i(X_t^{i,N},\mu^{X,N}_{t}, t, h) \mid \mathcal{F}_{t}\big]\right|^2 \Big]
    &\leq C h
    \int_t^{t+h}
         h  \dd s \le Ch^{3},
    \\
    \quad\bE \Big[    
    \left|\left(\mathrm{id}-\mathbb{E}\left[\cdot \mid \mathcal{F}_{t}\right]\right)(X_{t+h}^{i,N}-\Psi_i(X_t^{i,N},\mu^{X,N}_{t}, t, h))\right|^2 \Big] 
    &  \le C  \int_t^{t+h} 
    h  \dd s    
    \le C h^2.
\end{align*}
\end{proof}

\subsubsection{Proof of convergence for the SSM scheme} 
\begin{proof}[Proof of statement 3 in Theorem \ref{theorem:SSM: convergence all}]
At last, we will prove the third statement in Theorem \ref{theorem:SSM: convergence all}. By combining the first two statements and Theorem \ref{theorem:bc: convergence rate }, we first have
    \begin{align}
    \label{eq: conv result, discrete}
     \sup_{n\in \llbracket 0,M \rrbracket}\sup_{i\in \llbracket 1,N \rrbracket} 
  \bE\big[\,  |X_{n}^{i,N}-\hx_{n}^{i,N} |^2 \big]    &   \le Ch.
    \end{align}
 Now, we extend the strong convergence rate to the continuous time version of the SSM, which has not been discussed in \cite{2015ssmBandC}. In order to extend the result above to the continuous extension of the SSM, we consider, for all $ n\in \llbracket 0,M-1 \rrbracket$,  $i\in \llbracket 1,N \rrbracket$, $r\in[0,h]$, 
\begin{align}
    |X_{t_{n}+r}^{i,N}-\hx_{t_{n}+r}^{i,N}|^2 
    = &
    \Big| 
    X_{t_{n}}^{i,N}-\hx_{n}^{i,N} 
    +
    \int_{t_n}^{t_{n}+r}  \big( v(X_{s}^{i,N},\mu_{s}^{X,N})-v(Y_n^{i,N},\mu^{Y,N}_{n})   
    \big) \dd s  
    \label{eq: continuous conv t1}
    \\ 
\nonumber     
    &+ 
    \int_{t_n}^{t_{n}+r}  \big( b(s,X_{s}^{i,N},\mu_{s}^{X,N})-b(t_n,Y_n^{i,N},\mu^{Y,N}_{n})   
    \big) \dd s
    \\
    &
    \label{eq: continuous conv t5}
    + 
    \int_{t_n}^{t_{n}+r}  \big( \hs(s,X_{s}^{i,N},\mu_{s}^{X,N})
    -\hs(t_n,Y_n^{i,N},\mu^{Y,N}_{n})   
    \big) \dd W^i_s
    \\ \nonumber
    &+ 
    \int_{t_n}^{t_{n}+r}  \big( v(Y_n^{i,N},\mu^{Y,N}_{n}) -v (Y_{n}^{i,\star,N},\hm^{Y,N}_n )
    \big) \dd s 
    \\
    \nonumber
        &+ \int_{t_n}^{t_{n}+r}  \big(  b(t_n,Y_n^{i,N},\mu^{Y,N}_{n})   
        - b(t_n,Y_{n}^{i,\star,N},\hm^{Y,N}_n)
    \big) \dd s
    \\
    \nonumber
        &
        + \int_{t_n}^{t_{n}+r}  \big(  \hs(t_n,Y_n^{i,N},\mu^{Y,N}_{n})   
        - \hs(t_n,Y_{n}^{i,\star,N},\hm^{Y,N}_n)
    \big)  \dd W^i_s \Big|^2 , 
\end{align}
where $Y_n^{i,N}=Y_{t_n}^{i,N}$, $\mu^{Y,N}_{n}=\mu^{Y,N}_{t_n}$ are defined in \eqref{eq:def of Y i N from solution X}. Taking expectation on both sides and using Jensen's inequality, we derive 
\begin{align*}
    \bE \big[& |X_{t_{n}+r}^{i,N}-\hx_{t_{n}+r}^{i,N}|^2 \big]
    \\
    &\leq 
    C 
    \bE \Big[   \Big| \big( X_{t_{n}}^{i,N}+  v(Y_n^{i,N},\mu^{Y,N}_{n})r +
    b(t_n,Y_n^{i,N},\mu^{Y,N}_{n}) r
    + \hs (t_n,Y_n^{i,N},\mu^{Y,N}_{n}) \Delta W_{n,r}^i\big)
    \\
    & - \big(\hx_{n}^{i,N}+v (Y_{n}^{i,\star,N},\hm^{Y,N}_n )r 
    +b(t_n,Y_{n}^{i,\star,N},\hm^{Y,N}_n) r
    + \hs (t_n,Y_{n}^{i,\star,N},\hm^{Y,N}_n) \Delta W_{n,r}^i
    \big) \Big|^2
    \Big]  +Ch, 
\end{align*}%
where $\Delta W_{n,r}^i=W_{t_n+r}^i-W_{t_n}^i$ and we remark that the integral terms in \eqref{eq: continuous conv t1}-\eqref{eq: continuous conv t5} can be analysed using the results in Section \ref{section: proof of B-consistent}. We now consider the following differences: From \eqref{eq:SSTM:scheme 1} and following similar calculations to \cite[Section 4.2]{chen2022SuperMeasure}, we have
\begin{align*}
    \bE \big[   \big| \big( &X_{t_{n}}^{i,N}+  v(Y_n^{i,N},\mu^{Y,N}_{n})r 
    \big)
    - 
    \big(\hx_{n}^{i,N}+v (Y_{n}^{i,\star,N},\hm^{Y,N}_n )r 
    \big) \big|^2
    \big]  
    \\
    &=  
     \bE \big[   \big\langle \big( X_{t_{n}}^{i,N}
     -\hx_{n}^{i,N} 
    \big)
    + r
    \Delta V_n^{Y} 
    ,
    \big( Y_n^{i,N}-
     Y_n^{i,\star,N} 
    \big)
    -(h-r)
    \Delta V_n^{Y} 
    \big\rangle
    \big]  
    \\
    &
    \leq
    \bE \big[ 
     | X_{t_{n}}^{i,N}-
     \hx_{n}^{i,N}|^2
   \big] \tfrac{2h-r}{2h}
    +
    \bE \big[ 
     | 
   Y_n^{i,N}-
     Y_n^{i,\star,N}
    |^2\big]\tfrac{r}{2h}
     +
    \bE \big[ 
     \big\langle  Y_n^{i,N}-
     Y_n^{i,\star,N},
     \Delta V_n^{Y}
    \big\rangle\big]r,
\end{align*}
where $\Delta V_n^{Y}=v(Y_n^{i,N},\mu^{Y,N}_{n})  -  v (Y_{n}^{i,\star,N},\hm^{Y,N}_n )$. 
By Jensen's inequality and the results in Section \ref{subsubsection: SSM is of C-stable }, we conclude that for all $ n\in \llbracket 0, M-1 \rrbracket$,  $i\in \llbracket 1,N \rrbracket$, $r\in[0,h]$, we have  
\begin{align*}
    \bE \big[& |X_{t_{n}+r}^{i,N}-\hx_{t_{n}+r}^{i,N}|^2 \big]
    \leq 
    Ch+ 
    C\bE \big[ 
     | X_{t_{n}}^{i,N}-
     \hx_{n}^{i,N}|^2
   \big] %
    \\
    & + 2r\bE \Big[ 
     \big\langle  Y_n^{i,N}-
     Y_n^{i,\star,N},
     u(Y_n^{i,N},\mu^{Y,N}_{n})  -  u (Y_{n}^{i,\star,N},\hm^{Y,N}_n )
    \big\rangle  \Big]
    \\ &
     +2r\bE \Big[ 
    \big|\sigma (t_n,Y_n^{i,N},\mu^{Y,N}_{n})
    -\sigma (t_n,Y_{n}^{i,\star,N},\hm^{Y,N}_n)  
      \big|^2
    \Big]
    \\&+
     \frac{2r}{N}\sum_{j=1}^N\bE \Big[
     \big\langle
    Y_n^{i,N}-
     Y_n^{i,\star,N},
     f( Y_n^{i,N} - Y_n^{j,N})-f( Y_n^{i,\star,N}-Y_n^{j,\star,N})
     \big\rangle \Big]
     \\ &
     +\frac{2r}{N}\sum_{j=1}^N\bE \Big[
     |  \fs( Y_n^{i,N} - Y_n^{j,N})-\fs( Y_n^{i,\star,N}-Y_n^{j,\star,N})|^2
     \Big]
    \\
    &
    \leq Ch 
    + 
    C\bE \big[ 
     | X_{t_{n}}^{i,N}-
     \hx_{n}^{i,N}|^2
   \big]
   +
   C\bE \big[ 
     | Y_n^{i,N}-
     Y_n^{i,\star,N}|^2
   \big]
   \leq Ch
   ,
\end{align*}
where we used \eqref{eq: conv result, discrete} and $\bE \big[      | Y_n^{i,N}-   Y_n^{i,\star,N}|^2   \big]
   \leq (1+Ch)\bE \big[ 
     | X_{t_{n}}^{i,N}-
     \hx_{n}^{i,N}|^2
   \big]$.
\color{black}
    
\end{proof}

\subsection{Proof of Theorem  \ref{theo:SSTM:stabilty}: Mean-square contractivity for the SSM} 

\begin{proof}[Proof of Theorem  \ref{theo:SSTM:stabilty}]
Using the notations of Theorem \ref{theo:SSTM:stabilty} and Section \ref{subsubsection: SSM is of C-stable }, and recalling the results in \eqref{eq:c-stable:y-y: u} and \eqref{eq: c-stable y-y f-f},  for all $i\in \llbracket 1,N \rrbracket$, $n \in \llbracket 0,M-1 \rrbracket$, we have 
\color{black}
\begin{align*} 
    \bE\big[\, &  |
Y_{n}^{i,X,N}-Y_{n}^{i,Z,N}|^2  \big] 
\\
&
\le \frac{2h}{N} \sum_{j=1}^N 
 \bE\big[\,   \langle   Y_{n}^{i,X,N}-Y_{n}^{i,Z,N} ,
    f( Y_{n}^{i,X,N}- Y_{n}^{j,X,N})- f( Y_{n}^{i,Z,N}- Y_{n}^{j,Z,N})  
  \rangle\big]
\\
& +
\bE\big[|\hx_{n}^{i,N}-\hz_{n}^{i,N}|^2\big]
+
2h\bE\big[\langle 
Y_{n}^{i,X,N}-Y_{n}^{i,Z,N}, 
u(Y_{n}^{i,X,N},\hm^{Y,X,N}_n)  -u(Y_{n}^{i,Z,N},\hm^{Y,Z,N}_n ) 
\rangle \big] 
\\
&\le 
|\hx_{n}^{i,N}-\hz_{n}^{i,N}|^2 + 
h (4L^{(1),+}_{(f)}+2L^{(1)}_{(u\sigma)} +2L^{(2)}_{(u\sigma)}  ) \bE\big[\,    |Y_{n}^{i,X,N}-Y_{n}^{i,Z,N}|^2 \big],
\end{align*}
and therefore
\begin{align}
\bE\big[\, &  | Y_{n}^{i,X,N}-Y_{n}^{i,Z,N}|^2\big]
\le
\bE\big[\,   |\hx_{n}^{i,N}-\hz_{n}^{i,N}|^2  \big]
\frac{ 1}{ 1-h( 4L^{(1),+}_{(f)}+2L^{(1)}_{(u\sigma)} +2L^{(2)}_{(u\sigma)}  )   }.
\label{eq: ergodicity proof eq3}
\end{align}
 \color{black}	
Next, we consider 
\begin{align*}
     \bE\big[\,  |&\hx_{n+1}^{i,N}-\hz_{n+1}^{i,N} |^2 \big]
     = 
     \bE\Big[\, \big|Y_{n}^{i,X,N}
     + b(t_n,Y_{n}^{i,X,N},\hm_n^{Y,X,N})h
     + \hs(t_n,Y_{n}^{i,X,N},\hm_n^{Y,X,N}) \Delta W_n^i
     \\
     &\qquad \quad\quad\qquad\qquad ~
     -Y_{n}^{i,Z,N} 
     - b(t_n,Y_{n}^{i,Z,N},\hm_n^{Y,Z,N})h
     - \hs(t_n,Y_{n}^{i,Z,N},\hm_n^{Y,Z,N}) \Delta W_n^i
     \big|^2 
     \Big]
     \\
     &=\bE\big[\,  |Y_{n}^{i,X,N}
                  -Y_{n}^{i,Z,N} \big|^2  
     + \big|\hs(t_n,Y_{n}^{i,X,N},\hm_n^{Y,X,N})-  \hs(t_n,Y_{n}^{i,Z,N},\hm_n^{Y,Z,N}) \big|^2 h    
     \big]
     \\
     &\qquad +h^2
     \bE\big[\, |b(t_n,Y_{n}^{i,X,N},\hm_n^{Y,X,N})-  b(t_n,Y_{n}^{i,Z,N},\hm_n^{Y,Z,N})  |^2 
     \big]
     \\
     &\qquad +2 h
     \bE\big[\, \langle  Y_{n}^{i,X,N} -Y_{n}^{i,Z,N}, b(t_n,Y_{n}^{i,X,N},\hm_n^{Y,X,N})-  b(t_n,Y_{n}^{i,Z,N},\hm_n^{Y,Z,N})\rangle    
     \big]
     \\
     &\leq 
     \bE\big[\,|\hx_{n}^{i,N}-\hz_{n}^{i,N}|^2
     \big]
     \\
     &~+
      \bE\big[\,| Y_{n}^{i,X,N} -Y_{n}^{i,Z,N}|^2
     \big] \Big(   h \big(4L^{(1),+}_{(f)}+2L^{(1)}_{(u\sigma)} +2L^{(2)}_{(u\sigma)}+2L_{(b)}^{(2)}+2L_{(b)}^{(3)} \big)+2L_{(b)}^{(1)}h^2      \Big),
\end{align*}
where in the last inequality we used the results above, \eqref{eq: c-stable y-y f-f} and Cauchy–Schwarz inequality. Substituting \eqref{eq: ergodicity proof eq3} into the last inequality  yields the result.
\end{proof}

\appendix

\section{Properties of the convolved drift term after integration}
\begin{lemma}
\label{AppendixLemma-aux1}
Let $(\mathbf{A}^f,\mathbf{A}^{f_{\sigma}})$ in Assumption \ref{Ass:Monotone Assumption} hold. Then it holds for any $\mu\in \cP_2(\bR^d)$ and $m >2$
\begin{align*}
\int_{\bR^d}
\big(&
\langle x, (f\ast \mu)(x) \rangle
+(m-1) |(\fs \ast \mu)(x)|^2
\big)
\mu(\dd x) 
\\
&=
\int_{\bR^d} \int_{\bR^d}
\big(
\langle  x , f(x-y)   \rangle 
+ (m-1)| \fs(x-y)  |^2
\big)
\mu(\dd x)\mu(\dd y)
\\
&
\le L^{(1)}_{(f)}  \big( 
\mu(|\cdot|^2)-| \mu( \textrm{id} )|^2 \big) = L^{(1)}_{(f)}\textrm{Var}_\mu ,
\end{align*}
where $\mu(|\cdot|^2):=\int_{\bR^d} |x|^2 \mu(\dd x)$, $\mu( \text{id} ) :=\int_{\bR^d} x \mu(\dd x)$ and $\textrm{Var}_\mu= \mu(|\cdot|^2)-| \mu( \textrm{id} )|^2$.
\end{lemma}
\begin{proof} 
Using $f(0)=\fs(0)=0$ and that $f$ is an odd function we have \color{black}  	
\begin{align*}
  \int_{\bR^d} & \int_{\bR^d} \big(
    \langle  x , f(x-y)   \rangle 
    + (m-1)| \fs(x-y)  |^2
    \big)  \mu(\dd x)\mu(\dd y) 
  \\
  &
  =
    \int_{\bR^d} \int_{\bR^d} \frac12\big(
    \langle  x-y , f(x-y)   \rangle 
    + 2(m-1)| \fs(x-y)  |^2
    \big) \mu(\dd x)\mu(\dd y) 
  \\
  &
  \leq 
\frac12  L^{(1)}_{(f)} \int_{\bR^d} \int_{\bR^d} |x-y|^2  \mu(\dd x)\mu(\dd y)  
  =
  \frac12  L^{(1)}_{(f)}
  \int_{\bR^d} \int_{\bR^d} \big(|x|^2-2 \left \langle x,y \right \rangle +|y|^2\big)  \mu(\dd x)\mu(\dd y)
 \\ 
 &= \frac12  L^{(1)}_{(f)} \Big(
  2\mu(|\cdot|^2)
  -2 \int_{\bR^d} x \mu(\dd x) \int_{\bR^d} y \mu(\dd y)
  \Big) 
  \\
  &=
  L^{(1)}_{(f)} \big( \mu(|\cdot|^2) - \big|\int_{\bR^d} x \mu(\dd x)\big|^2 \big) 
  = L^{(1)}_{(f)} \textrm{Var}_\mu,
\end{align*}
\color{black}
where for the inequality we used the monotonicity condition on the convolution kernels and the symmetry of the double integration in $\mu$.   
\end{proof}

\begin{lemma}
\label{AppendixLemma-aux2}
Let $f$ and $f_{\sigma}$ satisfy conditions $(\mathbf{A}^f,~\mathbf{A}^{\fs})$ of Assumption \ref{Ass:Monotone Assumption}. Set $ L^{(1),+}_{(f)}=\max\{0,L^{(1)}_{(f)}\}$. Then, for any $\mu,\nu\in \cP_{2q+2}(\bR^d)$ with $q$ defined in Assumption \ref{Ass:Monotone Assumption}, we have
\begin{align*}
    &\int_{\bR^d} \int_{\bR^d}  \Big( \langle  x-y ,  (f\ast \mu )(x) - (f\ast \nu\ ) (y)  \rangle  
    +(m-1) \big| (\fs\ast \mu )(x)- (\fs\ast \nu ) (y)\big|^2 
    \Big)
    \mu(\dd x)\nu(\dd y) 
    \\
    & 
    \leq 
    2 L^{(1),+}_{(f)} 
 \int_{\bR^d} \int_{\bR^d}  |x-y|^2  \mu(\dd x)\nu(\dd y).
\end{align*}
\end{lemma}
Although not explicitly mentioned, this result \textit{requires} the random variables $X\sim \mu$ and $Y\sim \nu$ to be independent -- see Remark \ref{remark:IndependenceIsNeeded}.

\begin{proof}
For any $\mu,\nu\in \cP_{2q+2}(\bR^d)$, we compute 
\allowdisplaybreaks
\begin{align*}
    &
    \int_{\bR^d} \int_{\bR^d}   \langle  x-y ,  (f\ast \mu )(x) - (f\ast \nu ) (y)  \rangle   \mu(\dd x)\nu(\dd y) 
    \\
    &
         =
         \int_{\bR^d} \int_{\bR^d}\int_{\bR^d} \int_{\bR^d} \langle  x-y , f(x-x') -f(y-y')   \rangle  \mu(\dd x')\nu(\dd y') \mu(\dd x)\nu(\dd y) 
         \\
         &
         = \frac{1}{2} \Bigg[  
         \int_{\bR^d} \int_{\bR^d}\int_{\bR^d} \int_{\bR^d} \langle  x-y , f(x-x') -f(y-y')   \rangle  \mu(\dd x')\nu(\dd y') \mu(\dd x)\nu(\dd y)
         \\
         &\quad\quad
         -
         \int_{\bR^d} \int_{\bR^d}\int_{\bR^d} \int_{\bR^d} \langle  x'-y' , f(x-x') -f(y-y')   \rangle  \mu(\dd x)\nu(\dd y)\mu(\dd x')\nu(\dd y') 
         \Bigg]
         \\
         &
         = \frac{1}{2}
         \int_{\bR^d} \int_{\bR^d}\int_{\bR^d} \int_{\bR^d} \langle  (x-x')-(y-y') , f(x-x') -f(y-y')   \rangle  \mu(\dd x)\nu(\dd y)\mu(\dd x')\nu(\dd y'),
    \end{align*} 
and thus,
\allowdisplaybreaks
    \begin{align*}
    &
    \int_{\bR^d} \int_{\bR^d}   \Big( \langle  x-y ,  (f\ast \mu )(x) - (f\ast \nu ) (y)  \rangle  
    + (m-1)
    \big| (\fs\ast \mu )(x)- (\fs\ast \nu ) (y)\big|^2 \Big)
    \mu(\dd x)\nu(\dd y) 
    \\
    &
         =\frac{1}{2}
         \int_{\bR^d} \int_{\bR^d}\int_{\bR^d} \int_{\bR^d}
         \Big( 
         \langle  (x-x')-(y-y') , f(x-x') -f(y-y')  \rangle
         \\
         & \qquad\qquad\qquad\qquad\qquad
         + 2(m-1) |\fs(x-x') -\fs(y-y')  |^2
         \Big)
         \mu(\dd x')\nu(\dd y') \mu(\dd x)\nu(\dd y) 
         \\
         &
         \le \frac12 L^{(1)}_{(f)}
           \int_{\bR^d} \int_{\bR^d}\int_{\bR^d} \int_{\bR^d}
           \big|  (x-x')-(y-y')   \big|^2
           \mu(\dd x')\nu(\dd y') \mu(\dd x)\nu(\dd y)
         \\
         &
         \le 2 L^{(1),+}_{(f)} 
         \int_{\bR^d} \int_{\bR^d}  |x-y|^2  \mu(\dd x)\nu(\dd y).
    \end{align*} 
\end{proof}

\color{black}
\begin{remark}[Independence is needed]
\label{remark:IndependenceIsNeeded}
In the context of Lemma \ref{AppendixLemma-aux2}, take  $X\sim \mu,~Y\sim \nu$ with $\mu,\nu\in \cP_{2q+2}(\bR^d)$. 
Let $\pi\in \cP_{2q+2}(\bR^d\times \bR^d)$ be the joint distribution of $X,Y$. As in the proof of Lemma \ref{AppendixLemma-aux2}, we compute
\begin{align*}
\bE[  &\langle  X -Y,  (f\ast \mu )(X) -  (f\ast \nu )(Y)  \rangle   ]
\\
    &=
    \int_{\bR^d} \int_{\bR^d}   \langle  x-y ,  (f\ast \mu )(x) - (f\ast \nu ) (y)  \rangle   \pi(\dd x,\dd y) 
    \\
    &
         =
         \int_{\bR^d} \int_{\bR^d}\int_{\bR^d} \int_{\bR^d} \langle  x-y , f(x-x') -f(y-y')   \rangle  \mu(\dd x')\nu(\dd y') \pi(\dd x,\dd y)  
    \\
    &= 
    \frac12 \int_{\bR^d} \int_{\bR^d}\int_{\bR^d} \int_{\bR^d} \langle  x-y , f(x-x') -f(y-y')   \rangle  \mu(\dd x')\nu(\dd y') \pi(\dd x,\dd y)
    \\
    &\quad -
    \frac12 \int_{\bR^d} \int_{\bR^d}\int_{\bR^d} \int_{\bR^d} \langle  x'-y' , f(x-x') -f(y-y')   \rangle  \mu(\dd x)\nu(\dd y) \pi(\dd x',\dd y').
    \end{align*}
It is now obvious that for general choices of $f$, we cannot rearrange the measure components to obtain the result of Lemma \ref{AppendixLemma-aux2} since $\pi(\dd x,\dd y) \neq \mu (\dd x) \nu(\dd y)$. 
\end{remark}
 \color{black}


\bibliographystyle{amsplain}

\providecommand{\bysame}{\leavevmode\hbox to3em{\hrulefill}\thinspace}
\providecommand{\MR}{\relax\ifhmode\unskip\space\fi MR }
\providecommand{\MRhref}[2]{%
  \href{http://www.ams.org/mathscinet-getitem?mr=#1}{#2}
}
\providecommand{\href}[2]{#2}


\end{document}